\documentclass[11pt]{article} 
\usepackage{babel}
\usepackage{amsmath}
\usepackage{lmodern}
\usepackage[T1]{fontenc}
\usepackage[babel=true]{microtype}
\usepackage{amsxtra}%
\usepackage{amsfonts}%
\usepackage{amssymb}%
\usepackage{amsthm}
\usepackage{graphicx}
\usepackage{epstopdf} 
\usepackage{color,hyperref}
\hypersetup{colorlinks,breaklinks,
	linkcolor=blue,urlcolor=blue,
	anchorcolor=blue,citecolor=blue}

\usepackage{xcolor}

\usepackage{subcaption}
\usepackage{appendix}
\usepackage{enumerate}
\usepackage[font=footnotesize]{caption}
\usepackage{calc}
\usepackage{cite}
\usepackage{accents}

\usepackage[margin=0.75in]{geometry}

\newcommand{\eps}{\varepsilon}
\newcommand{\Z}{\mathbb{Z}}
\newcommand{\N}{\mathbb{N}}
\newcommand{\R}{\mathbb{R}}

\newcommand{\C}{\mathbb{C}}

\newcommand{\F}{\mathbb{F}}

\newcommand{\mcl}{\mathcal{L}}

\renewcommand{\Re}{\mathrm{Re} \,}
\renewcommand{\Im}{\mathrm{Im} \,}

\def\XXint#1#2#3{{\setbox0=\hbox{$#1{#2#3}{\int}$ }
		\vcenter{\hbox{$#2#3$ }}\kern-.6\wd0}}

\newtheorem*{thm*}{Theorem}

\newtheorem{prop}{Proposition}
\newtheorem{lemma}[prop]{Lemma}
\newtheorem{corollary}[prop]{Corollary}

\newtheorem*{thminf}{Main result (informal summary)}
\newtheorem{thmlocal}[prop]{Theorem}

\newtheorem{hyp}{Hypothesis}

\newtheorem{remark}[prop]{Remark}

\numberwithin{equation}{section}
\numberwithin{prop}{section}

\newcommand{\clin}{{c_\mathrm{lin}}}

\newcommand{\etalin}{\eta_\mathrm{lin}}
\newcommand{\uwt}{\mathbf{u}_\mathrm{wt}}

\newcommand{\lwt}{\mathcal{L}_\mathrm{wt}}

\renewcommand{\u}{\mathbf{u}}
\newcommand{\g}{\mathbf{g}}

\newcommand{\nuwt}{\nu_\mathrm{wt}}
\newcommand{\lambdawt}{\lambda_\mathrm{wt}}

\newcommand{\e}{\mathbf{e}}
\newcommand{\w}{\mathbf{w}}
\newcommand{\q}{\mathbf{q}}
\renewcommand{\v}{\mathbf{v}}

\newcommand{\re}{\mathrm{e}}
\newcommand{\ri}{\mathrm{i}}
\newcommand{\de}{\mathrm{d}}
\newcommand{\NT}{\widetilde{\mathcal{N}}}

\newcommand{\vt}{\widetilde{\mathbf{v}}}
\newcommand{\vf}{\mathring{\mathbf{v}}}
\newcommand{\uf}{\mathring{\mathbf{u}}}

\newcommand{\z}{\mathbf{z}}

\newcommand{\kwt}{k_\mathrm{wt}}

\newcommand{\cps}{c_\mathrm{ps}} 
\newcommand{\etaps}{\eta_\mathrm{ps}}
\newcommand{\Ups}{\mathbf{u}_\mathrm{ps}}
\newcommand{\Aps}{\mathcal{A}_\mathrm{ps}}
\newcommand{\Lps}{\mathcal{L}_\mathrm{ps}}
\newcommand{\Lwthat}{\check{\mathcal{L}}_\mathrm{wt}}
\newcommand{\psiad}{\mathbf{\psi}_\mathrm{ad}}
\newcommand{\tg}{\tilde{\mathbf{g}}}
\newcommand{\f}{\mathbf{f}}
\newcommand{\Ptr}{P_\mathrm{tr}}

\newcommand{\Deff}{D_\mathrm{eff}}

\newcommand\extrafootertext[1]{%
	\bgroup
	\renewcommand\thefootnote{\fnsymbol{footnote}}%
	\renewcommand\thempfootnote{\fnsymbol{mpfootnote}}%
	\footnotetext[0]{#1}%
	\egroup
}

\newcommand{\erf}{\mathrm{erf}}

\begin{document}
	\begin{center}
		{\fontsize{15}{15}\fontseries{b}\selectfont{Selection of pushed pattern-forming fronts in the FitzHugh-Nagumo system}}\\[0.2in] 
		Montie Avery$^1$, Paul Carter$^2$, and Bj\"orn de Rijk$^3$ \\[0.1in]

	\textit{\footnotesize 
		$^1$Department of Mathematics, Emory University, 400 Dowman Drive, Atlanta, GA 30322, USA \\
		$^2$Department of Mathematics, University of California, Irvine, 340 Rowland Hall, Irvine, CA 92697, USA \\
		$^3$Department of Mathematics, Karlsruhe Institute of Technology, Englerstra{\ss}e 2, 76131 Karlsruhe, Germany \\
	}
	\end{center}
	
	\begin{abstract}
		We establish nonlinear stability of fronts that describe the creation of a periodic pattern through the invasion of an unstable state. Our results concern pushed fronts, that is, fronts whose propagation is driven by a localized mode at the front interface. We prove that these pushed pattern-forming fronts attract initial data supported on a half-line, and therefore determine both propagation speeds and selected wave numbers for invasion from localized initial conditions. This provides to our knowledge the first proof of the \emph{marginal stability conjecture} for pattern-forming fronts, thereby confirming universal wave number selection laws widely used in the physics literature. We present our analysis in the specific setting of the FitzHugh--Nagumo system, but our methods can be applied to general dissipative PDE models which exhibit pattern formation. The main technical challenge is to control the interaction between the localized mode driving the propagation and outgoing diffusive modes in the wake of the front. Through a subtle far-field/core decomposition of the linearized evolution, we resolve this interaction and describe the nonlinear response of the front to perturbations as a dynamically driven phase mixing problem for the pattern in the wake. The methods we develop are generally useful in any setting involving the interaction of localized modes and outward diffusive transport, such as in the nonlinear stability of undercompressive viscous shock waves or source defects. \medskip


	\end{abstract}

	\extrafootertext{The authors gratefully acknowledge support from the US National Science Foundation, through DMS-2510541 \& DMS-2202714 (M.A.) and NSF-DMS-2238127 (P.C.). B.dR. is funded by the Deutsche Forschungsgemeinschaft (DFG, German Research Foundation) - Project-ID 491897824 and Project-ID 258734477 - SFB 1173.}  

	\section{Introduction}
	
	Invasion into unstable states plays a key role in describing the formation of complex coherent structures in many physical systems. Unstable states may be observed as transients, for instance after a change in system parameters or the introduction of an external destabilizing agent, such as an invasive species or novel disease. In large spatial domains, localized perturbations to the unstable state typically grow, saturate at finite amplitude, and spread into the unstable state, creating a new stable state in the wake of a propagating \emph{invasion front}. A fundamental question is then to predict both the speed of this invasion front as well as which new state it selects in its wake. 
	
	The study of invasion fronts in the mathematical literature has historically been limited to systems with comparison principles; see e.g.~\cite{Kolmogorov, Bramson1, Bramson2, Lau, AronsonWeinberger, Roquejoffre, Weinberger, HamelNRR, NRR1, NRR2, HamelNRR2, berestyckinirenberg}. Comparison principles are inherently incompatible with complex pattern formation, and so fronts in these systems typically select a new spatially constant state in their wake. On the other hand, complex pattern formation is frequently observed in experiments and simulations of invasion processes in a large variety of physical systems; see for instance the comprehensive review~\cite{vanSaarloos}. Moreover, there is interest in harnessing this self-organization of coherent patterns through invasion for manufacturing technologies in materials science~\cite{NatureMaterials, Bradley}.  Pattern-forming systems generically admit families of periodic patterns parameterized by the wave number. When the pattern is grown via invasion, from a compactly supported perturbation to the unstable state, a distinguished wave number is typically selected out of this family, independent of precise details of the initial condition~\cite{deelanger}. Such wave number selection laws have been widely observed across the sciences~\cite{vanSaarloos}, but beyond heuristics and matched asymptotics~\cite{deelanger}, there appear to be no mathematical results describing this selection mechanism. 
	
	In a first approach, to rigorously confirm these selection laws, one might hope to identify (within a given PDE model) a unique propagation speed and wave number for which there exists a stable traveling front solution connecting the unstable state to the periodic pattern in the wake. Such a front would then be expected to attract all sufficiently steep nearby initial data. Pattern-forming fronts have been constructed close to the onset of a Turing instability~\cite{ColletEckmann, EckmannWayne1, HaragusSchneider, Hexagons, Hilder2025}, and at large amplitude in phase separation problems~\cite{ScheelCoarsening1, ScheelCoarsening2}. However, generically such fronts exist and are stable for an open range of speeds and wave numbers~\cite{Eckmann2000,Eckmann2002}. Moreover, stability typically holds against perturbations which do not alter the tail decay rate, prohibiting steep (e.g.~supported on a half-line) initial conditions. So, merely searching for stable front solutions does not answer the question of which front is selected in invasion from steep initial data. 
	   
	The \emph{marginal stability conjecture}~\cite{AveryScheelSelection, colleteckmannbook, bers1983handbook, brevdo, deelanger, vanSaarloosMarginalstability, vanSaarloos} asserts that speeds and wave numbers selected by propagation of steep initial data are determined by the distinguished front solutions which are \emph{marginally spectrally stable} in an appropriate sense. There are two separate scenarios for marginal spectral stability: marginal stability may arise from marginal pointwise stability of the unstable state in the leading edge of the front in an appropriate moving frame, or from marginally stable point spectrum of the entire front solution. In the former case, the propagation is driven by the decaying tail in the leading edge of the front, and so these fronts are said to be \emph{pulled}. In the latter case, the propagation is driven by a mode localized near the front interface, and the fronts are called \emph{pushed}. We refer to~\cite{AHSReview, vanSaarloos, pp} for further details on pushed and pulled front propagation and the marginal stability conjecture. 
	
	For pushed fronts which select constant states in their wake, the marginal stability conjecture can be established through classical semigroup methods~\cite{Sattinger, Henry}, exploiting a spectral gap in the linearization about the front in an appropriate norm. For pulled fronts selecting constants states, the marginal stability conjecture was only established recently~\cite{AveryScheelSelection, AverySelectionRD}. Sharp nonlinear stability results for pulled pattern-forming fronts, which in light of~\cite{AveryScheelSelection, AverySelectionRD} appear to be a key ingredient in a proof of the marginal stability conjecture in this setting, were subsequently proven in~\cite{FHNpulled}. However, prior to the present work, there are no results establishing selection of any pattern-forming front from steep initial conditions. Our main result establishes that pushed pattern-forming fronts are indeed selected by steep initial data and provides a detailed description of the resulting convergence to the front.

	Our methods are broadly applicable to pushed pattern-forming fronts in general dissipative evolution problems, but for concreteness we present our analysis for the FitzHugh--Nagumo system,
	\begin{align*}
		u_t &= u_{xx} + u (u+a)(1-u-a) - w, \\
		w_t &= \eps (u - \gamma w), 
	\end{align*}
	which we write as a degenerate reaction-diffusion system for $\u = (u,w)^\top$, 
	\begin{align}
		\u_t = D \u_{xx} + F(\u), \qquad D = \begin{pmatrix}
			1 & 0 \\
			0 & 0 
		\end{pmatrix},
		\qquad F(\u) = \begin{pmatrix}
			u(u+a)(1-u-a) - w \\
			\eps (u - \gamma w)
		\end{pmatrix}. \label{e: FHN}
	\end{align}
	The FitzHugh--Nagumo system originally arose as a simplification of the Hodgkin--Huxley model for signal propagation in nerve fibers. However, it has since been recognized across the sciences as a paradigmatic model for excitable and oscillatory media far from equilibrium. Together with slight variations, it has been used to model, for instance, the onset of turbulence in fluids~\cite{Barkley}, oxidation on platinum surfaces~\cite{Oxidation1, Oxidation2}, and cardiac arrythmias~\cite{Cardiac}. Mathematically,~\eqref{e: FHN} is one of the simplest models which could, and does, exhibit complex spatio-temporal pattern formation. 

    \medskip
    
	\noindent \textbf{Existence of pushed fronts}. We consider~\eqref{e: FHN} in the oscillatory regime, $0 < a < \frac{1}{2}$ and $0 < \gamma < 4$, with $0 < \eps \ll 1$. In this regime, the system exhibits spatially homogeneous oscillations and the unique spatially constant equilibrium     $\u = (0,0)^\top$ is unstable. When the initial condition is sufficiently localized, these temporal oscillations are modulated by spatial spreading, leading to the formation of spatially periodic wave trains in the wake of a propagating invasion front; see Figure~\ref{fig: spacetime and sketch}. The modulation of temporal oscillations by spatial spreading has been identified as a ubiquitous mechanism for pattern formation in invasion processes, and appropriately combining the spatial spreading speed and temporal frequency of oscillations leads to a universal prediction scheme for the selected wave number~\cite{deelanger, vanSaarloos, AHSReview}. In the setting of the FitzHugh-Nagumo system~\eqref{e: FHN}, pushed pattern-forming fronts were constructed in~\cite{CarterScheel} using geometric singular perturbation theory. 

    \begin{figure}
		\centering
		\includegraphics[width=0.72\textwidth]{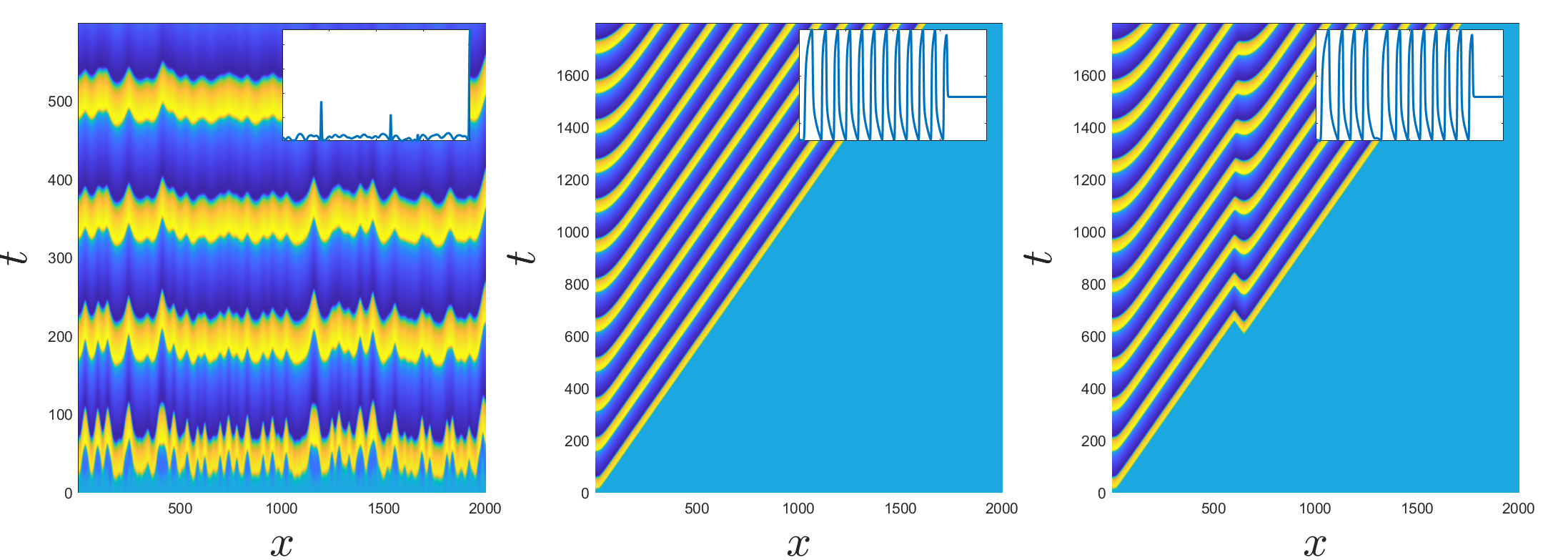}
		\includegraphics[width=0.27\textwidth]{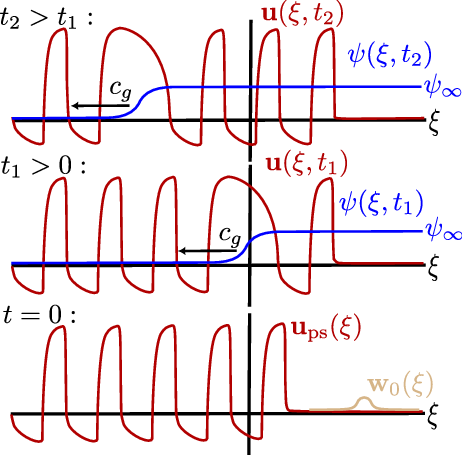}
		\caption{Left three panels: spacetime diagrams of numerical solutions to~\eqref{e: FHN}; insets show the graph of $u(x,t)$ against $x$ at 100 time units before the final time. Far left: simulation with initial condition consisting of the unstable equilibrium $\u \equiv (0,0)^\top$ perturbed with low amplitude white noise. Second from left: simulation with initial condition consisting of a small compactly supported perturbation of $\u \equiv (0,0)^\top$ near the left boundary. Third from left: initially the same simulation as the previous panel, but at $t = 600$ we pause the simulation and add a perturbation ahead of the front; the front adjusts its position and continues to propagate, leaving behind a phase defect. Right: sketch of the dynamics of the phase defect, in the frame $\xi = x - \cps t$ in which $\Ups(\xi)$ is stationary, with time increasing from bottom to top. Bottom right: front profile $\Ups(\xi)$ (red) and initial perturbation $\w_0(\xi)$ (tan). The plots at subsequent times show the response to the perturbation. The front adjusts its position, forming a phase defect $\psi(x,t)$ (blue) which propagates to the left (in this frame) with the group velocity.}
		\label{fig: spacetime and sketch}
	\end{figure}
    
	\begin{thmlocal}[{Existence of pushed pattern-forming fronts~\cite[Theorem 1.3]{CarterScheel}}]\label{t: existence}
		Fix $0 < a < \frac{1}{3}$ and $0 < \gamma < 4$. Then, there exist constants $C_0,\eps_0 > 0$ such that for all $\eps \in (0,\eps_0)$, there exist a speed $\cps > 0$, a spatial decay rate $\etaps > 0$, and solutions $\u(x,t) = \Ups (x - \cps t)$ and $\u(x,t) = \uwt(x - \cps t)$ to~\eqref{e: FHN} with smooth profiles $\Ups,\uwt \colon \R \to \R^2$ such that $\uwt$ is $L$-periodic, and there exist constants $C,\eta_* > 0$ and a vector $\Ups^0 \in \R^2 \setminus \{0\}$ such that	\begin{align}
			\left|\Ups(\xi) - \Ups^0 \re^{-\etaps \xi}\right| &\leq C\re^{-(\etaps + \eta_*) \xi}, & & \xi \geq 0, \label{e: front right asymptotics} \\
		\left|\Ups(\xi) - \uwt(\xi)\right| &\leq C\re^{\eta_*\xi}, & & \xi \leq 0. \label{e: front left asymptotics}
		\end{align}
        Moreover, we have the approximations
        \begin{align} \label{e:approx_eps}
        \left|\cps - \frac{1+a}{\sqrt{2}}\right|, \left|\etaps - \frac{1-a}{\sqrt{2}}\right| \leq C_0\eps, \qquad \left|\eps L - L_- - L_+\right| \leq C_0 \eps^{\frac13},
        \end{align}
        where
        \begin{align*} L_\pm = \int_{u_{1,\pm}}^{u_{2,\pm}} \frac{(1+a) f'(u)}{\sqrt{2}\, (\gamma f(u)-u)} \de u, \qquad u_{j,\pm} = \frac13 \left(1-2a \pm j \sqrt{1+a+a^2}\right), \qquad j = 1,2.\end{align*}
	\end{thmlocal}
	For the remainder of this paper, we fix the parameters $a, \gamma$, and $ \eps$ in~\eqref{e: FHN} such that Theorem~\ref{t: existence} holds. We emphasize that the periodic wave train $\uwt$ generated in the wake of $\Ups$ is a ``far-from-equilibrium pattern'' in the sense that it has large amplitude and is highly nonlinear, in contrast to weakly nonlinear, low-amplitude patterns selected by fronts near a Turing instability~\cite{ColletEckmann, EckmannWayne1}. In addition, the estimate~\eqref{e:approx_eps} provides explicit leading-order approximations of the propagation speed $\cps$ and exponential decay rate $\etaps$ of the front, as well as for the selected wave number $\kwt = \frac{2\pi}{L}$ of the wave train in its wake.  
	
	We illustrate the role of invasion fronts in selecting coherent patterns in~\eqref{e: FHN} in Figure~\ref{fig: spacetime and sketch}: the leftmost panel shows a solution arising from a perturbation of the unstable state $\u \equiv (0,0)^\top$ by low-amplitude white noise. This solution exhibits temporal oscillations, but there is initially no spatial coherence, although the oscillations begin to synchronize on large time scales. In stark contrast, the second panel shows the highly coherent spatially periodic pattern generated through invasion of the unstable state by a small perturbation that is supported near the left boundary. 
	    
    \noindent\textbf{Nonlinear stability and selection of pushed fronts.} We pass to the frame $\xi = x - \cps t$, so that $\Ups$ is a stationary solution to the transformed PDE
	\begin{align}
		\u_t = D\u_{\xi \xi} + \cps \u_\xi + F(\u). \label{e: FHN comoving}
	\end{align}
	Our main result establishes the nonlinear stability of $\Ups$ as a solution to~\eqref{e: FHN comoving} and its selection by initial data which are supported on a half line. We provide an informal summary of the result below, and discuss the relevant notions, such as group velocities and linear spreading speeds, in the following paragraph. Precise statements of the spectral hypotheses and the main theorem are given in Section~\ref{s: setup}.
	\begin{thminf}
		Let $\Ups$ be a stationary pattern-forming front solution to~\eqref{e: FHN comoving}, as established in Theorem~\ref{t: existence}. Assume that $\Ups$ satisfies the following conditions:
		\begin{itemize}
			\item The wave train $\uwt$ in the wake of the front is \emph{diffusively spectrally stable}, that is, the spectrum of the linearization of~\eqref{e: FHN comoving} about $\uwt$ touches the imaginary axis in a single quadratic tangency at the origin and is otherwise stable;
			\item The \emph{group velocity} of the wave train $\uwt$ points to the left, away from the front interface, when measured in the frame co-moving with the front;
			\item The propagation speed $\cps$ is greater than the \emph{linear spreading speed} associated to the unstable rest state $\u = (0,0)^\top$. As a result, $\u = (0,0)^\top$ is pointwise exponentially stable as a solution to~\eqref{e: FHN comoving};
			\item The linearization of~\eqref{e: FHN comoving} about $\Ups$ has a simple eigenvalue at the origin (embedded in the essential spectrum) with associated eigenfunction $\Ups'$, and the point spectrum is otherwise stable. 
		\end{itemize}
		Then, $\Ups$ is nonlinearly stable as a solution to~\eqref{e: FHN comoving} against sufficiently localized perturbations. The perturbed solution converges to a fixed spatial translate of $\Ups$, locally uniformly in space. Finally, the basin of attraction of $\Ups$ includes some initial data which are supported on a half line $(- \infty, \xi_0]$. Hence, $\Ups$ is a \emph{selected front} in the sense of~\cite[Definition 1]{AveryScheelSelection}. 
	\end{thminf}
	Our diffusive spectral stability assumption for $\uwt$ is generic for stable wave trains in reaction-diffusion systems and a standard assumption in nonlinear stability analyses, cf.~\cite{SchneiderAbstract, SSSU, JONZ, JNRZ_13_1}. The group velocity of the wave train measures the speed of transport of small perturbations, at the linear level. Assuming that the group velocity points away from the front interface naturally characterizes the front as a \emph{source} of patterns in the sense of~\cite{SandstedeScheelDefects}. We are not aware of any examples of pattern-forming invasion fronts for which the group velocity does not point away from the front interface. The linear spreading speed characterizes the speed of propagation of small disturbances in~\eqref{e: FHN}, linearized about $\u = (0,0)^\top$. Pushed fronts occur when the nonlinearity enhances propagation, so that selected fronts travel faster than the linear spreading speed. To summarize, we expect the above assumptions to be generic features for pushed pattern-forming fronts in general dissipative evolution problems. The assumptions are verified for the pushed fronts of Theorem~\ref{t: existence} in the companion work~\cite{spectral, SpectralFronts}. 
	
	\subsection{Overview, challenges, and related work}\label{s: overview}

    \noindent \textbf{Nonlinear stability of traveling waves: general strategy.} Consider a traveling-wave solution $\u(x,t) = \u_*(x-ct)$ to a general reaction-diffusion system
    \begin{align}
        \u_t = D\u_{xx} + F(\u), \quad \u (x,t) \in \mathbb{R}^d,\quad x \in \mathbb{R}, \quad t> 0. \label{e: general RD}
    \end{align}
    Passing to the comoving coordinate $\xi = x-ct$, the profile $\u_*$ becomes a stationary solution to
    \begin{align}
      \u_t = D \u_{\xi \xi} + c \u_\xi + F(\u). \label{e:RDsys}
    \end{align}
    To study stability of the traveling wave, one considers perturbed solutions to~\eqref{e:RDsys} of the form $\u(\xi, t) = \u_*(\xi) + \w(\xi, t)$. This leads to an evolution equation for the perturbation $\w$ of the form
    \begin{align*}
        \w_t = \mathcal{L} \w + N(\w), 
    \end{align*}
    where $\mathcal{L} = D \partial_{\xi \xi} + c \partial_\xi + F'(\u_*(\xi))$ is the linearization of~\eqref{e:RDsys} about $\u_*(\xi)$, and $N(\w) = F(\u_* + \w) - F(\w) - F'(\u_*) \w = \mathrm{O}(\w^2)$ is the nonlinear remainder. A first hope would be to prove that, in an appropriate norm, we have $\w(\cdot,t) \to 0$ as $t \to \infty$, so that the perturbed solution $\u(\cdot, t)$ converges back to the traveling wave $\u_*$ as $t \to \infty$. 

    In the presence of spectral information (and absence of additional variational or order-preserving structure), a standard approach to establish temporal decay of $\w(t)$ is to analyze its variation-of-constants (or Duhamel) formula
    \begin{align*}
        \w(\cdot, t) = \re^{\mcl t} \w_0 + \int_0^t \re^{\mcl (t-s)} N(\w(\cdot, s)) \, \de s,
    \end{align*}
    where $\re^{\mcl t}$ is the semigroup generated by $\mcl$. The strategy is then to convert spectral information on $\mcl$ into estimates on the linearized evolution $\re^{\mcl t}$ via its inverse Laplace representation. Provided these estimates are sufficiently strong, they can be passed to the nonlinear level through a perturbative argument, ultimately yielding temporal decay of $\w(\cdot, t)$ at rates inherited from the linearized evolution. In the present setting, however, there are several obstacles obstructing this approach, which we now explore.

    \textbf{Nonlinear stability of traveling waves: scalar pushed fronts.} As a basic illustrative example, consider the scalar Nagumo equation 
    \begin{align*}
    u_t = u_{xx} + u(u+a)(1-u-a),
    \end{align*}
    which for $0 < a < \frac{1}{3}$ admits a pushed front solution $u(x,t) = u_\mathrm{ps} (x- \cps^0 t)$ with speed $\cps^0 > 0$ and smooth profile $u_{\mathrm{ps}} \colon \R \to \R$, connecting the spatially constant stable state $u \equiv 1-a$ at $-\infty$ to the unstable state $u \equiv 0$ at $\infty$.   One can attempt to adopt the strategy outlined above to control perturbations $w(\xi, t)$ to this pushed front. 
    
    By restricting the allowed localization of perturbations, the essential spectrum of $\mcl$ may fully be stabilized. However, $\re^{\mcl t}$ does not decay in time due to the presence of a neutral eigenvalue at the origin with eigenfunction $u_\mathrm{ps}'$. This eigenvalue arises from translational invariance of the original equation, together with the steep tail decay of $u_{\mathrm{ps}}$. Since there is, however, a gap between this isolated eigenvalue and the essential spectrum, one can use a spectral projection to separate this neutral behavior from the rest of the linearized dynamics, resulting in the decomposition 
	\begin{align}
	\re^{\mcl t} = u_\mathrm{ps}' \Ptr + \mathrm{O}(\re^{-\mu t}), \qquad t \geq 0, \label{e: nagumo linear estimate}
	\end{align}
    where $\mu > 0$ is a fixed constant and the functional $\Ptr \colon L^2(\R) \to \R$ extracts the coefficient from the spectral projection. The leading-order linearized dynamics of the perturbed solution then have the form
    \begin{align*}
        u_\mathrm{ps} + u_\mathrm{ps}'P_\mathrm{tr}w \approx u_\mathrm{ps} (\cdot + P_\mathrm{tr} w),
    \end{align*}
    suggesting that the critical behavior of $\re^{\mcl t}$ leads to convergence to a phase shifted front. This observation can be leveraged at the nonlinear level by allowing for a temporal phase function to capture the excitation caused by the neutral translational eigenmode of $\mcl$. Ultimately, one then obtains convergence, which is exponential in time and uniform in space, to a fixed spatial translate $u_\mathrm{ps}(\xi + \xi_\infty)$ of the front, where the shift $\xi_\infty$ depends on the initial perturbation $w_0$. This program is carried out, for instance, in~\cite{Sattinger}; see also~\cite[Chapter~4]{KapitulaPromislow} for further details and references. 
	
	\medskip
    
    \noindent\textbf{Dynamics of perturbations to pushed pattern-forming fronts.} 
	  Our main result concerns pushed fronts $\Ups$ that select a periodic wave train $\uwt$ in their wake. For such fronts there is still a neutral translational eigenvalue with eigenfunction $\Ups'$, but this eigenvalue is now embedded in the marginally stable essential spectrum originating from the wave train; see Figure~\ref{fig: front and spectrum}. Guided by the above analysis for the simpler, non-pattern-forming pushed fronts, we may still expect that perturbations initially excite this neutral eigenmode, causing the front interface to adjust its position. This then alters the phase of the periodic wave train near the front interface, producing a ``phase defect''. That is, supposing the front interface is located near $\xi = 0$, the solution to the left of the front interface now resembles $\uwt(\xi + \psi_0(\xi))$, where
	\begin{align*}
		\lim_{\xi \to -\infty} \psi_0 (\xi) = 0 \neq \psi_0(0). 
	\end{align*}
	Thus, the adjustment of the front interface in response to the perturbation creates a \emph{phase mixing} problem for the pattern in the wake. Phase mixing problems for wave trains have been extensively studied~\cite{DSSS, SSSU, JNRZ_13_1, JNRZWhitham, JNRZInventiones, IyerSandstede}, but here the analysis is complicated by two facts: (i) the phase offset is driven by internal interactions in the system, rather than externally prescribed; (ii) the underlying coherent structure $\Ups$ is not spatially periodic, so tools such as the Floquet--Bloch transform are not available to study the linearized dynamics. 
	
	Overcoming these issues, we ultimately obtain an analogous picture to~\cite{DSSS,SSSU, JNRZ_13_1, JNRZWhitham, JNRZInventiones, IyerSandstede} for the phase dynamics in the wake:  the solution to the left of the front interface resembles $\uwt(\xi + \psi(\xi, t))$ for large times, where the \emph{phase modulation} $\psi(\xi, t)$ (approximately) solves the eikonal equation
	\begin{align}
		\psi_t = \Deff \psi_{\xi \xi} - c_g \psi_\xi + \beta \psi_\xi^2, \label{e: eikonal}
	\end{align}
	where $c_g < 0$ is the group velocity of the wave train, and the effective diffusivity $\Deff > 0$ and nonlinear coefficient $\beta \in \R$ may be computed via a Lyapunov--Schmidt reduction argument as outlined in~\cite{DSSS}. To the right of the front interface, $\psi(\xi,t)$ is driven by the translational eigenmode and converges to a fixed asymptotic phase $\psi_\infty \in \R$ as $t \to \infty$. See the right two panels of Figure~\ref{fig: spacetime and sketch} for numerical simulations and sketches illustrating the dynamics of this phase function.
	
	To obtain such a description analytically, we rely on a careful decomposition of the linearized dynamics into localized modes arising from the translational eigenvalue at $0$, diffusive modes associated with the essential spectrum, and interactions between them. We obtain this description using far-field/core decompositions to analyze the associated resolvent problem, expanding upon techniques developed for the study of pulled fronts in~\cite{AveryScheelSIMA, AveryScheelGL, AveryScheelSelection, AverySelectionRD, FHNpulled}. 
	
	The resulting nonlinear dynamics are then roughly as follows: (i) the initial perturbation excites the translational mode, adjusting the position of the front interface; (ii) the phase of the pattern in the wake responds to this shift in the interface according to~\eqref{e: eikonal} with the negative group velocity $c_g < 0$ transporting the phase defect to the left; see Figure~\ref{fig: spacetime and sketch}. One challenge is then to rule out significant effects of back-coupling, that is, the phase response in the wake having a significant influence on the position of the front interface, leading to a further response of the phase. Such a cascade would make it difficult to close a nonlinear argument. The key observation which rules out this cascade is that the translational mode only responds very weakly to the dynamics in the wake. This is ultimately an effect of the outward pointing group velocity, and is manifested in the exponential localization of the adjoint eigenfunction associated with the translational eigenvalue at $0$. 
	
	\begin{figure}
		\centering
		\includegraphics[width=0.6\textwidth]{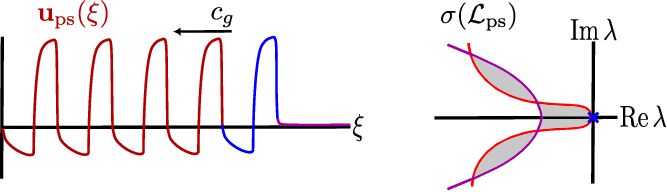}
		\caption{Left: schematic of a pushed pattern-forming front $\Ups(\xi)$. Pattern formation is driven by the front interface (blue) which creates wave trains in its wake with outward pointing group velocity $c_g < 0$. Right: spectrum of the linearization about the pushed pattern-forming front $\Ups(\xi)$, after stabilization with an exponential weight. Red and purple curves denote the spectra of $u \equiv 0$ and $\uwt$, which form the boundaries of the essential spectrum of $\Lps$ (shaded in grey). The blue cross at $\lambda = 0$ denotes the simple translational eigenvalue embedded in the essential spectrum. }

		\label{fig: front and spectrum}
	\end{figure}

    \medskip

    \noindent \textbf{Beyond pattern-forming fronts.} The technical tools we develop here are well-suited to any stability problem involving the interaction of (embedded) neutral eigenvalues with outgoing diffusive modes. Such problems naturally arise, for instance, in the stability of shock waves in viscous conservation laws~\cite{ZumbrunHoward,BeckSandstedeZumbrun} and of source defects in reaction-diffusion systems~\cite{BNSZ1, BNSZ2}. While nonlinear stability problems for viscous shocks are well-studied, the analysis there benefits from the conservation law structure which improves the behavior of nonlinear remainders and allows for some additional flexibility in the nonlinear argument. Likewise, the stability analysis for source defects in the complex Ginzburg--Landau equation in~\cite{BNSZ1} has benefited from an additional gauge symmetry. Our analysis here develops a framework which does not rely on these additional structures.

    A key step in all of these problems is to obtain detailed information on the linearized evolution through an analysis of the resolvent equation near neutral eigenvalues embedded in the essential spectrum. Previous approaches~\cite{ZumbrunHoward, BNSZ1, BNSZ2, BeckSandstedeZumbrun} analyze the resolvent via a spatial dynamics approach, recasting the resolvent equation as a first-order evolutionary system and then pasting together solutions which decay as $\xi \to \pm \infty$ to construct a Green's kernel for this system. We instead analyze the resolvent through a functional analytic approach, using partitions of unity to decompose into different spatial regions, a far-field/core ansatz to capture critical terms arising from the essential spectrum, and Fredholm properties of the linearization to solve for far-field parameters and core corrections. An advantage of our method over~\cite{ZumbrunHoward, BNSZ1, BNSZ2, BeckSandstedeZumbrun} is that it allows us to bound the linearized dynamics in fixed spatially weighted norms, avoiding the need for delicately chosen norms whose weights mix space and time dependence and must be carefully propagated through a nonlinear argument. Further perspectives on applying our approach to other diffusive nonlinear stability problems are provided in Section~\ref{sec:discussion}.

    \subsection{Notation} \label{sec:notation} We introduce (nonstandard) notation used throughout the manuscript. 

\medskip

    \noindent \textbf{Exponentially weighted Sobolev spaces.} To quantify spatial localization of perturbations, we employ exponentially weighted Sobolev spaces. Fix $\eta_\pm \in \R$, and define a smooth positive two-sided weight $\omega_{\eta_-, \eta_+}$ satisfying
	\begin{align*}
		\omega_{\eta_-, \eta_+}(\xi) = \begin{cases}
			\re^{\eta_- \xi}, & \xi \leq -1, \\
			\re^{\eta_+ \xi}, & \xi \geq 1.
		\end{cases}
	\end{align*}
	We choose the weight such that $\omega_{\eta_-,\eta_+}$ is non-decreasing for $\eta_\pm \geq 0$. Given $1 \leq p \leq \infty$, $\eta_\pm \in \R$, a field $\F \in \{\R, \C \}$ and integers $k \geq 0$ and $m \geq 1$, we define the exponentially weighted Sobolev space $\smash{W^{k,p}_{\mathrm{exp}, \eta_-, \eta_+} (\R, \F^m)}$ as
	\begin{align*}
		W^{k,p}_{\mathrm{exp}, \eta_-, \eta_+} (\R, \F^m) = \Big\{ \u \in W^{k,p}_\mathrm{loc}(\R, \F^m) : \| \u \|_{W^{k,p}_{\mathrm{exp}, \eta_-, \eta_+}} < \infty \Big\},
	\end{align*}
	where 
	\begin{align*}
		\| \u \|_{W^{k,p}_{\mathrm{exp}, \eta_-, \eta_+}} = \| \omega_{\eta_-, \eta_+} \u \|_{W^{k,p}}. 
	\end{align*}
	When $p = 2$, we write $\smash{W^{k,2}_{\mathrm{exp}, \eta_-, \eta_+} (\R, \F^m)} = \smash{H^k_{\mathrm{exp}, \eta_-, \eta_+} (\R, \F^m)}$. When $k = 0$, we write $\smash{W^{0,p}_{\mathrm{exp},\eta_-, \eta_+}(\R, \F^m)} = \smash{L^p_{\mathrm{exp}, \eta_-, \eta_+} (\R, \F^m)}$. When the ambient vector space is clear from context, we drop $(\R, \F^m)$ from our notation and simply write $\smash{W^{k,p}_{\mathrm{exp}, \eta_-, \eta_+}, H^k_{\mathrm{exp}, \eta_-, \eta_+}}$, or $\smash{L^p_{\mathrm{exp}, \eta_-, \eta_+}}$. 

    \medskip
 
    \noindent \textbf{Spaces of bounded functions.} We write $\mathcal{B}(X,Y)$ for the space of linear bounded operators between Banach spaces $X$ and $Y$. Often we simply write $\mathcal{B}(X)$ for $\mathcal{B}(X,X)$. Moreover, for integers $k \geq 0$ and $m \geq 1$, and a field $\F \in \{\R, \C \}$, we write $C_{\mathrm{ub}}^k(\R,\F^m)$, or simply $C_{\mathrm{ub}}^k(\R)$, for the space of bounded and uniformly continuous functions $u \colon \R \to \F^m$, which are $k$ times differentiable and whose $k$ derivatives are also bounded and uniformly continuous. We equip $C_{\mathrm{ub}}^k(\R,\F^m)$ with the standard $W^{k,\infty}$-norm, so that it is a Banach space. 

    \medskip

    \noindent \textbf{Partition of unity.} To separate the dynamics in the wake from those in the leading edge of the front, we introduce a smooth partition of unity $\chi_\pm \colon \R \to [0,1]$ such that $\chi_-$ is monotonically decreasing, $\chi_-(\xi) = 1$ for $\xi \leq -1$ and $\chi_-(\xi) = 0$ for $\xi \geq 0$. 

    \medskip

   	\noindent \textbf{Suppression of constants.} For a given set $S$ and functions $A,B : S \to \R$, the expression $A(x) \lesssim B(x)$ for $x \in S$ means that there exists a constant $C > 0$ independent of $x \in S$ such that $A(x) \leq C B(x)$ for all $x \in S$. 

    \medskip
    
	\noindent \textbf{Additional notation.} For a metric space $M$, $x \in M$, and $\delta > 0$, we denote by $B(x,\delta)$ the open ball of radius $\delta > 0$ centered at $x \in M$. We will often abuse notation by writing a function $\u(\xi, t)$ of space and time as $\u(t)$, viewing it as a function of time with values in a given Banach space. Similarly, we will often write a function $\u(\xi; \lambda)$ of a spectral parameter $\lambda$ as $\u(\lambda)$. 

    \subsection{Outline of paper}

    In Section~\ref{s: setup}, we precisely formulate our main result under general spectral stability assumptions, which are satisfied by pushed pattern-forming fronts in the FitzHugh--Nagumo system. Section~\ref{sec:techniques} outlines the proof strategy and provides a roadmap of the techniques used to obtain our main result. In Section~\ref{s: resolvent}, we carry out a far-field/core decomposition of the resolvent of the linearization about the front and establish corresponding estimates. These results are converted into a decomposition of the associated semigroup, together with suitable estimates, in Section~\ref{s: linear estimates}. We subsequently develop the nonlinear iteration scheme in Section~\ref{sec: nonlinear iteration}. Section~\ref{sec:proof} contains the nonlinear stability argument and yields the proof of the main result. A discussion and outlook on future research directions are provided in Section~\ref{sec:discussion}. Finally, Appendix~\ref{app: exp weight} collects several exponentially weighted mean-value-type estimates needed for our analysis.
 
    \section{Main result}\label{s: setup}

    In this section, we precisely formulate our main result on the nonlinear stability and selection of pushed pattern-forming fronts. The result is stated under general spectral stability assumptions, which indeed hold for the fronts of Theorem~\ref{t: existence}, by Theorem \ref{thm: spectral hyp}. 
	
	\subsection{Spectral stability assumptions}
	
    Pushed invasion fronts are characterized by marginally stable point spectrum, spectral stability of the selected state in the wake, and pointwise exponential stability of the unstable state in the leading edge of the front. We state explicit spectral conditions that capture these properties.

    \medskip

    \noindent \textbf{Spectral stability of wave train.} We begin with the spectral stability of the selected state in the wake which, in the present setting of a pushed pattern-forming front, is an $L$-periodic wave train $\uwt$. The linearization of~\eqref{e: FHN comoving} about $\uwt$ is given by the $L$-periodic differential operator
	\begin{align}
		\lwt = D \partial_\xi^2 + \cps \partial_\xi + F'(\uwt),
	\end{align}    
    acting on $L^2(\R) \times L^2(\R)$ with domain $H^2(\R) \times H^1(\R)$.  By Floquet--Bloch theory, see e.g.~\cite{ReedSimon}, the spectrum of   
    $\lwt$ is determined by the family of Bloch operators 
    \begin{align*}
		\Lwthat (\nu) = D (\partial_\xi + \nu)^2 + \cps (\partial_\xi + \nu) + F'(\uwt), \qquad \nu \in \C,
	\end{align*}
    acting on $L^2 (\R/L\Z) \times L^2 (\R/L\Z)$ with domain $H^2(\R / L \Z) \times H^1 (\R / L\Z)$. Since $\smash{\Lwthat(\nu)}$ has compact resolvent for each $\nu \in \C$, its spectrum consists of isolated eigenvalues of finite algebraic multiplicities only. In contrast, the spectrum of $\lwt$ is purely esssential and arises from the union of the spectra of $\lwt(\nu)$ for purely imaginary values of $\nu$. That is, we have the spectral relation
	\begin{align*}
		\Sigma(\lwt) = \bigcup_{k \in \big[- \tfrac{\kwt}{2}, \tfrac{\kwt}{2}\big)} \Sigma (\Lwthat(\ri k)), 
	\end{align*}
    where $\smash{\kwt = \frac{2\pi}{L}}$ is the wave number. By translational invariance, $0$ is an eigenvalue of $\smash{\Lwthat(0)}$ with eigenfunction $\smash{\uwt'}$. Consequently, the spectrum of the $\lwt$ must touch the imaginary axis at the origin. The following \emph{diffusive spectral stability} assumption, which is standard in the stability theory of periodic wave trains~\cite{SchneiderAbstract, JONZ, SSSU}, ensures that this touching is nondegenerate and the remainder of the spectrum is confined to the open left-half plane.
    
    \begin{hyp}[Diffusive spectral stability of wave train]\label{hyp: wave train stability}
    There exists $\theta > 0$ such that the following conditions hold:
		\begin{enumerate}
			\item $\Sigma(\lwt) \subset \{ \lambda \in \C: \Re \lambda < 0 \} \cup \{ 0\}$. 
			\item For any $k \in [-\kwt/2, \kwt/2)$, we have $\Re \Sigma(\Lwthat(\ri k)) \leq -\theta k^2$. 
			\item $\lambda = 0$ is a simple eigenvalue of $\Lwthat(0)$. 
		\end{enumerate}
	\end{hyp}

    By analytic perturbation theory~\cite{Kato1995Perturbation}, the simple $0$-eigenvalue of $\smash{\Lwthat}(0)$ can be continued analytically in $\nu$, resulting in a simple eigenvalue 
    \begin{align}
		\lambdawt(\nu) = - c_g \nu + \Deff \nu^2 + \mathrm{O}(\nu^3) \label{e: wt lambda expansion}
	\end{align}
    of $\smash{\Lwthat(\nu)}$ for $|\nu| \ll 1$. The coefficients $c_g \in \R$ and $\Deff > 0$ in the expansion~\eqref{e: wt lambda expansion} represent the group velocity and effective diffusivity of the wave train, respectively; see~\cite{DSSS} for details. Via a standard Lyapunov--Schmidt reduction argument, one obtains
    \begin{align} \label{e: group velocity}
    c_g &= 2\left\langle \u_{\mathrm{ad}}, D \partial_{\xi\xi} \uwt\right\rangle_{L^2(\R/L\Z,\C^2)} + \cps = 2\left\langle u_{\mathrm{ad},1}, \partial_{\xi\xi} u_{\mathrm{wt},1}\right\rangle_{L^2(\R/L\Z,\C)} + \cps, \end{align}
    where $\u_{\mathrm{ad}}$ spans the kernel of the adjoint operator $\smash{\Lwthat(0)^*}$ normalized by $\langle\uwt',\u_\mathrm{ad}\rangle_{L^2(\R/L\Z,\C^2)} = 1$. We assume that the group velocity,  which measures the speed of transport of small perturbations along the wave train, is negative, so that the front acts as a source of patterns. Note that we are measuring the group velocity in the frame \eqref{e: FHN comoving} already co-moving with the front speed $\cps$. 
	\begin{hyp}[Outgoing group velocity for the wave train]\label{hyp: negative group velocity}
		Given that Hypothesis~\ref{hyp: wave train stability} holds, the group velocity $c_g$, given by~\eqref{e: group velocity}, is negative.
	\end{hyp}

    \medskip
   
    \noindent \textbf{Pointwise exponential stability of leading edge.} Next, we turn to the state $\u \equiv 0$ in the leading edge of the front. Although this state is unstable, with perturbations growing in any translation-invariant norm, we require it to be \emph{pointwise exponentially stable}, that is, sufficiently localized perturbations decay exponentially at each \emph{fixed} point $\xi \in \R$. On the spectral level, this is reflected by the fact that the $L^2$-spectrum of the linearization 
   \begin{align*}
    \mathcal{A}_+ = D \partial_\xi^2 + \cps \partial_\xi + F'(0)
   \end{align*}    
   of~\eqref{e: FHN comoving} about $\u \equiv 0$ can be fully stabilized by introducing an exponential weight; see~\cite{Sandstede2000absolute}. The resulting exponentially weighted linearization is given by the constant-coefficient operator
   \begin{align*}
    \mcl_+ = D (\partial_\xi - \eta_0)^2 + \cps (\partial_\xi - \eta_0) + F'(0),
   \end{align*}
   posed on $L^2(\R) \times L^2(\R)$ with domain $H^2(\R) \times L^2(\R)$, where $\eta_0 > 0$ denotes the weight. Taking $\eta_0 < \etaps$ permits perturbations that decay slower than the tail of the front as $\xi \to \infty$; see~\eqref{e: front right asymptotics}. Allowing such perturbations in the nonlinear stability analysis of the front makes it possible to cut off the front tail and thereby admit initial data supported on a half-line $(-\infty,\xi_0]$ with $\xi_0 > 0$. We therefore impose the following assumption.
   \begin{hyp} [Pointwise exponential stability of leading edge]\label{hyp: leading edge}
    There exists $0 < \eta_0 < \etaps$ such that $\Sigma(\mcl_+) \subset \{\lambda \in \C : \Re \lambda < 0\}$. 
    \end{hyp}

    \begin{remark}{\upshape
    We note that \emph{pulled} fronts are characterized by the fact that the spectrum of $\mathcal{A}_+$ can only be \emph{marginally} stabilized by an exponential weight~\cite{AHSReview,pp}. More precisely, for a choice of exponential weight $\eta_0$, the spectrum of $\mcl_+$ is contained in the closed left-half plane, but Hypothesis~\ref{hyp: leading edge} is violated for any $\eta_0 > 0$. Ultimately, this is a consequence of the fact that pulled fronts propagate with the \emph{linear spreading speed} $\clin$ and have spatial tail decay rate $\etalin$, while pushed fronts travel with speed $\cps > \clin$ and exhibit steeper decay at rate $\etaps > \etalin$ in the leading edge; see~\cite{AHSReview,FHNpulled}. Here, the linear spreading speed $\clin > 0$ describes the rate at which disturbances of the unstable state $\u \equiv 0$ spread in the linearized equation $\u_t = D \u_{xx} + F'(0)\u$, whereas $\etalin$ corresponds to their pointwise exponential decay rate. The speed $\clin$ and decay rate $\etalin$ can be obtained by locating pinched double roots of the \emph{linear dispersion relation}
    \begin{align} \label{e: linear dispersion relation}
	d_c(\lambda,\nu) = \det\big[ D \nu^2 + c \nu I + F'(0) - \lambda I \big].
    \end{align}
    This computation has been carried out in the FitzHugh--Nagumo system in~\cite{CarterScheel}.}
    \end{remark}

    \noindent \textbf{Marginal stability of point spectrum.} The linearization of~\eqref{e: FHN comoving} about the pushed front $\Ups$ is given by the operator
    $$\Aps = D \partial_\xi^2 + \cps \partial_\xi + F'(\Ups),$$ 
    acting on on $L^2(\R) \times L^2(\R)$ with domain $H^2(\R) \times L^2(\R)$. By the Weyl essential spectrum theorem, the right boundary of the essential spectrum of $\Aps$ is given by the right boundary of $\Sigma(\lwt) \cup \Sigma(\mathcal{A}_+)$, which lies in the open right-half plane thanks to the instability of the rest state $\u \equiv 0$ in the leading edge of the front. To recover stability, we use Hypothesis~\ref{hyp: leading edge} to define a smooth positive exponential weight $\omega_0$ satisfying
	\begin{align*}
		\omega_0 (\xi) = \begin{cases}
			\re^{\eta_0 \xi}, &\xi \geq 1, \\
			1, & \xi \leq -1, 
		\end{cases}
	\end{align*}
    which restricts the allowed tail behavior of perturbations. The spectrum of $\Aps$ acting on this class of perturbations equals the spectrum of the conjugate operator 
    \begin{align} \label{e:defLps}
    \Lps = \omega_0 \Aps \left( \frac{\cdot}{\omega_0} \right) 
    \end{align}
    posed on $L^2(\R) \times L^2(\R)$ with domain $H^2(\R) \times H^1(\R)$. Hypotheses~\ref{hyp: wave train stability} and~\ref{hyp: leading edge} now guarantee that the right boundary of the essential spectrum of $\Lps$, which equals the right boundary of $\Sigma(\lwt) \cup \Sigma(\mcl_+)$, is confined to the open left-half plane, except for the quadratic touching of $\Sigma(\lwt)$ at $0$; see Figure~\ref{fig: front and spectrum}. This reflects that the dynamics of sufficiently localized perturbations of pushed fronts are dominated by localized effects in the front interface and diffusive behavior in the wake, rather than by the tail dynamics of the front.

    For pushed pattern-forming fronts, the linearization $\Lps$ possesses a marginal translational eigenvalue embedded in the essential spectrum of the wave train in the wake. The following spectral stability condition guarantees that this neutral eigenvalue is simple and that no additional critical point spectrum is present. To measure multiplicity, we use the exponentially weighted spaces introduced in~\S\ref{sec:notation}, which shift the essential spectrum of $\Lps$ into the open left-half plane and thereby isolate the translational eigenvalue.

   \begin{hyp}[Marginally stable point spectrum]\label{hyp: point spectrum} For any $\eta > 0$ sufficiently small, we have:
   \begin{enumerate}   
   \item[1.] There are no eigenvalues $\lambda \in \C \setminus \{0\}$ of $\Lps$ with $\Re \lambda \geq 0$.
   \item[2.] $0$ is a simple isolated eigenvalue of $\Lps$, considered as an operator on the exponentially weighted space $L^2_{\mathrm{exp}, \eta, 0} \times L^2_{\mathrm{exp}, \eta, 0} $ with domain $H^2_{\mathrm{exp}, \eta, 0} \times H^1_{\mathrm{exp}, \eta, 0} $. The range of the associated spectral projection is spanned by $\omega_0 \Ups'$.
   \end{enumerate}
	\end{hyp}

  A direct consequence of Hypothesis~\ref{hyp: point spectrum} is that $0$ is, for each $\eta > 0$ sufficiently small, likewise a simple eigenvalue of the $L^2$-adjoint operator $\Lps^*$, acting on the dual space $L^2_{\mathrm{exp}, -\eta, 0} \times L^2_{\mathrm{exp}, -\eta, 0}$ with domain $H^2_{\mathrm{exp}, -\eta, 0} \times H^1_{\mathrm{exp}, -\eta, 0}$. Let $\psiad \in H^2_{\mathrm{exp}, -\eta, 0} \times H^1_{\mathrm{exp}, -\eta, 0}$ denote the corresponding eigenfunction, normalized by
  \begin{align} \label{e:pairing_adjoint}
  \langle \omega_0 \Ups', \psiad \rangle_{L^2} = 1.
   \end{align}
  Clearly, the eigenfunction $\psiad$ is exponentially localized on the left. On the other hand, it follows from the Fredholm properties of $\Lps$ as an operator on $\smash{L^2_{\mathrm{exp}, \eta, 0} \times L^2_{\mathrm{exp}, \eta, 0}}$ that $\psiad$ is also exponentially localized on the right; see Proposition \ref{prop: fredholm properties}. The functional, given by
   \begin{align} \label{e: spectral functional}
    \Ptr \f = \langle \f, \psiad \rangle_{L^2},
  \end{align}
  then extracts the coefficient of the associated spectral projection. This functional plays a central role in our stability analysis, as it measures the excitation of the translational mode by perturbations. 
   
    \begin{remark} \label{rem:exp weight wake}
    { \upshape
	The characterization of the translational eigenvalue in Hypothesis~\ref{hyp: point spectrum} relies on the fact that the marginally stable essential spectrum of $\Lps$, associated with the wave train in the wake, can be fully stabilized by introducing an exponential weight. Such a weight permits perturbations that grow exponentially as $\xi \to -\infty$, in contrast to the weight $\omega_0$, which restricts to exponentially decaying perturbations. The necessity of a weight that admits exponentially growing perturbations is a consequence of the outgoing group velocity in the wake. We emphasize that such weights must be avoided in the \emph{nonlinear} stability analysis: allowing exponentially growing perturbations introduces exponentially growing coefficients in the nonlinear terms of the perturbation equation, thereby preventing the closure of a nonlinear argument.
    }
    \end{remark}

	\noindent \textbf{Verification of spectral stability assumptions.} We formulated the Hypotheses~\ref{hyp: wave train stability} through~\ref{hyp: point spectrum} for pushed pattern-forming fronts in general reaction-diffusion systems. We emphasize that these hypotheses can be naturally extended to \emph{modulated} pushed pattern-forming fronts, which are time-periodic in a co-moving frame and occur for instance near the onset of a Turing instability; see~\cite[Definition~6.6]{AHSReview}. Hypotheses~\ref{hyp: wave train stability} through~\ref{hyp: point spectrum} are validated for the pushed pattern-forming fronts in the FitzHugh--Nagumo system from Theorem~\ref{t: existence} in the companion papers~\cite{spectral,SpectralFronts,CarterScheel}, leading to the following result.
    
    \begin{thmlocal}[{\!\!~\cite{SpectralFronts}}] \label{thm: spectral hyp}
   Fix $0 < a < \frac13$ and $0 < \gamma < 4$. Then, there exist constants $C_0, \eps_0 > 0$ such that, for all $\eps \in (0,\eps_0)$, the pushed pattern-forming front solution $\u(x,t) = \Ups(x-\cps t)$ to~\eqref{e: FHN}, established in Theorem~\ref{t: existence}, fulfills Hypotheses~\ref{hyp: wave train stability} through~\ref{hyp: point spectrum}. In particular, Hypothesis~\ref{hyp: leading edge} holds for any $\eta_0$ in the open interval $(\etalin,\etaps)$, where $\etalin \in (0,\etaps)$ satisfies
   \begin{align*}
   \left|\etalin - \sqrt{a(1-a)}\right| \leq C_0\eps.
   \end{align*}
   \end{thmlocal}

   	\subsection{Statement of main result}
	
    We are now in position to state our main result precisely. We assume that Hypotheses~\ref{hyp: wave train stability}-\ref{hyp: point spectrum} hold and will be in force throughout the remainder of this paper.
   
	\begin{thmlocal}[Nonlinear stability of pushed pattern-forming fronts]\label{t: main}
 		Let $\u(x,t) = \Ups(x - \cps t)$ be a pushed pattern-forming front solution to~\eqref{e: FHN} obtained in Theorem~\ref{t: existence}, and suppose that Hypotheses~\ref{hyp: wave train stability} through~\ref{hyp: point spectrum} are satisfied. Fix constants $K,\delta_c > 0$ such that $c_g + \delta_c < 0$. Then, there exist $M, \delta_0, \mu > 0$ such that, whenever $\w_0 \in H^3(\R) \times H^2 (\R)$ satisfies 
		\begin{align*}
			E_0 := \| \omega_0 \w_0 \|_{H^3 \times H^2} < \delta_0,
		\end{align*}
        the following assertions hold:
        \begin{itemize}
        \item[(i)] \emph{(Global existence of nearby data).} There exists a unique global classical solution 
        \begin{align*}\u \in C\big([0, \infty), C_{\mathrm{ub}}^2(\R) \times C_{\mathrm{ub}}^1(\R) \big) \cap C^1\big([0,\infty),C_{\mathrm{ub}}(\R) \times C_{\mathrm{ub}}(\R)\big)\end{align*} to~\eqref{e: FHN comoving} with initial condition $\u(0) = \Ups + \w_0$.
        \item[(ii)] \emph{(Lyapunov stability).} We have
         \begin{align} \label{e:maindecaybounds}
        \begin{split}
            \|\omega_0 [\u(t) - \Ups]\|_{L^\infty} &\leq ME_0
       \end{split}
       \end{align}
        for all $t \geq 0$.
        \item[(iii)] \emph{(Locally uniform asymptotic orbital stability).} There exists an asymptotic phase shift $\psi_\infty \in \R$ with 
        \begin{align} \label{e:maindecaybounds21}
        |\psi_\infty| \leq ME_0, \qquad |\psi_\infty - \Ptr(\omega_0 \w_0)| \leq ME_0^2,
        \end{align}
        such that
        \begin{align} \label{e:maindecaybounds22}
        \sup_{\xi \in [-K,\infty)} \omega_0(\xi) |\u(\xi,t) - \Ups(\xi + \psi_\infty)| \leq ME_0 \re^{-\mu t}
        \end{align}
        for each $t \geq 0$, where $\Ptr$ is given by~\eqref{e: spectral functional}.
        \item[(iv)] \emph{(Asymptotic modulational stability).} There exists a smooth function $\psi \colon [0, \infty) \times \R \to \R$ such that
		\begin{align} \label{e:maindecaybounds3}
       \begin{split}
			\big\| \omega_0 [ \u(t) - \Ups(\cdot + \psi(t)) ]\big\|_{H^3 \times H^2} + \| \nabla \psi(t)\|_{H^3} &\leq M E_0(1+t)^{-\frac14}, \\
            \big\| \omega_0 [ \u(t) - \Ups(\cdot + \psi(t)) ]\big\|_{L^\infty} + \| \nabla \psi(t)\|_{L^\infty} &\leq M E_0 (1+t)^{-\frac12},\\
            \|\psi(t)\|_{L^\infty} &\leq ME_0
      \end{split}
      \end{align}
       for $t \geq 0$, where $\nabla \psi(t) = (\psi_\xi(t),\psi_t(t))^\top$ denotes the spacetime gradient.
       \item[(v)] \emph{(Light cone estimates).} The refined bounds
   	\begin{align} \label{e:refineddecaybounds}
  \begin{split}
			\sup_{\xi \geq (c_g + \delta_c) t} \Big[ \omega_0 (\xi) | \u(\xi, t) - \Ups(\xi + \psi_\infty)| + |\psi(\xi,t) - \psi_\infty|\Big] &\leq M E_0 \re^{-\mu t}, \\
			\sup_{\xi \leq (c_g - \delta_c) t} \Big[\omega_0 (\xi) | \u(\xi, t) - \Ups(\xi) | + |\psi(\xi,t)|\Big] &\leq M E_0 (1+t)^{-\frac{1}{4}}
   \end{split}
		\end{align}
    hold for all $t \geq 0$.
       \end{itemize}
   \end{thmlocal}
    
    Roughly speaking, Theorem~\ref{t: main} asserts that solutions to~\eqref{e: FHN comoving} arising from small perturbations to $\Ups$, which decay at least as fast as $\smash{\re^{-\eta_0 \xi}}$ as $\xi \to \infty$, converge to a fixed spatial translate $\Ups(\cdot + \psi_\infty)$ of the front as $t \to \infty$, locally uniformly on $\R$. Recall from Theorem~\ref{t: existence} that the front itself has asymptotics~\eqref{e: front right asymptotics} with $\etaps > \eta_0$. That is, the front decays to zero at a faster exponential rate than $\re^{-\eta_0 \xi}$ as $\xi \to \infty$. Hence, the initial conditions allowed in Theorem~\ref{t: main} include initial data which are compactly supported on the right. 
	
	\begin{corollary}[Selection of pushed pattern-forming fronts]
		The basin of attraction in~\eqref{e: FHN comoving} of the family of translates $\{\Ups(\cdot + \psi_0) : \psi_0 \in \R \}$ of the front includes initial conditions which vanish on a half line $[\xi_0,\infty)$ for some $\xi_0 > 0$. In particular, $\Ups$ is a selected front in the sense of~\cite[Definition 1]{AveryScheelSelection}.
	\end{corollary}
	To our knowledge, this is the first result which establishes selection of a pattern-forming front from steep initial conditions. 
 
    The refined estimates~\eqref{e:refineddecaybounds} in Theorem~\ref{t: main} give further details of the convergence to $\Ups$, in particular explaining why convergence to the translated front is only locally uniform. Recall from Hypothesis~\ref{hyp: negative group velocity} that $c_g < 0$, where $c_g$ is the group velocity describing the speed of transport of small perturbations along the wave train. The light cone estimates~\eqref{e:refineddecaybounds} say that the convergence to the spatial translate $\Ups(\cdot + \psi_\infty)$ of the front is mediated by a ``phase front'', connecting the two phases $\psi_0 = 0$ and $\psi_\infty$, which propagates to the left with speed approximately equal to $c_g$, thereby rigorously confirming the intuition discussed in Section~\ref{s: overview} and depicted in Figure~\ref{fig: spacetime and sketch}. We emphasize that estimate~\eqref{e:maindecaybounds21} implies that generically $\psi_\infty \neq 0$, since the kernel of the functional $\Ptr$ has codimension 1.

    The temporal convergence rates towards the modulated front solution in~\eqref{e:maindecaybounds3} coincide with those obtained in~\cite{JNRZ_13_1,SSSU} for phase-mixing problems of wave trains in reaction-diffusion systems. As explained in~\cite[Section~6.1]{deRijknonlocalized}, these rates reflect the diffusive behavior of solutions to the viscous eikonal equation~\eqref{e: eikonal}, which governs the leading-order phase dynamics of the wave train~\cite{JNRZWhitham,SSSU,DSSS}, and are therefore sharp. 
 
    We present the proof of Theorem~\ref{t: main} in Section~\ref{sec:proof}. The proof strategy and an overview of used techniques is provided in the upcoming Section~\ref{sec:techniques}.

    \begin{remark}{ \upshape
    Light cone estimates of the form~\eqref{e:refineddecaybounds} have also been obtained in the nonlinear stability analysis of source defects in~\cite{BNSZ2}. The left light cone estimate in~\eqref{e:refineddecaybounds} corresponds to cutting off the portion of the front where the phase defect is present and recovers the standard diffusive decay rates for $L^2$-localized perturbations of wave trains. However, the corresponding outer light cone estimates for source defects in~\cite{BNSZ2} yield \emph{exponential} convergence in time. This is a consequence of the Gaussian localization assumptions imposed on perturbations in~\cite{BNSZ2}. Under analogous Gaussian localization hypotheses, one would likewise expect exponential decay in the left light cone estimate in~\eqref{e:refineddecaybounds}. We do not pursue this direction here.
    }\end{remark}

    \section{Overview of techniques} \label{sec:techniques}

    In this section, we present our strategy to prove Theorem~\ref{t: main} and provide an overview of the main techniques.

    \medskip
    
	\noindent\textbf{Inverse Laplace representation of linear evolution.} Key to the proof of Theorem~\ref{t: main} is a detailed description of the linearized dynamics of~\eqref{e: FHN comoving} near the pushed pattern-forming front $\Ups$. To this end, we represent the semigroup $\re^{\Lps t}$ generated by the weighted linearization~\eqref{e:defLps} via the inverse Laplace formula, so that we can decompose and eventually control the linearized dynamics through a careful analysis of the resolvent operator $(\lambda - \Lps)^{-1}$. 
	
	\begin{prop}[Inverse Laplace representation of semigroup]\label{p: C0 semigroup}
	Fix $k \in \N_0$. The operator $\Lps$, acting on $H^k(\R) \times H^k(\R)$ with domain $D(\Lps) = H^{k+2} (\R) \times H^{k+1}(\R)$, generates a strongly continuous semigroup $\re^{\Lps t}$. Moreover, there exists $\Lambda > 0$ such that every $\lambda \in \C$ with $\Re(\lambda) \geq \Lambda$ lies in the resolvent set $\rho(\Lps)$ and the inverse Laplace representation
			\begin{align}
				\re^{\Lps t} \u = \frac{1}{2 \pi \ri} \lim_{R\to \infty} \int_{\Lambda - \ri R}^{\Lambda + \ri R} \re^{\lambda t} (\lambda - \Lps)^{-1} \u \, \de \lambda \label{e: contour rep}
			\end{align}
	holds for any $\u \in D(\Lps)$ and $t > 0$, where the limit in~\eqref{e: contour rep} is taken with respect to the $H^k$-norm.
	\end{prop}
	Proposition~\ref{p: C0 semigroup} follows from standard semigroup theory~\cite{EngelNagel,ArendtBattyHieber2011,Pazy}. Its proof is identical to that of~\cite[Proposition~3.1]{Alexopoulos2026} and~\cite[Proposition~4.1 and Corollary~4.2]{FHNpulled}. 

    \medskip
    	
	\noindent\textbf{Resolvent analysis.} Since $\Lps$ only generates a $C^0$-semigroup, spectral mapping identities of the form $\Sigma(\re^{\Lps t}) \setminus \{0\} = \re^{\Sigma(\Lps) t}$, converting spectral bounds into temporal growth or decay bounds, are not automatic. However, in the stability analysis of pulled fronts in~\eqref{e: FHN}, spectral mapping properties were established directly by analyzing the high-frequency behavior of the resolvent~\cite{FHNpulled}. This analysis, depending only on the basic structure of the equation rather than on the particular solution we are linearizing about, carries over to the current setting without modification. Since the spectrum of $\Lps$ away from a small neighborhood of the origin is contained in the open left-half plane, the contour in~\eqref{e: contour rep} can be shifted into the left-half plane except in this neighborhood; see Figure~\ref{fig: contour shifting}. Hence, the leading-order linear dynamics arise from the contour integral near the marginally stable spectrum of $\Lps$ in the vicinity of $\lambda = 0$, and all other contributions decay exponentially in time. 

    A key step, then, is to understand the behavior of the resolvent $(\lambda - \Lps)^{-1}$ near $\lambda = 0$, and translate this understanding to temporal linear estimates via the inverse Laplace representation~\eqref{e: contour rep}. The challenge is that the resolvent operator is not defined on, say, $L^2(\R)$ in a neighborhood of $\lambda = 0$ due to the presence of essential spectrum. To study the resolvent problem $(\lambda - \Lps) \u = \g$ near $\lambda = 0$, we use the partition of unity $(\chi_-, \chi_+)$ defined in~\S\ref{sec:notation} to decompose the data as $\g = \chi_-\g + \chi_+ \g =: \g_- + \g_c$. We will construct $\u$ via the far-field/core ansatz $\u = \chi_- \u_- + \u_c$, where we  choose $\u_-$ as a solution to the far-field problem $(\lambda - \lwt) \u_- = \g_-$. Since $\lwt$ has periodic coefficients, we can use Floquet--Bloch theory to obtain a (relatively) explicit Green's kernel for this far-field resolvent problem, and thereby identify the unique solution $\u_-$ that is bounded for $\lambda$ to the right of the essential spectrum and admits a pointwise analytic continuation past the essential spectrum. It then remains to solve for the center piece $\u_c$, which satisfies $(\lambda - \Lps) \u_c = \tg$, where the modified data $\tg(\lambda)$ is exponentially localized with rate uniform in $\lambda$. This exponential localization recovers Fredholm properties of the operator $\lambda-\Lps$, so that we can solve for $\u_c$ using Lyapunov--Schmidt reduction. We carry out this procedure, which is inspired by related work on the stability of pulled fronts~\cite{AveryScheelSIMA, AveryScheelGL, AveryScheelSelection, AverySelectionRD, FHNpulled}, in Section~\ref{s: resolvent}. This ultimately yields a meaningful pointwise representation $\u(\xi; \lambda)$ of the solution to the resolvent problem $(\lambda - \Lps) \u = \g$ near $\lambda = 0$. 

    As in the linear stability analyses~\cite{ZumbrunHoward, BNSZ2} of viscous shock waves and source defects, the embedded eigenvalue at $\lambda = 0$ gives rise to a (simple) pole of the pointwise resolvent. We find, to leading order, 
    \begin{align}
        \u(\xi; \lambda) \approx \omega_0 (\xi) \Ups' (\xi) \frac{1}{\lambda} \left[ \chi_-(\xi) \re^{\nuwt(\lambda) \xi} + 1 - \chi_-(\xi)\right] \Ptr \g \label{e: intro transl mode}
    \end{align}
    near $\lambda = 0$, where $\Ptr$ is the functional given by~\eqref{e: spectral functional} and $\nuwt(\lambda)$ is the critical Floquet exponent of the periodic far-field problem $(\lambda - \lwt) \u = 0$ with $\nuwt(0) = 0$. The corresponding leading-order \emph{excited term} in~\eqref{e: contour rep} does not decay in time, but instead captures the error-function dynamics of the phase $\psi(\xi, t)$ pictured in Figure~\ref{fig: spacetime and sketch}. Terms in~\eqref{e: contour rep} which only ``see'' the effects of the essential spectrum are faster decaying, and are referred to as \emph{scattering terms} (adopting terminology from~\cite{ZumbrunHoward}).

    \medskip

    \noindent\textbf {Linearized dynamics --- intuition.} The slowest decaying terms in the inverse Laplace representation~\eqref{e: contour rep} of the semigroup thus correspond to the translational mode~\eqref{e: intro transl mode}, which gives rise to a pole of the resolvent. The coefficient of this term is determined by the functional $\Ptr \g$, representing the pairing of $\g$ with the adjoint eigenfunction $\psiad$. Since $\psiad$ is exponentially localized, we expect that although these leading-order modes do not decay in time, they only weakly respond to behavior on the left, gaining us spatial localization in the form of a factor $\omega_{\kappa, 0}$ with $\kappa > 0$ in our estimates. 

    The next-order terms in~\eqref{e: contour rep} correspond to the neutrally stable curve of essential spectrum of $\lwt$. The expansion~\eqref{e: wt lambda expansion} of this curve formally encodes dynamics of the form
    \begin{align}
        \u_t = \Deff \u_{\xi \xi} - c_g \u_{\xi} \label{e: intro scattering terms}
    \end{align}
    with effective diffusivity $\Deff > 0$. These dynamics are pointwise exponentially stable due to the outward transport arising from the group velocity $c_g < 0$. That is, sufficiently localized solutions to~\eqref{e: intro scattering terms} decay exponentially at any fixed point $\xi \in \R$, though not in any translation-invariant norm. We can quantify this pointwise decay by measuring in the weighted norm $\|\omega_{\kappa, 0} \u\|_{L^\infty}$, which gives less weight to components of the solution as they are advected to the left. Conjugating the model equation~\eqref{e: intro scattering terms} with a weight $\re^{\kappa \xi}$ and inspecting the resulting spectral gap, we expect decay rates 
    \begin{align*}
        \| \omega_{\kappa, 0} \u(t) \|_{L^\infty} \sim \re^{(\kappa c_g + \mathrm{O}(\kappa^2)) t} \| \omega_{\kappa, 0} \u(0) \|_{L^\infty}
    \end{align*}
    for $t\geq0$ and $\kappa > 0$. Recall that $c_g < 0$, so this estimate measures decay with a rate slightly slower than $\kappa c_g$. Quantifying this decay rate will be crucial in obtaining exponential convergence to the phase-shifted front to the right of the characteristic $\xi \approx c_g t$; see the right light cone estimate in~\eqref{e:refineddecaybounds}. 

    \medskip
    
    \noindent\textbf{Linearized dynamics --- estimates.} In Section~\ref{s: linear estimates}, we combine the contour integral representation~\eqref{e: contour rep} of the semigroup with the resolvent analysis developed in Section~\ref{s: resolvent}. This yields a detailed description of the linearized dynamics and rigorously confirms the heuristics outlined above. The resulting semigroup decomposition, together with the associated linear estimates, are summarized in the following result. 
  
	\begin{thmlocal}[Semigroup decomposition and linear estimates] \label{t: linear estimates}
    Fix $\Delta \mu > 0$ such that $c_g + \Delta \mu < 0$. Let $j,k \in \N_0$. The semigroup $\re^{\Lps t}$ generated by $\Lps$ may be extended to a bounded linear operator on $L^2(\R)$, or on $C_0(\R)$, and admits the decomposition
		\begin{align} \label{e: semigroup decomposition}
  			\re^{\Lps t} &= \omega_0 \Ups' s_p(t) + s_c(t) + s_e(t), \qquad s_j(t) = s_{j,1} (t) + s_{j,2} (t), \qquad j = p,c
		\end{align}        
	      for $t \geq 0$, where $s_p(t)$ vanishes identically for $t \in [0,1]$. Moreover, there exist constants $C,D_0,\kappa,\mu > 0$ such that the following estimates hold for all $t \geq 0$, $\xi \in \R$, $\u \in L^2(\R) \cap C_0(\R)$:
		\begin{enumerate}[(i)]
			\item \emph{(Estimates on leading-order excited terms).}
            \begin{align} \label{e:error_est}
                \big|[\partial_\xi^j \partial_t^k s_{p,1}(t) \u] (\xi)\big| \leq C \left( \left| \partial_\xi^j \partial_t^k \erf \left( \frac{\xi - c_g t}{\sqrt{D_0 (1+t)}} \right) \right| + \re^{-\mu t} \re^{-\mu |\xi|}\right) \| \omega_{\kappa, 0} \u\|_{L^\infty},
            \end{align} 
         where $\erf(z) = \int_{-\infty}^z \re^{-w^2} \, \de w$ is the error function. Hence, if $j + k \geq 1$, we have
			\begin{align*}
				\| s_{p,1} (t) \u\|_{L^\infty} &\leq C \| \omega_{\kappa, 0} \u \|_{L^\infty}, \\
				\| \partial_\xi^j \partial_t^k s_{p,1} (t) \u \|_{L^\infty} &\leq C (1+t)^{-1/2} \| \omega_{\kappa, 0} \u \|_{L^\infty}, \\
                \|\chi_- s_{p,1} (t) \u \|_{L^2} & \leq C (1+t)^{1/2} \| \omega_{\kappa, 0} \u \|_{L^\infty}, \\
				\| \partial_\xi^j \partial_t^k s_{p,1} (t) \u \|_{L^2} &\leq C (1+t)^{-1/4} \| \omega_{\kappa, 0} \u \|_{L^\infty}, \\
				\| \omega_{\kappa, 0} \partial_\xi^j \partial_t^k s_{p,1}(t) \u \|_{L^\infty} &\leq C \re^{ \kappa (c_g + \Delta \mu) t} \| \omega_{\kappa, 0} \u \|_{L^\infty}.
			\end{align*}
			\item \emph{(Asymptotics for leading-order excited terms).}
			\begin{align*}
				\| \omega_{\kappa, 0} [s_{p,1} (t) - \Ptr] \u \|_{L^\infty} &\leq C \re^{ \kappa ( c_g + \Delta \mu) t} \| \omega_{\kappa, 0} \u \|_{L^\infty}, \\
				| \Ptr \u | &\leq C \| \omega_{\kappa, 0} \u \|_{L^\infty}. 
			\end{align*}
			
			\item \emph{(Estimates on leading-order scattering terms).}
			\begin{align*}
                \| s_{p,2} (t) \u \|_{L^2} &\leq C\| \u \|_{L^2}, \\
				\| s_{p,2} (t) \u \|_{L^\infty} &\leq C (1+t)^{-1/4} \| \u \|_{L^2}, \\
                \| \omega_{\kappa, 0} \partial_\xi^j \partial_t^k s_{p,2}(t) \u \|_{L^\infty} &\leq C \re^{ \kappa ( c_g + \Delta \mu) t} \| \omega_{\kappa, 0} \u \|_{L^\infty}.
            \end{align*}
            Additionally, if $j+k \geq 1$, we have
            \begin{align*}
				\| \partial_\xi^j \partial_t^k s_{p,2} (t) \u \|_{L^\infty} &\leq C (1+t)^{-3/4} \| \u \|_{L^2}, \\
				\| \partial_\xi^j \partial_t^k s_{p,2} (t) \u \|_{L^2} &\leq C (1+t)^{-1/2} \| \u \|_{L^2}.
            \end{align*}
			\item \emph{(Estimates on residual excited terms).}
			\begin{align*}
				\| s_{c,1}(t) \u \|_{L^\infty} &\leq C (1+t)^{-1/2} \| \omega_{\kappa, 0} \u \|_{L^\infty}, \\
				\| s_{c,1}(t) \u \|_{L^2} &\leq C (1+t)^{-1/4} \| \omega_{\kappa, 0} \u \|_{L^\infty}, \\
				\| \omega_{\kappa, 0} s_{c,1} (t) \u \|_{L^\infty} &\leq C \re^{ \kappa ( c_g + \Delta \mu) t} \| \omega_{\kappa, 0} \u \|_{L^\infty}. 
			\end{align*}
			\item \emph{(Estimates on residual scattering terms).} 
			\begin{align*}
				\| s_{c,2} (t) \u \|_{L^\infty} &\leq C (1+t)^{-3/4} \| \u \|_{L^2}, \\
				\| s_{c,2} (t) \u \|_{L^2}  &\leq C (1+t)^{-1/2} \| \u \|_{L^2}, \\
				\| \omega_{\kappa, 0} s_{c,2} (t) \u \|_{L^\infty} &\leq C \re^{ \kappa (c_g + \Delta \mu) t} \| \omega_{\kappa, 0} \u \|_{L^\infty}. 
			\end{align*}
			\item \emph{(Estimates on exponentially damped terms).}
			\begin{align*}
				\| s_e (t) \u \|_{L^2} &\leq C \re^{-\mu t} \| \u \|_{L^2}, \\
				\| s_e (t) \u \|_{L^\infty} &\leq C \re^{-\mu t} \| \u \|_{L^\infty}, \\
				\| \omega_{\kappa, 0} s_e(t) \|_{L^\infty} &\leq C \re^{- \mu t} \| \omega_{\kappa, 0} \u \|_{L^\infty}.
			\end{align*}
		\end{enumerate}
	\end{thmlocal}

    \medskip

\noindent\textbf{Nonlinear iteration.} We now turn to the strategy for the nonlinear stability argument that leads to the proof of Theorem~\ref{t: main}. Let $\u(t)$ denote the solution to~\eqref{e: FHN comoving} with initial condition $\u(0) = \Ups + \w_0$, where $\w_0$ is sufficiently small and localized. Our objective is to control the long-time behavior of the perturbation $\w(t) = \u(t) - \Ups$. To stabilize the unstable state in the leading edge of the front, we impose exponential localization of $\w_0$ on $[0,\infty)$. This motivates the introduction of the weighted perturbation $\vt(t) = \omega_0 \w(t)$, which has the initial condition $\vt(0) = \v_0 := \omega_0 \w_0$ and obeys the semilinear equation
    \begin{align}
     \vt_t = \Lps \vt + \NT(\vt), \label{e:umodpert intro}
     \end{align}
with quadratic nonlinear remainder
\begin{align*}
\NT(\vt) = \omega_0\widetilde{N}\left(\omega_0^{-1} \vt\right), \qquad 
\widetilde{N}(\w) = F(\Ups+\w) - F(\Ups) - F'(\Ups)\w = \mathrm{O}(\w^2).
\end{align*}
As outlined in Section~\ref{s: overview}, a standard approach is to iteratively estimate the associated Duhamel formula
\begin{align*}
    \vt(t) = \re^{\Lps t} \v_0 + \int_0^t \re^{\Lps (t-s)} \NT(\vt(s)) \de s.
\end{align*}
The strategy is then to show that, if $\vt(t)$ decays as prescribed by the linearized dynamics $\re^{\Lps t} \v_0$, then the nonlinear term $\smash{\int_0^t \re^{\Lps (t-s)} \NT(\vt(s)) \, \de s}$ obeys the same temporal decay rate. By Theorem~\ref{t: linear estimates} the linearized dynamics of $\vt(t)$ are given by 
\begin{align} \label{e: linearized dynamics}
\re^{\Lps t} \v_0 &= \omega_0 \Ups' s_p(t) \v_0 + \mathrm{O}\left(\frac{\|\omega_{\kappa,0} \v_0\|_{L^\infty}}{\sqrt{1+t}}\right) + \mathrm{O}\left(\frac{\|\v_0\|_{L^2}}{(1+t)^{\frac34}}\right) \\
&= \omega_0 \Ups' s_{p,1}(t) \v_0 + 
\mathrm{O}\left(\frac{\|\omega_{\kappa,0} \v_0\|_{L^\infty}}{\sqrt{1+t}}\right) + \mathrm{O}\left(\frac{\|\v_0\|_{L^2}}{(1+t)^{\frac14}}\right) \label{e: linearized dynamics 2}
\end{align} 

From~\eqref{e: linearized dynamics 2} we observe that, similar to the case of pushed fronts selecting constant states considered in~\S\ref{s: overview}, decay is obstructed by the principal part $s_{p,1}(t)\v_0$ of the semigroup. Yet, comparing~\eqref{e: linearized dynamics 2} with the corresponding decomposition~\eqref{e: nagumo linear estimate} of the linearized dynamics for pushed fronts selecting constant states, we identify two key differences. The first is that $s_{p,1}(t) \v_0$ does not equal (and is in fact not even uniformly converging to) the constant phase shift $\Ptr \v_0$; rather, $s_{p,1}(t)$ resembles a leftward-propagating error function; see Theorem~\ref{t: linear estimates}(i). Second, the residual term in~\eqref{e: linearized dynamics 2} only exhibits diffusive decay, which in general is too weak to close a nonlinear argument in the presence of quadratic nonlinearities. For instance, in the case of the nonlinear heat equation $u_t = u_{xx} + u^2$ all nonnegative nontrivial initial data blow up in finite time~\cite{FUJI}. The diffusive decay of the residual term in~\eqref{e: linearized dynamics 2} stems from the scattering terms in~\eqref{e: semigroup decomposition}, which reflect the diffusive stability of the wave train in the wake.

A natural attempt to address both of these challenges would be to shift the essential spectrum associated with the wave train into the open left-half plane by applying the exponential weight $\omega_{\kappa,0}$ with $\kappa > 0$. Then one recovers the spectral setting of pushed fronts selecting constant states, where the translational $0$-eigenvalue is separated from the remainder of the spectrum. Multiplying~\eqref{e: linearized dynamics 2} with $\omega_{\kappa,0}$ and recalling the estimates from Theorem~\ref{t: linear estimates}, we obtain
\begin{align} \label{e: weighted linearized dynamics}
\omega_{\kappa,0} \re^{\Lps t} \v_0 = \omega_{\kappa,0} \omega_0 \Ups' \Ptr \v_0 + \mathrm{O}\left(\re^{-\mu t}\|\omega_{\kappa,0} \v_0\|_{L^\infty}\right).
\end{align}
The weighted linearized dynamics~\eqref{e: weighted linearized dynamics} indeed closely resembles that of the pushed fronts selecting constant states, cf.~\eqref{e: nagumo linear estimate}. However, the weight $\omega_{\kappa,0}$ is incompatible with closing a nonlinear argument, since the weighted nonlinear remainder $\smash{\omega_{\kappa,0} \widetilde{\mathcal{N}}(\omega_{\kappa,0}^{-1} \mathbf{y})}$ has exponentially growing coefficients and hence cannot be effectively bounded in terms of $\mathbf{y}$.

We resolve this issue by \emph{coupling} the nonlinear system for the weighted perturbation $\omega_{\kappa,0} \vt(t)$ to that of $\vt(t)$ itself. This allows us to bound the weighted nonlinear remainder $\omega_{\kappa,0} \widetilde{\mathcal{N}}(\vt)$ in terms of $\omega_{\kappa,0} \vt$ and $\vt$. We can now proceed as in the nonlinear stability analysis~\cite{Sattinger} for pushed fronts selecting constant states to control the weighted perturbation $\omega_{\kappa,0}\vt(t)$. That is, we introduce the \emph{inverse-modulated perturbation}
\begin{align} \label{e: def inverse mod pert}
\v(t) = \omega_0 \left(\u(\cdot - \psi(t),t) - \Ups\right),
\end{align}
so that the phase function $\psi(t)$ can accommodate the excitation caused by the translational mode. Subsequently, an iterative argument for the weighted perturbation $\z(t) := \omega_{\kappa,0} \v(t)$ with exponential temporal decay rates can be closed \emph{as long as} $\v(t)$ stays bounded. 

It remains to close a nonlinear argument for $\v(t)$ itself, where we can now exploit that $\z(t)$ decays exponentially in time. Due to the exponential localization of the adjoint eigenfunction $\psiad$ associated with the translational mode, we gain a factor $\omega_{\kappa,0}$ upon applying the ``excited'' terms $s_{p,1}(t)$ and $s_{c,1}(t)$ in the semigroup decomposition~\eqref{e: semigroup decomposition} to nonlinearities in the Duhamel formula. These nonlinearities thus inherit exponential decay in time from $\z(t)$. The obtained exponential decay is sufficient to control the excited terms in the nonlinear argument, allowing us to focus on the scattering terms. 

The diffusive scattering terms $s_{p,2}(t)$ and $s_{c,2}(t)$ in~\eqref{e: semigroup decomposition} are a direct manifestation of the critical curve of essential spectrum associated with the diffusive spectral stability of the wave train in the wake. Such scattering terms also arise in the decomposition of the semigroup $\re^{\lwt t}$ generated by the linearization of~\eqref{e: FHN comoving} about the wave train itself. As mentioned before, an important obstruction is that the weak diffusive decay exhibited by $s_{p,2}(t)$ is insufficient to close a nonlinear argument. To address this issue we take inspiration from the nonlinear stability analyses~\cite{JONZ,SSSU} of wave trains in reaction-diffusion systems, which rely on the observation that the leading-order diffusive dynamics of perturbations of wave trains can be captured by a spatio-temporal phase modulation~\cite{DSSS} as suggested by the linearized scattering dynamics, cf.~\eqref{e: linearized dynamics}, of the perturbed solution
$$\Ups + \Ups' s_{p,2}(t) \v_0 \approx \Ups(\cdot + s_{p,2}(t) \v_0).$$ 

Thus, we allow the phase function $\psi(t)$ in~\eqref{e: def inverse mod pert} to depend on both time and space such that it can accommodate the most critical diffusive dynamics exhibited by the scattering terms. One then obtains that the spatio-temporally modulated perturbation $\v(t)$ obeys a quasilinear equation, whose nonlinearity depends on $\v(t)$ and \emph{derivatives} of $\psi(t)$ only, which are expected to exhibit better decay rates due to diffusive smoothing; see Theorem~\ref{t: linear estimates}. In fact, in the nonlinear stability analyses~\cite{JONZ,SSSU} of wave trains against localized perturbations it turns out that all nonlinear terms in a scheme consisting of the variables $\v(t),\psi_\xi(t)$ and $\psi_t(t)$ are \emph{irrelevant} and, thus, a nonlinear argument can be closed. In the current setting the phase function $\psi(t)$ also accounts for the excitation attributed to the translational eigenvalue at $0$. This causes the decay rates of $\psi(t)$ and its derivatives to worsen by a factor $\smash{t^{-\frac12}}$ rendering the worst terms in the nonlinearity \emph{marginal}. 

The key observation allowing us to handle these marginal (and all other) nonlinear terms in the iteration argument is that we can distribute localization between $s_{p,2}(t)$ (or $s_{c,2}(t)$) and the nonlinearity. An optimal distribution is reached by employing an  $L^2$-$L^\infty$-iteration scheme, which is, as far as the authors are aware, nonstandard in the current nonlinear stability literature. We refer to Remark~\ref{rem: motivation L2 Linfty scheme} for more details.

A last remaining challenge is that the apparent loss of derivatives in the quasilinear equation for $\v(t)$ needs to be addressed. We tackle this issue by following the same strategy as in the case of pulled pattern-forming fronts~\cite{FHNpulled}, which employs forward-modulated damping estimates~\cite{ZUM22} to control regularity. We refer to~\cite[Section~2.2]{FHNpulled} for an overview of this approach.  
	
\section{Resolvent analysis near the origin}\label{s: resolvent}
        In this section, we decompose the resolvent $(\lambda - \Lps)^{-1}$ near the origin using a far-field/core ansatz. In Section~\ref{s: linear estimates}, this decomposition and the associated bounds are then transferred to the semigroup $\re^{\Lps t}$ via the inverse Laplace transform, yielding the proof of Theorem~\ref{t: linear estimates}. We therefore consider the resolvent equation
        \begin{align}
            (\lambda - \Lps) \u = \g \label{e: full resolvent eqn}
        \end{align}
        near $\lambda = 0$ and decompose the data $\g$ as $\g = \chi_- \g + \chi_+ \g  =: \g_- + \g_+$. To obtain a far-field/core decomposition for the solution $\u$, we let $\u_-$ solve
        \begin{align}
        (\lambda - \lwt) \u_- = \g_-. \label{e: left resolvent eqn}
        \end{align}
        If the data $\g$ belongs to, say, $L^2(\R)$ or $L^\infty(\R)$, then~\eqref{e: left resolvent eqn} is uniquely solvable in that space for $\lambda$ to the right of the essential spectrum of $\lwt$, and the solutions are analytic in $\lambda$ in this region. So, we may use this solution in an ansatz $\u = \chi_- \u_- + \u_c$ for the full resolvent problem~\eqref{e: full resolvent eqn}. We then find that the center component $\u_c$ solves 
        \begin{align}
        (\lambda - \Lps) \u_c = \tg (\lambda), \label{e: center resolvent eqn}
        \end{align}
        where $\tg(\lambda) = \g - (\lambda - \Lps) (\chi_- \u_-(\lambda))$.

        Our goal is to extend the solution $\u(\lambda)$ to~\eqref{e: full resolvent eqn} into the essential spectrum near $\lambda = 0$. We find that it is (pointwise) meromorphic in a neighborhood of the origin, with the embedded translational eigenvalue of $\Lps$ contributing a simple pole at $\lambda = 0$. We first analyze extensions of the far-field component $\u_- (\lambda)$ through the essential spectrum. We then show that in a neighborhood of the origin $\tg(\lambda)$ is exponentially localized in space, with rate uniform in $\lambda$. This exponential localization recovers Fredholm properties of $\Lps$, which we ultimately use to solve~\eqref{e: center resolvent eqn} via a Lyapunov--Schmidt type reduction procedure. 

\medskip

        \noindent \textbf{Fredholm properties.} Recall that an operator $\mcl: X \to Y$ between Banach spaces $X$ and $Y$ is \emph{Fredholm} if
        \begin{enumerate}[(i)]
            \item the range of $\mcl$ is closed in $Y$; 
            \item the kernel $\ker (\mcl)$ is finite-dimensional; 
            \item the cokernel $\mathrm{coker}(\mcl)$ is finite-dimensional,
        \end{enumerate}
        and the \emph{Fredholm index} of $\mcl$ is given by
        \begin{align*}
        \mathrm{ind}(\mcl) = \dim \ker(\mcl) - \dim \mathrm{coker}(\mcl).
        \end{align*}
        Fredholm properties of the linearization near traveling waves may be characterized via Palmer's theorem~\cite{Palmer2}; see for instance~\cite{KapitulaPromislow, SandstedeReview, FiedlerScheel} for a review. In particular, the Fredholm index is determined solely by the linearizations about the asymptotic end states of the traveling wave at $\xi=\pm\infty$ and may thus be affected by the choice of exponential weights.

        In the present setting, the relevant exponentially weighted spaces are
        \begin{align*}
            X_\eta^p := L^p_{\mathrm{exp}, -\eta, 0}  \times L^p_{\mathrm{exp}, -\eta, 0} , \qquad Y_\eta^p := W^{2, p}_{\mathrm{exp}, -\eta, 0} \times W^{1,p}_{\mathrm{exp}, -\eta, 0} 
        \end{align*}
        for $\eta \in \R$ and $1 \leq p \leq \infty$; see Section~\ref{sec:notation} for notation. The following result characterizes the Fredholm properties of $\Lps$ acting on these spaces. These properties are fully determined by the assumptions on the operators $\lwt$, $\mcl_+$, and $\Lps$ specified in Hypotheses~\ref{hyp: wave train stability} through~\ref{hyp: point spectrum}.
        \begin{prop} \label{prop: fredholm properties}
            Let $1 \leq p \leq \infty$. For any $\eta>0$ sufficiently small, the following assertions hold:
            \begin{enumerate}[(i)]
                \item Consider $\Lps$ as an operator $\Lps \colon Y_{-\eta}^p \to X_{-\eta}^p$. Then, $\Lps$ is Fredholm of index $0$, its kernel is spanned by $\omega_0 \Ups'$, and $\omega_0 \Ups'$ does not lie in the range of $\Lps$. 
                \item Consider $\Lps$ as an operator $\Lps \colon Y_{\eta}^p \to X_{\eta}^p$. Then, $\Lps$ is Fredholm of index $-1$, has trivial kernel, and its cokernel is spanned by $\psiad$, which is exponentially localized.
            \end{enumerate}
        \end{prop}

        \subsection{Resolvent in the wake}\label{s: wake resolvent}

        The resolvent operator for the linearization of the FitzHugh--Nagumo system about a diffusively spectrally stable wave train was analyzed in~\cite[Section~6.1]{FHNpulled}. This was done by recasting the resolvent problem~\eqref{e: left resolvent eqn} as a first-order $L$-periodic system, using Floquet theory to transform to a constant-coefficient system, and writing down an explicit Green's kernel for this system. A crucial role in extending the solution $\u_-(\lambda)$ of~\eqref{e: left resolvent eqn} through the spectrum of $\lwt$ is played by the \emph{critical spatial Floquet exponent} $\nuwt(\lambda)$ of this first-order periodic system, which is analytic in $\lambda$ and arises by inverting~\eqref{e: wt lambda expansion} around $\lambda = 0$. This yields the expansion
        \begin{align} \label{e:nuwt_expansion}
        \nuwt(\lambda) = -c_g^{-1} \lambda + \mathrm{O}(\lambda^2)
        \end{align}
        for $|\lambda| \ll 1$. In an open disk $B(0, \delta)$ of radius $\delta > 0$ centered at the origin, the spectrum of $\lwt$ may then equivalently be described as
       \begin{align} \label{e:spectral_id}
         \Sigma(\lwt) \cap B(0, \delta) = \{ \lambda \in B(0, \delta) : \nuwt(\lambda) \in \ri \R \}. 
       \end{align}
       Thus, crossing the spectrum of $\lwt$ near the origin corresponds to a sign change of $\Re \nuwt(\lambda)$, which leads to a loss of localization of $\u_-(\lambda)$. Despite this loss of localization, $\u_-(\lambda)$ remains pointwise analytic near the origin. The decomposition of $\u_-(\lambda)$ into pointwise analytic terms and an $L^p$-analytic remainder was carried out in~\cite[Section~6.1]{FHNpulled}. We summarize the result here.
    
        \begin{lemma} \label{lem: summary wave train resolvent}
        	Fix $1 \leq p \leq \infty$. There exist constants $C,\delta> 0$ such that for all $\lambda \in B(0,\delta)$ and $\g \in L^p(\R)$ the resolvent problem~\eqref{e: left resolvent eqn} possesses a solution $\u_-(\lambda) \in W^{2,p}_{\mathrm{loc}}(\R) \times W^{1,p}_\mathrm{loc}(\R)$ of the form
        	\begin{align}
        		\u_- (\lambda) = \uwt' \bar{s}_p^\mathrm{wt}(\lambda) \g_- + \bar{s}_c^\mathrm{wt}(\lambda) \g_- + \bar{s}_e^\mathrm{wt}(\lambda) \g_-, \label{e: u minus}
        	\end{align}
        	where:
        	\begin{enumerate}
                \item The residual $\bar{s}_c^\mathrm{wt}(\lambda) \g_-$ has the form
        		\begin{align} \begin{split}
        			[\bar{s}_c^\mathrm{wt}(\lambda) \g_-] (\xi) &= Q_2(\xi, \lambda) \int_\R \re^{\nuwt(\lambda) (\xi - \zeta)} \chi_- (\xi - \zeta) P_\mathrm{wt}^\mathrm{cu}(\lambda) Q_1(\zeta, \lambda) \g_- (\zeta) \, \de \zeta \\
        			&\qquad -\, Q_2(\xi, 0) \int_\R \re^{\nuwt(\lambda) (\xi -\zeta)} \chi_- (\xi - \zeta) P_\mathrm{wt}^\mathrm{cu}(0) Q_1(\zeta, 0) \g_-(\zeta) \, \de \zeta, \end{split}\label{e: scwt def}
        		\end{align} 
        		where $P_\mathrm{wt}^\mathrm{cu} \colon B(0,\delta) \to \C^{3\times3}$, $Q_1 \colon \R \times B(0,\delta) \to \C^{3\times 2}$, and $Q_2 \colon \R \times B(0,\delta) \to \C^{2\times3}$ are such that $P_\mathrm{wt}^\mathrm{cu}$ and $Q_{1,2}(\zeta,\cdot)$ are analytic for each $\zeta \in \R$, and $Q_{1,2}(\cdot,\lambda)$ is smooth and $L$-periodic for each $\lambda \in B(0,\delta)$. 
        		\item The leading-order part $\bar{s}_p^\mathrm{wt} (\lambda) \g_-$ is given by 
        		\begin{align}
        			[\bar{s}_p^\mathrm{wt}(\lambda) \g_-] (\xi) = \phi \left( \int_\R \re^{\nuwt(\lambda) (\xi - \zeta)} \chi_- (\xi-\zeta) P_\mathrm{wt}^\mathrm{cu}(0) Q_1(\zeta, 0) \g_- (\zeta) \, \de \zeta \right), \label{e: spwt def}
        		\end{align}
        		where $\phi \colon \C^3 \to \C$ is a linear map. 
        		\item The residual $\bar{s}^\mathrm{wt}_e \colon B(0, \delta) \to \mathcal{B}(L^p(\R))$ is analytic and we have
                \begin{align*}
                   \|\bar{s}^\mathrm{wt}_e(\lambda) \g \|_{L^p} \leq C \| \g \|_{L^p}.
                \end{align*} 
        	\end{enumerate}
            Moreover, if $\lambda \in B(0,\delta)$ lies to the right of $\Sigma(\lwt)$, then it holds $\u_-(\lambda) \in W^{2,p}(\R) \times W^{1,p}(\R)$.
        \end{lemma}

        Localization of the pointwise analytic solution~\eqref{e: u minus} to the far-field resolvent problem~\eqref{e: left resolvent eqn} can be recovered by measuring in exponentially weighted spaces for $\lambda$ near the origin, which is reflected by the fact that the spectrum of $\lwt$ can be shifted away from the origin by introducing an exponential weight. This observation leads to the following weighted $L^p$-estimates.
                
        \begin{lemma}\label{l: left resolvent exponentially localized estimates} Fix $1 \leq p \leq \infty$. For any $\kappa_0 > 0$ sufficiently small, there exist constants $C, \delta, \kappa > 0$ such that
        	\begin{align}
            \begin{split}
        		\left\| \re^{\kappa_0 \cdot} \chi_- \u_- (\cdot; \lambda) \right\|_{L^p} &\leq C \| \re^{\kappa \cdot} \g_- \|_{L^p}, \\
        		\left\| \re^{\kappa_0 \cdot} \chi_- [ \u_- (\cdot; \lambda) - \u_- (\cdot; 0)] \right\|_{L^p} &\leq C |\lambda| \| \re^{\kappa \cdot} \g_- \|_{L^p}\end{split} \label{e: left resolvent exponentially localized estimates}
        	\end{align}
            for all $\lambda \in B(0,\delta)$ and $\g \in L^p(\R)$, where $\u_-(\lambda)$ is the solution to~\eqref{e: left resolvent eqn}, established in Lemma~\ref{lem: summary wave train resolvent}.
        \end{lemma}
        \begin{proof}
        	We focus on the term $\uwt' \bar{s}_p^\mathrm{wt}(\lambda) \g_-$ in the decomposition~\eqref{e: u minus} of $\u_-(\lambda)$. The estimates on the other terms are either analogous or strictly easier. For this term, we have 
        	\begin{align*}
        		\left| \re^{\kappa_0 \xi} \chi_-(\xi) \uwt' (\xi) [\bar{s}_p^\mathrm{wt}(\lambda) \g_-] (\xi) \right| \leq C \chi_-(\xi) \int_\R \re^{\kappa_0 \xi} \left| \re^{\nuwt(\lambda) (\xi - \zeta)} \right| \chi_- (\xi - \zeta) |\g_- (\zeta)| \, \de \zeta
        	\end{align*}
        	for $\xi \in \R$. Since $\chi_-$ is supported on $(-\infty,0]$, we have $\xi \leq \zeta$ on the region of integration, and hence we may estimate $\smash{\re^{\kappa_0 \xi/2} \leq \re^{\kappa_0 \xi/4} \re^{\kappa_0 \zeta/4 }}$. Using these exponentially decaying factors and~\eqref{e:nuwt_expansion} to control $\re^{\nuwt(\lambda) (\xi - \zeta)}$ for $|\lambda|$ sufficiently small and applying H\"older's inequality, we obtain 
        	\begin{align*}
        		\left| \re^{\kappa_0 \xi} \chi_-(\xi) \uwt' (\xi) [\bar{s}_p^\mathrm{wt}(\lambda) \g_-] (\xi) \right| \leq C \re^{\frac{\kappa_0}{4} \xi} \chi_-(\xi) \left\| \re^{ \frac{\kappa_0}{4} \cdot } \g_- \right\|_{L^p} \left\|\re^{\frac{\kappa_0}{4} \cdot} \chi_-\right\|_{L^{p'}}
        	\end{align*}
            for $\xi \in \R$, where $p' \in [1,\infty]$ is the H\"older conjugate of $p$. Taking the $L^p$-norm then gives the first estimate in~\eqref{e: left resolvent exponentially localized estimates}. The proof of the second estimate in~\eqref{e: left resolvent exponentially localized estimates} is analogous, with the exponential localization in both $\xi$ and $\zeta$ allowing us to Taylor expand the exponential to extract a factor of $|\nuwt(\lambda)| \sim |\lambda|$. 
        \end{proof} 
        In solving the center resolvent equation~\eqref{e: center resolvent eqn}, we will also make use of the pointwise analytic solution $\e_-(\lambda)$ to the
        homogeneous problem $(\lambda - \lwt) \u = 0$ with $\e_-(0) = \uwt'$, which loses localization as $\lambda$ passes through the spectrum of $\lwt$. We characterize this solution in the following result from~\cite{FHNpulled}.
        
       \begin{lemma}[\!{\!\cite[Lemma~6.10]{FHNpulled}}]\label{l: left homogeneous solution}
        	For any $\delta> 0$ sufficiently small, there exists a solution $\e_- (\lambda) = \q (\lambda) \re^{\nuwt(\lambda) \cdot }$ to $(\lambda - \lwt) \u = 0$, where $\q \colon B(0,\delta) \to H^\ell_{\mathrm{per}}(\R)$ is analytic for any $\ell \in \N$. Moreover, we have $\q(0) = \uwt'$. 
        \end{lemma}
        
        \subsection{Far-field/core decomposition}
        We now aim to solve the center resolvent equation~\eqref{e: center resolvent eqn} for $\u_c$. To do this, we leverage the fact that we have captured the far-field behavior of $\g(\xi)$ for $\xi \ll -1$ with the asymptotic solution $\u_-$ to~\eqref{e: left resolvent eqn}, and so we gain exponential localization of the data $\tg(\lambda)$ in~\eqref{e: center resolvent eqn} as $\xi \to -\infty$.
        
        \begin{lemma}[Control of center data]\label{l: control of center data}
    	Fix $1 \leq p \leq \infty$. There exists $\kappa > 0$ such that, for any $\eta > 0$ sufficiently small, there exist constants $C,\delta > 0$ such that
    	\begin{align}
    		\left\| \tg (\lambda) \right\|_{X^p_\eta} \leq C \left\| \omega_{\kappa,0} \g \right\|_{L^p}, \qquad  \left\|\tg(\lambda) - \tg(0) \right\|_{X^p_\eta} \leq C | \lambda| \left\|\omega_{\kappa,0} \g \right\|_{L^p}\label{e: center data exponentially localized estimates}
    	\end{align}
        for all $\lambda \in B(0,\delta)$ and $\g \in L^p(\R)$, where we denote $\tg(\lambda) = \g - (\lambda - \Lps) (\chi_- \u_-(\lambda))$ and $\u_-(\lambda)$ is the solution to~\eqref{e: left resolvent eqn}, established in Lemma~\ref{lem: summary wave train resolvent}. 
        \end{lemma}
        \begin{proof}
        To take advantage of the fact that $\u_-$ solves the far-field resolvent equation~\eqref{e: left resolvent eqn}, we first rewrite $\tg(\lambda)$ as 
        \begin{align*}
        	\tg(\lambda) = \g - (\lambda-\lwt) (\chi_- \u_-) + (\Lps - \lwt) (\chi_- \u_-).
        \end{align*}
        Then, we use~\eqref{e: left resolvent eqn} to obtain
        \begin{align*}
        	\tg(\lambda) = \g - \chi_-^2 \g + [\lwt, \chi_-] \u_- + (\Lps - \lwt) (\chi_- \u_-),
        \end{align*}	
        where $[A,B] = AB - BA$ denotes the commutator of operators $A$ and $B$. The first term, $\g - \chi_-^2 \g$, is supported on $[-1,\infty)$ and is independent of $\lambda$, and so readily satisfies the desired estimates. For the terms involving $\u_-$, note that the operators $[\lwt, \chi_-]$ and $(\Lps - \lwt) [\chi_- \cdot]$ each have coefficients which are supported on $(-\infty,0]$ and are exponentially localized by Theorem~\ref{t: existence}. Hence, by Lemma~\ref{l: left resolvent exponentially localized estimates}, there exist $\kappa,\kappa_0 > 0$ such that, for any $\eta > 0$ sufficiently small, there exist $C,\delta > 0$ such that
        \begin{align*}
        	\left\| [\lwt, \chi_-] \u_- \right\|_{X^p_\eta} + \left\| (\Lps - \lwt) (\chi_- \u_-) \right\|_{X^p_\eta} \leq C \left\| \re^{\kappa_0 \cdot} \chi_- \u_- \right\|_{L^p} \leq C \left\| \re^{\kappa \cdot} \g_-\right\|_{L^\infty} \leq C \left\| \omega_{\kappa,0} \g\right\|_{L^\infty}
        \end{align*}
         for any $\lambda \in B(0,\delta)$ and $\g \in L^\infty(\R)$. The proof of the second estimate in~\eqref{e: center data exponentially localized estimates} is analogous. 
        \end{proof}

        Even though $\tg$ is exponentially localized on the left, the solution $\u_c$ to the center resolvent equation~\eqref{e: center resolvent eqn} will still lose localization as $\lambda$ approaches the essential spectrum of $\Lps$, which is an obstacle to inverting $\lambda-\Lps$ using the Fredholm properties stated in Proposition~\ref{prop: fredholm properties}. To overcome this, we explicitly capture the loss of localization through the \emph{far-field/core ansatz}
        \begin{align}
        	\u_c (\lambda) = \alpha_- \chi_- \e_- (\lambda) + \w, \label{e: ff core ansatz}
        \end{align}
        where we take $\alpha_- \in \C$, $\e_-(\lambda)$ the solution to $(\lambda - \lwt) \u = 0$ captured in Lemma~\ref{l: left homogeneous solution}, and $\w$ exponentially localized. Inserting the ansatz~\eqref{e: ff core ansatz} into~\eqref{e: center resolvent eqn} leads to an equation
        \begin{align*}
        	G (\w, \alpha_-; \lambda) = \tg(\lambda), 
        \end{align*}
        where 
        \begin{align}
        	G(\w, \alpha_-; \lambda) = (\lambda - \Lps)[\alpha_- \chi_- \e_- (\lambda) + \w]. \label{e:defG}
        \end{align}
        Since $\tg$ is exponentially localized on the left, we consider $G$ as an operator between exponentially weighted spaces. 
        
        \begin{lemma}
        	Fix $1 \leq p \leq \infty$. For all sufficiently small $\eta,\delta > 0$, the map 
        	\begin{align*}
        		G : Y^p_\eta \times \C \times B(0,\delta) \to X^p_\eta,
        	\end{align*}
        	given by~\eqref{e:defG}, is well-defined and analytic in its last component. 
        \end{lemma}
        \begin{proof}
            Lemma~\ref{l: left homogeneous solution} yields that $G$ is analytic in its last component. The main issue is then to make sure that $G$ preserves the exponential localization  of the input $\w$. This follows in a manner similar to the proof of Lemma~\ref{l: control of center data}, using the facts that, by Theorem~\ref{t: existence}, the coefficients of $\Lps$ converge exponentially to those of $\lwt$ as $\xi \to -\infty$ and that $\e_-$ satisfies $(\lambda - \lwt) \e_- = 0$.
        \end{proof}

        Using the Fredholm properties of $\Lps$ on exponentially weighted spaces established in Proposition~\ref{prop: fredholm properties}, we can apply the analytic Fredholm theorem to invert the linear map $G(\cdot,\cdot\,;\lambda)$. This then leads to a solution to the center resolvent equation~\eqref{e: center resolvent eqn}, which is meromorphic in $\lambda$ in a small neighborhood of the origin.
        
        \begin{prop}\label{p: ff core inversion}
        	Fix $1 \leq p \leq \infty$. For any sufficiently small $\eta, \delta > 0$, the linear map
        	\begin{align} \label{e:lin1}
        	Y_\eta^p \times \C \to X_\eta^p, \qquad (\w, \alpha_-) \mapsto G(\w, \alpha_-; \lambda)
        	\end{align}
        	has an inverse 
            \begin{align*}
            X_\eta^p \to Y_\eta^p \times \C, \qquad \f \mapsto \left(T(\lambda) \f,A_-(\lambda) \f\right)
            \end{align*}
            for each $\lambda \in B(0,\delta) \setminus \{0\}$. Here, $T \colon B(0,\delta) \setminus \{0\} \to \smash{\mathcal{B}(X_\eta^p,Y_\eta^p)}$ and $A_- \colon B(0,\delta) \setminus \{0\} \to \smash{\mathcal{B}(X_\eta^p,\C)}$ are meromorphic on $B(0,\delta)$ with a simple pole at the origin. In particular, there exist analytic maps $\smash{\tilde{\w}} \colon B(0,\delta) \to \smash{\mathcal{B}(X_\eta^p,Y_\eta^p)}$ and $\smash{\tilde{\alpha}_-} \colon B(0,\delta) \to \smash{\mathcal{B}(X_\eta^p,\C)}$ such that 
            \begin{align} \label{e:laurent inversion}
        		T(\lambda) \f = \frac{\Ptr \f}{\lambda} [\omega_0 \Ups' - \chi_- \uwt'] + \tilde{\w}(\lambda)\f,\qquad A_-(\lambda) \f = \frac{1}{\lambda} \Ptr \f + \tilde{\alpha}_-(\lambda) \f 
        	\end{align}
        	  for $\f \in X_\eta^p$ and $\lambda \in B(0,\delta)\setminus\{0\}$, where $\Ptr$ is given by~\eqref{e: spectral functional}.
        \end{prop}
        \begin{proof}
        	For $\eta > 0$ sufficiently small, Proposition~\ref{prop: fredholm properties} yields that the operator $\Lps \colon Y_{\eta}^p \to X_{\eta}^p$ is Fredholm of index $-1$. Hence, the Fredholm bordering lemma implies that the linear map $(\w, \alpha_-) \mapsto G(\w, \alpha_-; 0) : Y_\eta^p \times \C \to X_\eta^p, $ is Fredholm of index $0$. It then follows from the analytic Fredholm theorem that there exists $\delta > 0$ such that the linear map~\eqref{e:lin1} has an inverse for $\lambda \in B(0,\delta) \setminus \{0\}$, which is meromorphic in $\lambda$ on $B(0,\delta)$. This establishes the first part of the statement.
            
        	Let $\f \in X_\eta^p$. All that remains is to compute the leading-order coefficients of the Laurent series
            \begin{align*}
        		A_-(\lambda)\f = \sum_{k = -n}^\infty \alpha_{k} \lambda^k, \qquad T(\lambda) \f = \sum_{k = -n}^\infty \w_k \lambda^k,
        	\end{align*}
            where $n \in \N_0$ is the multiplicity of the pole at $0$. Arguing by contradiction, we asssume $n > 1$. Inserting $\u_c(\lambda) = [A_-(\lambda) \f]\chi_- \e_-(\lambda) + T(\lambda) \f$ into~\eqref{e: center resolvent eqn} and equating coefficients at order $\lambda^{-n}$ and $\lambda^{-n+1}$, we find
        	\begin{align}
        		\Lps \u_{-n} = 0, \qquad \Lps \u_{-n+1} = \u_{-n}, \label{e: laurent expansion}
        	\end{align}
        	where $\u_{-n} = \alpha_{-n} \chi_- \e_-(0) + \w_{-n}$, and $\u_{-n+1}= \alpha_{-n+1} \chi_- \e_- (0) + \alpha_{-n} \chi_- \partial_\lambda \e_- (0) + \w_{-n+1}$. Lemma~\ref{l: left homogeneous solution} implies that $\u_{-n}(\xi)$ and $\u_{-n+1}(\xi)$ grow at most linearly in $\xi$ as $\xi \to -\infty$ and are $L^p$-localized on $[0,\infty)$. Hence, they belong to $L^p_{\mathrm{exp}, \eta, 0}$ for $\eta > 0$ and we conclude from~\eqref{e: laurent expansion} and Proposition~\ref{prop: fredholm properties}(i) that we must have $\u_{-n} = \beta_{-n} \omega_0 \Ups'$ for some $\beta_{-n} \in \C$. The second equation in~\eqref{e: laurent expansion} then becomes 
        	\begin{align*}
        		\Lps \u_{-n+1} = \beta_{-n} \omega_0 \Ups'.
        	\end{align*}
        	However, again by Proposition~\ref{prop: fredholm properties}(i), $\omega_0 \Ups'$ is not in the range of $\Lps$ on $L^p_{\mathrm{exp}, \eta, 0} $, and so we must have $\beta_{-n} = 0$ and $\u_{-n} = 0$. Since $\e_-(0) = \uwt'$ is $L$-periodic and $\w_{-n}$ is localized on $(-\infty,0]$, this implies that $\alpha_{-n} = 0$ and $\w_{-n} = 0$, which contradicts that $n$ is the multiplicity of the pole at $0$. So, we have $n \leq 1$.
            
            Inserting $\u_c(\lambda) = [A_-(\lambda) \f]\chi_- \e_-(\lambda) + T(\lambda) \f$ into~\eqref{e: center resolvent eqn} and equating at order $\lambda^{-1}$ and $\lambda^0$ now yields the identities
        	\begin{align} \label{e: Laurent 2}
        		\Lps \u_{-1} = 0, \qquad \Lps \u_0 = \u_{-1} - \f. 
        	\end{align}
        	We still conclude that $\u_{-1} = \beta_{-1} \omega_0 \Ups'$ for some $\beta_{-1} \in \C$. Subsequently, we solve the second equation in~\eqref{e: Laurent 2} in order to determine the coefficient $\beta_{-1}\in \C$. Note that we have $\Lps \u_0 \in X_\eta^p$ and $\psiad$ is exponentially localized by Proposition~\ref{prop: fredholm properties}. Therefore, provided $\eta > 0$ is sufficiently small, all boundary terms vanish after pairing and integrating by parts. Consequently,
            \begin{align*}
            \langle \Lps \u_0, \psiad\rangle_{L^2} = \langle \u_0, \Lps^* \psiad\rangle_{L^2} = 0.
            \end{align*}
            Hence, taking the $L^2$-scalar product of the second equation in~\eqref{e: Laurent 2} with $\psiad$ and using~\eqref{e:pairing_adjoint} and~\eqref{e: spectral functional}, we find 
            \begin{align*}
                \beta_{-1} = \langle \f, \psiad \rangle_{L^2} = \Ptr \f.
            \end{align*}
            Using $\e_-(0) = \uwt'$, we infer
            \begin{align*}
                \alpha_{-1} \chi_- \uwt' + \w_{-1} = \u_{-1} = \beta_{-1} \omega_0 \Ups'.
            \end{align*}
            Since $\Ups'$ converges exponentially  to $\uwt$ as $\xi \to -\infty$ by Theorem~\ref{t: existence}, $\alpha_{-1} = \beta_{-1} = \Ptr \f$ is the only choice for $\alpha_{-1}$ which guarantees that
            \begin{align*}
            \beta_{-1} \omega_0 \Ups' - \alpha_{-1} \chi_- \uwt' = \w_{-1} \in Y_{\eta}^p
            \end{align*}
            is exponentially localized on the left.
            This yields that $\alpha_{-1} = \Ptr \f$, $\w_{-1} = \Ptr \f [\omega_0 \Ups' - \chi_- \uwt']$, and $n = 1$, as desired.
        \end{proof}

        With the aid of Proposition~\ref{p: ff core inversion}, we can solve the center resolvent problem~\eqref{e: center resolvent eqn}. The resulting solution $\u_c(\lambda)$ is meromorphic in $\lambda$ on a small neighborhood of the origin and possesses a simple pole at $\lambda = 0$. In the next result, we decompose $\u_c(\lambda)$ into a leading-order meromorphic part and an analytic remainder.
        
\begin{lemma} \label{lem: u c final}
Fix $1 \leq p \leq \infty$. There exists $\delta> 0$ such that for all $\lambda \in B(0,\delta) \setminus \{0\}$ and $\g \in L^p(\R)$ the center resolvent problem~\eqref{e: center resolvent eqn} possesses a solution $\u_c(\lambda) \in W^{2,p}_{\mathrm{loc}}(\R) \times W^{1,p}_\mathrm{loc}(\R)$ of the form
        \begin{align} \begin{split}
        	\u_c(\lambda) &= \frac{\Ptr \g}{\lambda} \omega_0 \Ups' \big[\chi_- \re^{\nuwt(\lambda) \cdot} + (1-\chi_-)\big] + \frac{\Ptr \tg(\lambda) - \Ptr \tg(0)}{\lambda} \omega_0 \Ups' \big[\chi_- \re^{\nuwt(\lambda) \cdot} + (1-\chi_-)\big]\\
            &\qquad + \, \frac{\Ptr \tg(\lambda)}{\lambda} \chi_- [\q(\lambda) - \q(0)] \re^{\nuwt(\lambda) \cdot} + \frac{\Ptr \tg(\lambda)}{\lambda} \chi_- (\uwt' - \omega_0 \Ups') \big(\re^{\nuwt(\lambda) \cdot}-1\big)\\
            &\qquad + \, \big[\tilde{\alpha}_-(\lambda)\tg(\lambda) \big] \chi_- \q(\lambda) \re^{\nuwt (\lambda) \cdot}  + \tilde{\w}(\lambda)\tg(\lambda).\end{split}\label{e: u c final}
        \end{align}
where we denote $\tg(\lambda) = \g - (\lambda - \Lps) (\chi_- \u_-(\lambda))$, where $\tilde{\alpha}_-(\lambda)$ and $\tilde{\w}(\lambda)$ are as in Proposition~\ref{p: ff core inversion}, where $\q(\lambda)$ is as in Lemma~\ref{l: left homogeneous solution}, and where $\u_-(\lambda)$ is the solution to~\eqref{e: left resolvent eqn}, established in Lemma~\ref{lem: summary wave train resolvent}. Moreover, if $\lambda \in B(0,\delta)$ lies to the right of $\Sigma(\lwt)$, then it holds $\u_c(\lambda) \in W^{2,p}(\R)\times W^{1,p}(\R)$.
\end{lemma}
\begin{proof}
        By Lemma~\ref{l: left homogeneous solution} and Proposition~\ref{p: ff core inversion}, there exists $\delta > 0$ such that
        \begin{align} \label{e:uc_expr}
        	\u_c (\lambda) &= [A_- (\lambda) \tg(\lambda)] \chi_- \q(\lambda) \re^{\nuwt(\lambda) \cdot} + T(\lambda)\tg(\lambda)
        \end{align}
        lies in $W^{2,p}_{\mathrm{loc}}(\R)\times W^{1,p}_\mathrm{loc}(\R)$ and solves~\eqref{e: center resolvent eqn} for $\lambda \in B(0,\delta) \setminus \{0\}$. Moreover, as $\Re \nuwt(\lambda)$ is positive for $\lambda \in B(0,\delta)$ to the right of $\Sigma(\lwt)$ by~\eqref{e:nuwt_expansion} and~\eqref{e:spectral_id}, it follows that $\u_c(\lambda) \in W^{2,p}(\R)\times W^{1,p}(\R)$ for all such $\lambda$. 

        To capture the leading-order contribution in the Laurent expansion of $\u_c(\lambda)$, we first insert~\eqref{e:laurent inversion} into~\eqref{e:uc_expr} and write
        \begin{align} \label{e:uc_expr2}
        \begin{split}
        	\u_c(\lambda) &= \frac{\Ptr \tg(\lambda)}{\lambda}\big[\omega_0 \Ups' - \chi_- \uwt'\big] + \frac{\Ptr \tg(\lambda)}{\lambda} \chi_- \q(\lambda) \re^{\nuwt(\lambda) \cdot}
            \\
            &\qquad + \,\big[\tilde{\alpha}_-(\lambda)\tg(\lambda) \big] \chi_- \q(\lambda) \re^{\nuwt (\lambda) \cdot} + \tilde{\w}(\lambda)\tg(\lambda).
            \end{split}
        \end{align}
        We proceed with computing
            \begin{align*}
                \Ptr \tg(0) = \langle \tg(0), \psiad \rangle_{L^2} = \langle \g, \psiad \rangle_{L^2} - \langle \Lps (\chi_- \u_-(0)), \psiad \rangle_{L^2}. 
            \end{align*}
        Note that: i) $\chi_- \u_-(0)$ is supported on $(-\infty, 0]$ and bounded by Proposition~\ref{lem: summary wave train resolvent}; and ii) $\psiad$ is exponentially localized by Proposition~\ref{prop: fredholm properties}. All boundary terms upon integration by parts therefore vanish, and one obtains 
            \begin{align*}
                \langle \Lps (\chi_- \u_- (0)), \psiad \rangle_{L^2} = \langle \chi_- \u_-(0), \Lps^* \psiad \rangle_{L^2} = 0,
        \end{align*}
        which implies
        \begin{align*}
        \Ptr \tg(0) = \Ptr \g.
        \end{align*}
        Inserting the latter into~\eqref{e:uc_expr2}, using $\q(0) = \uwt'$, and rearranging terms, we finally arrive at~\eqref{e: u c final}.
        \end{proof}
        
        Since $\tg(\lambda)$, $\q(\lambda)$, $\nuwt(\lambda)$, $\tilde{\alpha}_-(\lambda) \tg(\lambda)$, and $\tilde{\w}(\lambda)\tg(\lambda)$ are analytic in $\lambda$ by Lemmas~\ref{lem: summary wave train resolvent} and~\ref{l: left homogeneous solution} and Proposition~\ref{p: ff core inversion}, we observe that only the first term in~\eqref{e: u c final} has a simple pole at $\lambda = 0$, whereas the remainder terms are pointwise analytic in $\lambda$ on $B(0,\delta)$. 

        \subsection{Final resolvent decomposition}\label{s: final resolvent decomposition}
        
        Fix $1 \leq p \leq \infty$. By Lemmas~\ref{lem: summary wave train resolvent} and~\ref{lem: u c final}, there exists $\delta > 0$ such that $\u(\lambda) = \chi_- \u_-(\lambda) + \u_c(\lambda) \in W^{2,p}_{\mathrm{loc}}(\R) \times W^{1,p}_\mathrm{loc}(\R)$ is a solution to the resolvent equation~\eqref{e: full resolvent eqn} for all $\lambda \in B(0,\delta) \setminus \{0\}$,  where $\u_-(\lambda)$ is given by~\eqref{e: u minus} and $\u_c(\lambda)$ is given by~\eqref{e: u c final}. Moreover, if $\lambda \in B(0,\delta)$ lies to the right of $\Sigma(\lwt)$, then $\u(\lambda) \in W^{2,p}(\R)\times W^{1,p}(\R)$.

        To group terms which will contribute to the dynamics in similar ways, we use~\eqref{e: u minus} and~\eqref{e: u c final} to rewrite this solution as 
        \begin{align}
        \begin{split}
        	\u (\lambda) &= \omega_0 \Ups' [\bar{s}_{p,1} (\lambda) + \bar{s}_{p,2} (\lambda)] \g + [\bar{s}_{c,1} (\lambda) + \bar{s}_{c,2} (\lambda)] \g + \bar{s}_e (\lambda) \g,  
            \end{split} 
            \label{e: final resolvent decomp}
        \end{align}
        where 
        \begin{align}
        	\bar{s}_{p,1} (\lambda) \g = \frac{\Ptr \g}{\lambda} \big[ \chi_- \re^{\nuwt(\lambda) \cdot} + (1-\chi_-)\big] \label{e: sp1 resolvent def} 
        \end{align}
        contributes to non-decaying dynamics excited by the translational mode,
        \begin{align}
        	\bar{s}_{p,2} (\lambda) \g = \chi_- \bar{s}_p^\mathrm{wt}(\lambda) \g_- \label{e: sp2 def}
        \end{align}
        contributes to diffusively decaying dynamics associated with the spectrum of the wave train in the wake,
        \begin{align}
        \begin{split}
         	\bar{s}_{c,1} (\lambda) \g &= \left(\frac{\Ptr \tg(\lambda) - \Ptr \tg(0)}{\lambda} \omega_0 \Ups'  + \frac{\Ptr \tg(\lambda)}{\lambda} [\q(\lambda) - \q(0)] + \big[\tilde{\alpha}_-(\lambda)\tg(\lambda) \big]  \q(\lambda)\right)\chi_- \re^{\nuwt (\lambda) \cdot} \label{e: sc1bar def}
                    \end{split}
        \end{align}
        contributes to dynamics excited by the translational mode but exhibiting improved decay,
        \begin{align}
        	\bar{s}_{c,2} (\lambda) \g = \chi_- \bar{s}_c^\mathrm{wt}(\lambda) \g_- \label{e: sc2 def}
        \end{align}
        contributes to dynamics associated with the spectrum of the wave train in the wake but exhibiting enhanced diffusive decay, and
        \begin{align}
        \begin{split}
        	\bar{s}_e(\lambda) \g &= \big[\bar{s}_p^{\mathrm{wt}}(\lambda) \g_-\big] \chi_-\left(\uwt' - \omega_0 \Ups'\right) + \chi_- \bar{s}_e^\mathrm{wt}(\lambda) \g_- + \frac{\Ptr \tg(\lambda) - \Ptr \tg(0)}{\lambda} \omega_0 \Ups' (1-\chi_-)\\
            &\qquad + \,  \frac{\Ptr \tg(\lambda)}{\lambda} \chi_- (\uwt' - \omega_0 \Ups') \big(\re^{\nuwt(\lambda) \cdot}-1\big) + \tilde{\w}(\lambda)\tg(\lambda)
        \end{split}
        \end{align}
        captures those terms which are analytic in $\lambda$ in a neighborhood of the origin in $L^p(\R)$, and hence contribute to exponentially decaying dynamics. 

    	\section{Semigroup decomposition and linear estimates}\label{s: linear estimates}

        \begin{figure}
    		\centering
    		\includegraphics[width=0.65\textwidth]{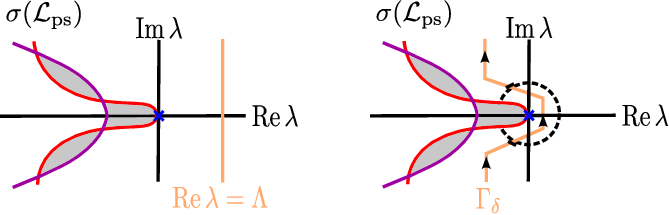}
    		\caption{Spectrum of $\Lps$ together with integration contours used in the inverse Laplace transform. Left: the initial contour $\Re \lambda = \Lambda > 0$ (tan) used in Proposition~\ref{p: C0 semigroup}. Right: the shifted contour $\Gamma_\delta$ (tan) used in Proposition~\ref{p: contour shifting}. The contour $\Gamma_\delta$ lies in the open left-half plane, except for a portion lying in the disk $B(0, \delta)$, whose boundary is denoted by the dashed curve.}
    		\label{fig: contour shifting}
        \end{figure}

       In this section, we obtain a decomposition of the $C_0$-semigroup $\re^{\Lps t}$ together with associated bounds. Our starting point is the inverse Laplace representation~\eqref{e: contour rep} of $\re^{\Lps t}$. The strategy is to shift the integration contour in~\eqref{e: contour rep} into the open left-half plane, except for a segment in a small neighborhood of the origin; see Figure~\ref{fig: contour shifting}. This naturally leads to a decomposition of $\re^{\Lps t}$ into an exponentially damped high-frequency contribution and a critical low-frequency component
        \begin{align} \label{e:lowfreq}
                \frac{1}{2\pi \ri} \int_{\Gamma_{c,\delta}} \re^{\lambda t} (\lambda - \Lps)^{-1} \g \, \de \lambda,
        \end{align}
        where $\Gamma_{c,\delta}$ lies in the disk $B(0,\delta)$. Thus, we may apply the resolvent decomposition~\eqref{e: final resolvent decomp} to split~\eqref{e:lowfreq} into components $s_{p,1/2}(t)\g$ associated with the leading- and higher-order dynamics excited by the translational mode, components $s_{c,1/2}(t)\g$ corresponding to the leading- and higher-order scattering dynamics of the outgoing diffusive modes in the wake, and an exponentially decaying remainder. This procedure ultimately leads to the proof of Theorem~\ref{t: linear estimates}.
        
        \medskip

        \noindent \textbf{Splitting off the high-frequency component.} It was shown in the linear stability analysis of pulled pattern-forming fronts in the FitzHugh--Nagumo system in~\cite[Corollary~4.4]{FHNpulled} that the integration contour in the inverse Laplace representation may be shifted into the left-half plane except for a piece lying in the disk $B(0,\delta)$. The proof uses a rescaling argument to analyze the high-frequency behavior of the resolvent for $|\Im(\lambda)| \gg 1$. This argument depends only on the basic structure of the original system, and in particular is independent of the profile about which we are linearizing, and so applies without modification in our setting to give the following result. 

        \begin{prop}[High-frequency damping]\label{p: contour shifting}
        Fix $1 \leq p \leq \infty$. For any $\delta > 0$, there exist a constant $\mu > 0$ and a contour $\Gamma_\delta \subset \C$, which is symmetric in the real axis, such that we have the representation 
            \begin{align*}
                \re^{\Lps t} \g = \frac{1}{2\pi \ri} \int_{\Gamma_\delta} \re^{\lambda t} (\lambda -\Lps)^{-1} \g \, \de \lambda 
            \end{align*}
            for all $\g \in C_c^\infty(\R)$ and $t > 0$, where the integral is interpreted in the principal value sense. Moreover, $\Gamma_\delta$ lies strictly to the right of $\Sigma(\Lps)$ and we have $\Gamma_\delta \cap \partial B(0,\delta) = \{\lambda_\delta,\overline{\lambda_\delta}\}$, where $\lambda_\delta \in \C$ satisfies $\Re \lambda_\delta < 0 < \Im \lambda_\delta$. Finally, $\Gamma_\delta \setminus B(0,\delta)$ lies in the open left-half plane and we have the estimate 
            \begin{align}
                \left\| \int_{\Gamma_\delta \setminus B(0,\delta)} \re^{\lambda t} (\lambda-\Lps)^{-1} \g \, \de \lambda \right\|_{L^p} \lesssim \re^{-\mu t} \| \g \|_{L^p} 
            \end{align}
            for all $\g \in C^\infty_c (\R)$ and $t > 0$.  
        \end{prop}

        \medskip

        \noindent \textbf{Pointwise estimates.} The main remaining task is to use the resolvent decomposition of Section~\ref{s: final resolvent decomposition} to study the integral~\eqref{e:lowfreq} with $\Gamma_{c,\delta} := \Gamma_\delta \cap B(0,\delta)$, where $\Gamma_\delta$ is as in Proposition~\ref{p: contour shifting}. To estimate these low-frequency contributions, we will rely on pointwise estimates, originally developed in the stability theory of viscous shock waves~\cite{ZumbrunHoward}. These estimates translate information about the (pointwise) meromorphic continuation of the resolvent across the essential spectrum into estimates on the spatio-temporal dynamics of the linearized evolution. Our approach differs from the traditional pointwise Green's function approach in that we obtain the pointwise description of the resolvent through a far-field/core decomposition, while the traditional approach would construct a Green's kernel for the resolvent problem by gluing together exponential dichotomies. 

        We will use the following formulation of pointwise estimates from~\cite{FHNpulled}. A detailed proof, based on ideas from~\cite{ZumbrunHoward, MasciaZumbrun}, can be found in~\cite[Appendix~D]{FHNpulled}.

        \begin{prop}[Pointwise estimates]\label{p: pointwise estimates}
        Fix $\delta_1,\Delta \mu > 0$ and $j,\ell,m \in \N_0$. Let $g \colon \R^2 \times B(0, \delta_1) \to \C^2$ be a function such that
        \begin{align*}
        		B(0, \delta_1) \to \C, \qquad \lambda \mapsto g(\xi, \zeta; \lambda)
        \end{align*}
        is analytic for each $\xi,\zeta \in \R$. Suppose that there exists $C_0 > 0$ such that $\| g (\cdot, \cdot; \lambda) \|_{L^\infty} \leq C_0 |\lambda|^m$ for $\lambda \in B(0,\delta_1)$. Then, for any $\delta \in (0,\delta_1)$ sufficiently small, there exist a function $G^{j, \ell, m} \colon [0, \infty) \times \R\to \R$ and constants $C, D_0, \mu> 0$ such that, for each $(\xi, \zeta, t) \in \R^2 \times [0, \infty)$, there exists a contour $\Gamma_{\xi, \zeta, t,\delta}$, which lies in $B(0, \delta)$ and connects $\overline{\lambda_\delta}$ to $\lambda_\delta$ where $\lambda_\delta \in \C$ is as in Proposition~\ref{p: contour shifting}, such that we have the estimate
        	\begin{align}
        		\int_{\Gamma_{\xi, \zeta, t,\delta}} | \lambda|^j |\nuwt(\lambda)|^\ell \re^{\Re (\lambda t + \nuwt(\lambda) (\xi - \zeta))} | g(\xi, \zeta; \lambda)| \chi_- (\xi - \zeta) \, | \de \lambda | \leq G^{j, \ell, m} (t, \xi - \zeta),
        	\end{align}
        	where
            \begin{align} \begin{split}
                |G^{j, \ell, m} (t, \xi)| &\leq \frac{C}{(1+t)^{\frac{1}{2} + \frac{\ell + j + m}{2}}} \chi_-(\xi) \chi_+(\xi - (c_g - \Delta \mu) t) \chi_- (\xi - (c_g + \Delta \mu) t + 1) \re^{-\frac{(\xi - c_g t)^2}{D_0t} } \\ &\qquad \, + C\,\chi_-(\xi) \re^{-\mu t}  \re^{-\mu |\xi|} \end{split}\label{e: pointwise G estimate refined}
            \end{align}
            for all $\xi \in \R$ and $t > 0$. 
        \end{prop}
\begin{proof}
    This follows directly by applying~\cite[Lemma~D.1 through~D.4]{FHNpulled}, thereby using that
    \begin{align*}
    \re^{-\frac{(\xi - c_g t)^2}{Mt}} \leq \re^{-\frac{(\Delta \mu)^2 t}{M} + 2\frac{\Delta \mu}{M}}
    \end{align*}
    for $\xi \in \R$ and $M,t > 0$ with $|\xi - c_g t + 1| \geq t \Delta \mu$, and the fact that we have $\nuwt'(0) = -c_g^{-1}$ by~\eqref{e:nuwt_expansion}.
\end{proof}
        
The estimate~\eqref{e: pointwise G estimate refined} encodes a Gaussian wave packet which propagates to the left with speed $|c_g|$ and decays diffusively. In a weighted norm whose weight decays exponentially on the left, this leftward transport induces exponential decay in time. 
        \begin{lemma}\label{l: exp weighted pointwise bound}
            Fix $\kappa, \Delta \mu > 0$ with $c_g + \Delta \mu < 0$. Let $G \colon [0,\infty) \times \R \to \R$ be any function satisfying the pointwise bound
            \begin{align}
                |G(t, \xi)| \lesssim \chi_- (\xi - (c_g + \Delta \mu) t + 1). \label{e: G inequality}
            \end{align}
            for $\xi \in \R$ and $t \geq 0$. Then, there we have the temporal decay estimate
            \begin{align*}
                \| \omega_{\kappa, 0} G(t, \cdot) \|_{L^\infty} \lesssim \re^{\kappa (c_g + \Delta \mu) t}
            \end{align*}
            for all $t \geq 0$.
        \end{lemma}
        \begin{proof}
           On the support of the right hand side of~\eqref{e: G inequality}, we have $\omega_{\kappa, 0} (\xi) = \re^{\kappa \xi}$, and $\xi \leq (c_g + \Delta \mu) t - 1$, from which the result follows. 
        \end{proof}

        \medskip

        \noindent\textbf{Semigroup decomposition.} By Proposition~\ref{p: contour shifting}, we can write the linearized evolution as 
        \begin{align}
            \re^{\Lps t} \g = \frac{1}{2\pi \ri} \int_{\Gamma_\delta} \re^{\lambda t} (\lambda - \Lps)^{-1} \g \, \de \lambda 
        \end{align}
        for $\g \in C_c^\infty(\R)$, where the contour $\Gamma_\delta$ is depicted in Figure~\ref{fig: contour shifting}, and the integral is interpreted in the principal value sense. To aid in separating small- and large-time behavior, we let $\tilde{\chi} \colon [0, \infty) \to [0,1]$ be a smooth temporal cutoff function satisfying $\tilde{\chi}(t) = 0$ for $0 \leq t \leq 1$ and $\tilde{\chi}(t) = 1$ for $t \geq 2$. We then define the exponentially damped part of the semigroup as 
        \begin{align}
        \begin{split}
            s_e(t) \g &= \frac{\tilde{\chi}(t)}{2\pi \ri} \int_{\Gamma_\delta \setminus B(0,\delta)} \re^{\lambda t} (\lambda-\Lps)^{-1} \g \, \de \lambda + (1 - \tilde{\chi}(t))\re^{\Lps t} \g +\frac{\tilde{\chi}(t)}{2\pi \ri} \int_{\Gamma_{c,\delta}} \re^{\lambda t} \bar{s}_e (\lambda) \g \, \de \lambda\end{split}\label{e: s e t}
        \end{align}
        for $\g \in C_c^\infty(\R)$, where we recall $\Gamma_{c,\delta} := \Gamma_\delta \cap B(0,\delta)$. The term $s_e(t)$ collects all exponentially decaying contributions: those arising from the stable portion $\Gamma_\delta \setminus B(0,\delta)$ of the contour, and the contribution of the remainder $\bar{s}_e(\lambda)$ in the resolvent decomposition~\eqref{e: final resolvent decomp} along the segment $\Gamma_{c,\delta}$.
        
        The remaining terms in our semigroup decomposition, which capture the slowest decaying large-time behavior, are then defined in accordance with the resolvent decomposition~\eqref{e: final resolvent decomp} as 
        \begin{align}
        \label{e:defspsc}
        \begin{split}
            s_{p,j}(t) \g =  \frac{\tilde{\chi}(t)}{2\pi \ri} \int_{\Gamma_{c,\delta}} \re^{\lambda t} \bar{s}_{p, j}(\lambda) \g \, \de \lambda, \qquad
            s_{c,j}(t) \g =  \frac{\tilde{\chi}(t)}{2\pi \ri} \int_{\Gamma_{c,\delta}} \re^{\lambda t} \bar{s}_{c, j}(\lambda) \g \, \de \lambda
            \end{split}
        \end{align}
        for $j = 1,2$ and $\g \in C_c^\infty(\R)$, where $\bar{s}_{p/c,j}(\lambda)$ are given by \eqref{e: sp1 resolvent def}-\eqref{e: sc2 def}. We may then decompose $\re^{\Lps t}$ as 
        \begin{align} \label{e:semgr_decomp}
            \re^{\Lps t} = \omega_0 \Ups' s_p(t) + s_c(t) + s_e(t) 
        \end{align}
        with
        \begin{align*}
        s_{p}(t) = s_{p,1}(t) + s_{p,2}(t), \qquad s_{c}(t) = s_{c,1}(t) + s_{c,2}(t), 
        \end{align*}
        where the propagators $s_p(t)$ and $s_c(t)$ vanish identically for $t \in [0,1]$. 

        \subsection{Proof of Theorem~\ref{t: linear estimates}}
        
        We now complete the proof of Theorem~\ref{t: linear estimates} by establishing sharp $L^2$- and $L^\infty$-bounds on each term in the decomposition~\eqref{e:semgr_decomp}. By density of $C_c^\infty(\R)$, these bounds imply that the operators $s_{p,j}(t)$, $s_{c,j}(t)$, and $s_e(t)$ can be extended to bounded linear operators with domain $L^2(\R)$ or $C_0(\R)$ for $j = 1,2$.
        
        We start with bounding the exponentially decaying term $s_e(t)$. 

        \begin{lemma}[Estimates on exponentially damped terms]\label{l: estimates on exponentially damped terms}
        For any sufficiently small $\delta > 0$, there exists $\mu > 0$ such that the term $s_e(t)$, given by~\eqref{e: s e t}, enjoys the bound            \begin{align} \label{e: set 1}
                \| s_e(t) \g \|_{L^\infty} &\lesssim \re^{-\mu t} \| \g \|_{L^\infty}, \qquad \| s_e(t) \u \|_{L^2} \lesssim \re^{-\mu t} \| \u \|_{L^2},
            \end{align}
        for $\g \in C_0(\R)$, $\u \in L^2(\R)$, and $t \geq 0$. Moreover, for any  sufficiently small $\kappa >0$, we have 
        \begin{align} \label{e: set 2}
        \| \omega_{\kappa, 0} s_e(t) \g \|_{L^\infty} &\lesssim \re^{-\mu t} \| \omega_{\kappa, 0} \g \|_{L^\infty} 
            \end{align}
        for $\g \in C_0(\R)$ and $t \geq 0$. 
        \end{lemma}
        \begin{proof}
            We first prove~\eqref{e: set 1}. The corresponding estimate on the first term in~\eqref{e: s e t} was established in Proposition~\ref{p: contour shifting}. The estimate on the second term in~\eqref{e: s e t} follows immediately from compact support in time of $\tilde{\chi}(t)$ and standard semigroup theory, cf.~\cite[Lemma~I.5.5]{EngelNagel}. For the estimate on the last term in~\eqref{e: s e t}, we note that $\bar{s}_e \colon B(0,\delta) \to \mathcal{B}(L^p(\R))$ is analytic by Lemma~\ref{lem: summary wave train resolvent}, This can be used to shift the integration contour $\Gamma_{c,\delta}$ into the open left-half plane, yielding exponential decay in time. This concludes the proof of~\eqref{e: set 1}. 
            
            To prove~\eqref{e: set 2}, we first observe that, provided $\kappa > 0$ is sufficiently small, the proof of Lemma~\ref{l: left resolvent exponentially localized estimates} can be used to extract analyticity in $\lambda$ of $\smash{\omega_{\kappa, 0} \bar{s}_e(\lambda) (\, \cdot\ \omega_{\kappa, 0}^{-1} )}$, in a neighborhood of $0$ whose size is independent of $\kappa$, which implies the desired estimate for the last term~\eqref{e: s e t}. 
            
            To prove the last estimate for the first two terms in~\eqref{e: s e t}, one conjugates with the weight $\omega_{\kappa, 0}$, and then repeats the proof of Proposition~\ref{p: contour shifting} from~\cite[Appendix~A]{FHNpulled}. The key is that the weight $\omega_{\kappa, 0}$ is non-decreasing. As a result, the additional zeroth order terms introduced to the bottom-right block of this operator by the conjugation are non-positive, and so can only contribute to additional damping. See~\cite[Appendix~A]{FHNpulled} for further details, with $\omega_{\kappa, 0}$ here playing exactly the same role as $\omega$ in that proof. The choice of $\mu$ is then limited only by the size of the spectral gap as $|\Im \lambda| \to \infty$. In particular, the analysis in~\cite[Appendix~A]{FHNpulled} shows that we can choose $\mu = \frac{\eps \gamma}{2}$, independent of $\kappa$. 
        \end{proof}

        We proceed with estimating the residual component $s_c(t)$ in~\eqref{e:semgr_decomp}, which exhibits algebraic decay in time. 

        \begin{lemma}[Estimates on residual excited terms $s_{c,1}(t)$]\label{l: residual excited terms} Fix $\Delta \mu > 0$ with $c_g + \Delta \mu < 0$. For any sufficiently small $\delta > 0$,  the term $s_{c,1}(t)$, given by~\eqref{e:defspsc}, satisfies, for any sufficiently small $\kappa > 0$, the following estimates 
        \begin{align*}
			\| s_{c,1}(t) \g \|_{L^\infty} &\lesssim (1+t)^{-1/2} \| \omega_{\kappa, 0} \g \|_{L^\infty}, \\
			\| s_{c,1}(t) \g \|_{L^2} &\lesssim (1+t)^{-1/4} \| \omega_{\kappa, 0} \g\|_{L^\infty}, \\
			\| \omega_{\kappa, 0} s_{c,1} (t) \g \|_{L^\infty} &\lesssim \re^{\kappa (c_g + \Delta \mu) t} \| \omega_{\kappa, 0} \g \|_{L^\infty} 
		\end{align*}
        for $\g \in C_0(\R)$ and $t \geq 0$.
        \end{lemma}
        \begin{proof}
        Recall the definition~\eqref{e: sc1bar def} of $\bar{s}_{c,1}(\lambda)\g$. Applying Lemma~\ref{l: control of center data} to bound $\tg(\lambda) = \g - (\lambda - \Lps) (\chi_- \u_-(\lambda))$, and using that $\q \colon B(0,\delta) \to L^\infty(\R)$ and $\tilde{\alpha}_- \colon B(0,\delta) \to \mathcal{B}(X_\eta^p,\C)$ are analytic by Lemma~\ref{l: left homogeneous solution} and Proposition~\ref{p: ff core inversion}, respectively, we readily obtain, provided $\delta > 0$ is sufficiently small, that there exists $\kappa_0 > 0$ such that            \begin{align}
                | [\bar{s}_{c,1}(\lambda) \g] (\xi)| \lesssim \chi_-(\xi) \big| \re^{\nuwt(\lambda) \xi} \big| \| \omega_{\kappa_0, 0} \g \|_{L^\infty}
            \end{align}
        for $\xi \in \R$, $\lambda \in B(0,\delta)$, and $\g \in C_0(\R)$. Hence, using that $[\bar{s}_{c,1}(\cdot)\g](\xi) \colon B(0,\delta) \to \C^2$ is analytic for each $\xi \in\R$ by Lemmas~\ref{lem: summary wave train resolvent} and~\ref{l: left homogeneous solution} and Proposition~\ref{p: ff core inversion}, and applying Proposition~\ref{p: pointwise estimates} with $\zeta = 0$ yields, provided $\delta > 0$ is sufficiently small, constants $D_0,\mu > 0$ such that the pointwise bound
            \begin{align} \label{e: sc1 pointwise}
            \begin{split}
                |[s_{c,1}(t) \g](\xi)| &\lesssim \chi_- (\xi) \left( \frac{\re^{-\frac{(\xi-c_gt)^2}{D_0 t}}}{\sqrt{1+t}} \chi_+(\xi - (c_g - \Delta \mu) t) \chi_- (\xi - (c_g + \Delta \mu) t + 1)  +  \re^{-\mu t} \re^{-\mu |\xi|} \right)\\ 
                &\qquad \cdot \| \omega_{\kappa_0, 0} \g \|_{L^\infty}
            \end{split}
            \end{align}
            holds for $\xi \in \R$, $\g \in C_0(\R)$, and $t \geq 0$, 
            which readily implies the first two desired estimates. The final estimate also follows from this pointwise bound by applying Lemma~\ref{l: exp weighted pointwise bound} and taking $\kappa \in (0,\kappa_0]$ so small that $-\mu \leq \kappa (c_g + \Delta \mu)$. 
        \end{proof}

        \begin{lemma}[Estimates on residual scattering terms $s_{c,2}(t)$]\label{l: residual scattering terms}
            Fix $\Delta \mu > 0$ with $c_g + \Delta \mu < 0$. For any sufficiently small $\delta > 0$, the term $s_{c,2}(t)$, given by~\eqref{e:defspsc}, satisfies 
            \begin{align*}
    			\| s_{c,2} (t) \g \|_{L^\infty} &\lesssim (1+t)^{-3/4} \| \g \|_{L^2}, \qquad
    			\| s_{c,2} (t) \g \|_{L^2}  \lesssim (1+t)^{-1/2} \| \g \|_{L^2}
            \end{align*}
            for $\g \in L^2(\R)$ and $t \geq 0$. Moreover, for any sufficiently small $\kappa > 0$, we have
            \begin{align*}
    			\| \omega_{\kappa, 0} s_{c,2} (t) \g \|_{L^\infty} &\lesssim \re^{\kappa (c_g + \Delta \mu) t} \| \omega_{\kappa, 0} \g \|_{L^\infty} 
    		\end{align*}
            for $\g \in C_0(\R)$ and $t \geq 0$. 
        \end{lemma}
        \begin{proof}
            The proof is similar to that of Lemma~\ref{l: residual excited terms}, but in the first two estimates we gain improved temporal decay thanks to the extra $\mathrm{O}(\lambda)$-factor in $\bar{s}_{c,2}(\lambda)$.
        \end{proof}

        Next, we establish bounds on the principal component $s_p(t)$ in~\eqref{e:semgr_decomp}.

        \begin{lemma}[Estimates on leading-order scattering terms $s_{p,2}(t)$]\label{l: leading order scattering}
           Fix $\Delta \mu > 0$ with $c_g + \Delta \mu < 0$. Let $j,k\in \N_0$. For any sufficiently small $\delta > 0$, the term $s_{p,2}(t)$, given by~\eqref{e:defspsc}, satisfies 
            \begin{align*}
               \| s_{p,2} (t) \g \|_{L^2} &\lesssim \| \g \|_{L^2}, \qquad
    			\| s_{p,2} (t) \g \|_{L^\infty} \lesssim (1+t)^{-1/4} \| \g \|_{L^2}
            \end{align*}
            for $\g \in L^2(\R)$ and $t \geq 0$. Moreover, for any sufficiently small $\kappa > 0$, we have
            \begin{align*}
                \| \omega_{\kappa, 0} \partial_\xi^j \partial_t^k s_{p,2}(t) \g \|_{L^\infty} &\lesssim \re^{\kappa (c_g + \Delta \mu) t} \| \omega_{\kappa, 0} \g \|_{L^\infty}
            \end{align*}
            for $\g \in C_0(\R)$ and $t \geq 0$. Finally, if $j+k \geq 1$, then it holds
            \begin{align*}
    			\| \partial_\xi^j \partial_t^k s_{p,2} (t) \g \|_{L^\infty} &\lesssim (1+t)^{-3/4} \| \g \|_{L^2}, \qquad
    			\| \partial_\xi^j \partial_t^k s_{p,2} (t) \g \|_{L^2} \lesssim (1+t)^{-1/2} \| \g \|_{L^2}
      		\end{align*}
            for $\g \in L^2(\R)$ and $t \geq 0$. 
        \end{lemma}
        \begin{proof}
            The proof of the first three estimates follows exactly as in Lemma~\ref{l: residual excited terms}. The last two estimates are similar, except the derivatives introduce an extra factor of $\lambda$ into the resolvent, which translates to improved temporal decay by Proposition~\ref{p: pointwise estimates}. 
        \end{proof}

        \begin{lemma}[Estimates on the leading-order excited terms $s_{p,1}(t)$]\label{l: sp1 estimates}
            Fix $\Delta \mu > 0$ with $c_g + \Delta \mu < 0$. Let $j, k \in \N_0$. For any sufficiently small $\delta > 0$, there exist constants $D_0,\mu > 0$ such that, for any sufficiently small $\kappa > 0$, the term $s_{p,1}(t)$, given by~\eqref{e:defspsc}, satisfies the pointwise bound
            \begin{align} \label{e:sp1_pointwise}
                \big|[\partial_\xi^j \partial_t^k s_{p,1}(t) \g](\xi)\big| \lesssim \left( \left| \partial_\xi^j \partial_t^k \erf \left( \frac{\xi - c_g t}{\sqrt{D_0 (1+t)}} \right) \right| + \re^{-\mu t} \re^{-\mu |\xi|}\right) \| \omega_{\kappa, 0} \g\|_{L^\infty}
            \end{align} 
            for $\xi \in \R$, $t \geq 0$, and $\g \in C_0(\R)$, 
            where $\erf(z) = \int_{-\infty}^z \re^{-w^2} \, \de w$ is the error function. Moreover, if $j + k \geq 1$, we have the estimates
			\begin{align}
            \begin{split}
                |\Ptr \g| &\lesssim \| \omega_{\kappa, 0} \g \|_{L^\infty},  \\
				\| s_{p,1} (t) \g\|_{L^\infty} &\lesssim \| \omega_{\kappa, 0} \g \|_{L^\infty},  \\
				\| \partial_\xi^j \partial_t^k s_{p,1} (t) \g \|_{L^\infty} &\lesssim (1+t)^{-1/2} \| \omega_{\kappa, 0} \g \|_{L^\infty},  \\
				\| \partial_\xi^j \partial_t^k s_{p,1} (t) \g \|_{L^2} &\lesssim (1+t)^{-1/4} \| \omega_{\kappa, 0} \g \|_{L^\infty},  \\
				\| \omega_{\kappa, 0} \partial_\xi^j \partial_t^k s_{p,1}(t) \g \|_{L^\infty} &\lesssim \re^{ \kappa ( c_g + \Delta \mu) t} \| \omega_{\kappa, 0} \g \|_{L^\infty}, \\
                \| \chi_- s_{p,1}(t) \g \|_{L^2} &\lesssim (1+t)^{1/2} \| \omega_{\kappa, 0} \g \|_{L^\infty}
                \end{split}\label{e: sp1 norm estimates}
			\end{align}
            for $\g \in C_0(\R)$ and $t \geq 0$.
        \end{lemma}
        \begin{proof}
            Since the adjoint eigenfunction $\psiad$ is exponentially localized  by Proposition~\ref{prop: fredholm properties}, there exists $\kappa_0 > 0$ such that 
            \begin{align} \label{e: ptr estimate}
            |\Ptr \g| = |\langle \g,\psiad\rangle_{L^2}| \lesssim \|\omega_{\kappa_0,0} \g\|_{L^\infty}
            \end{align} 
            for $\g \in C_0(\R)$, which establishes the first estimate in~\eqref{e: sp1 norm estimates}.
            
            We proceed with proving~\eqref{e:sp1_pointwise} for the case $j = k = 0$. Using the formula~\eqref{e: sp1 resolvent def} for $\bar{s}_{p,1}(\lambda)$, we have 
            \begin{align} \label{e: sp1 rewrite}
            \begin{split}
                s_{p,1} (t)\g &=  \frac{\tilde{\chi}(t)}{2\pi \ri} \int_{\Gamma_{c,\delta}} \re^{\lambda t} \bar{s}_{p,1}(\lambda) \g \, \de \lambda \\
                &= [\Ptr \g] \frac{\tilde{\chi}(t)}{2\pi \ri} \int_{\Gamma_{c,\delta}} \frac{\re^{\lambda t}}{\lambda} \left[ \chi_- \re^{\nuwt(\lambda) \cdot} + (1-\chi_-) \right] \, \de \lambda \\
                &= [\Ptr \g] \frac{\tilde{\chi}(t)}{2\pi \ri} \left(  \int_{\Gamma_{c,\delta}} \frac{\re^{\lambda t}}{\lambda} \chi_- \re^{\nuwt(\lambda) \cdot} \, \de \lambda + \int_{\Gamma_{c,\delta}} \frac{\re^{\lambda t}}{\lambda} (1-\chi_-) \, \de \lambda \right)
                \end{split}
            \end{align}
            for $\g \in C_c^\infty(\R)$ and $t \geq 0$. Since $\Gamma_{c,\delta}$ lies to the right of $\Sigma(\Lps) \supset \Sigma(\lwt)$, we have $\Re \nuwt(\lambda) > 0$ for $\lambda \in \Gamma_{c,\delta}
            $ by~\eqref{e:nuwt_expansion} and~\eqref{e:spectral_id}. Therefore, it holds
            \begin{align*}\frac{1}{\lambda} \re^{\nuwt(\lambda) \xi} = \frac{\nuwt(\lambda)}{\lambda} \int_{-\infty}^\xi \re^{\nuwt(\lambda) \zeta} \, \de \zeta
            \end{align*}
            for $\lambda \in \Gamma_{c,\delta}$ and $\xi \leq 0$. Thus, swapping the order of integration, we may rewrite the first integral on the right-hand side of~\eqref{e: sp1 rewrite} as 
            \begin{align*}
                \int_{\Gamma_{c,\delta}} \frac{\re^{\lambda t}}{\lambda} \chi_-(\xi) \re^{\nuwt(\lambda) \xi} \, \de \lambda &= \int_{-\infty}^\xi \chi_-(\xi) \int_{\Gamma_{c,\delta}} \frac{\nuwt(\lambda) \re^{\nuwt(\lambda)}}{\lambda} \chi_-(\zeta-1) \re^{\lambda t + \nuwt(\lambda)(\zeta-1)} \, \de \lambda \, \de \zeta 
            \end{align*}
            for $\xi \in \R$ and $t \geq 0$, 
            where we have used that $\chi_-(\zeta-1) = 1$ for $\zeta \leq \xi \leq 0$. Applying the pointwise estimates of Proposition~\ref{p: pointwise estimates}, we therefore obtain, provided $\delta > 0$ is sufficiently small, constants $D_0,\mu > 0$ such that
            \begin{align*}
                \left| \int_{\Gamma_{c,\delta}} \frac{\re^{\lambda t}}{\lambda} \chi_-(\xi) \re^{\nuwt(\lambda) \xi} \, \de \lambda \right | \lesssim \int_{-\infty}^\xi G(t, \zeta-1) \, \de \zeta
            \end{align*}
            for $\xi \in \R$ and $t \geq 1$, where we denote
            \begin{align*}
            G(t, \zeta) &= \frac{\chi_-(\zeta)}{\sqrt{1+t}} \re^{-\frac{(\zeta - c_g t)^2}{D_0(1+t)}} + \chi_-(\zeta)\re^{-\mu t}  \re^{-\mu |\zeta|}. 
            \end{align*}
            Integrating the Gaussian term in $G(t, \zeta)$ leads to the error function in the pointwise estimate~\eqref{e:sp1_pointwise}, while the exponentially small corrections in $G(t,\zeta)$ again propagate to exponentially small corrections here. We conclude that the first term on the right-hand side of~\eqref{e: sp1 rewrite} can be bounded by the right-hand side of~\eqref{e:sp1_pointwise}.
            
            On the other hand, adding a line segment to close the contour $\Gamma_{c,\delta}$ and using the residue theorem, we find $\mu > 0$ such that
            \begin{align} \label{e: residue}
                \left|\frac{1}{2 \pi \ri}\int_{\Gamma_{c,\delta}} \frac{\re^{\lambda t}}{\lambda} \, \de \lambda - 1\right| \lesssim \re^{-\mu t} 
            \end{align}
            for $t \geq 0$. Hence, using that $1-\chi_-$ is supported on $[-1,\infty)$ and the leftward propagating error function is uniformly bounded from below for $\xi \geq 0$ and $t \geq 0$, the pointwise bound on the second integral on the right-hand side of~\eqref{e: sp1 rewrite} may be absorbed into the error function in~\eqref{e:sp1_pointwise}. All in all, we have established~\eqref{e:sp1_pointwise} for $j = k = 0$.
            
            Now let us prove~\eqref{e:sp1_pointwise} for the case $j = 1$ and $k = 0$. We compute
            \begin{align*}
                \partial_\xi \bar{s}_{p,1}(\lambda) \g = \frac{\Ptr \tg(0)}{\lambda} \left[ \nuwt(\lambda) \chi_- \re^{\nuwt(\lambda)\xi} + \chi_-' \big(\re^{\nuwt(\lambda) \xi} - 1\big) \right].
            \end{align*}
            Note that the singularity in $\lambda$ has been removed in both terms, since $\nuwt(0) = 0$ by~\eqref{e:nuwt_expansion}. After inverse Laplace transform, the first term then corresponds to the fundamental solution of the advection-diffusion equation, and the second term is exponentially localized in space and time since the compact support of $\chi_-'$ allows shifting the contour into the left-half plane. The proofs of the pointwise estimate~\eqref{e:sp1_pointwise} for other $j,k \in \N_0$ with $j + k \geq 1$ are analogous. 
            
            Finally, the remaining estimates in~\eqref{e: sp1 norm estimates} immediately follow from the pointwise bound~\eqref{e:sp1_pointwise} as in the proof of Lemma~\ref{l: residual excited terms}. 
        \end{proof}

        Finally, we determine the leading-order asymptotics of the principal component $s_p(t)$ in~\eqref{e:semgr_decomp}.

        \begin{lemma}[Asymptotics for leading-order excited terms $s_{p,1}(t)$]\label{l: asymptotics for leading order terms}
        Fix $\Delta \mu > 0$ with $c_g + \Delta \mu < 0$.  For any sufficiently small $\delta > 0$, the term $s_{p,1}(t)$, given by~\eqref{e:defspsc}, obeys, for any sufficiently small $\kappa > 0$, the following approximation
            \begin{align*}
    			\| \omega_{\kappa,0} [s_{p,1} (t) - \Ptr] \g \|_{L^\infty} &\lesssim \re^{ \kappa (c_g + \Delta \mu) t} \| \omega_{\kappa, 0} \g \|_{L^\infty}
    		\end{align*}
            for $\g \in C_0(\R)$ and $t \geq 0$.
        \end{lemma}
        \begin{proof}
            For the first estimate, we rewrite $[s_{p,1}(t) - \Ptr] \g$ as 
            \begin{align} \label{e:spt1 approx}
            \begin{split}
                [s_{p,1}(t) - \Ptr] \g &= \Bigg( \frac{\tilde{\chi}(t)}{2\pi \ri}\int_{\Gamma_{c,\delta}} \frac{\re^{\lambda t}}{\lambda} \chi_- \big(\re^{\nuwt(\lambda) \cdot} -1\big) \, \de \lambda +\frac{\tilde{\chi}(t) -1}{2\pi \ri} \int_{\Gamma_{c,\delta}} \frac{\re^{\lambda t}}{\lambda} \, \de \lambda\\ 
                &\qquad + \, \frac{1}{2\pi \ri} \int_{\Gamma_{c,\delta}} \frac{\re^{\lambda t}}{\lambda} \, \de \lambda - 1  \Bigg)[\Ptr \g]
            \end{split}
            \end{align}
            for $\g \in C_0(\R)$ and $t \geq 0$. We focus on the first term on the right-hand side of~\eqref{e:spt1 approx}. We rewrite the integrand, and then swap the order of integration to obtain
            \begin{align*}
                \int_{\Gamma_{c,\delta}} \frac{\re^{\lambda t}}{\lambda} \chi_-(\xi) \big(\re^{\nuwt(\lambda) \xi} -1\big) \, \de \lambda &= - \chi_- (\xi) \int_{\Gamma_{c,\delta}} \int_\xi^0 \frac{\nuwt(\lambda)}{\lambda} \re^{\lambda t + \nuwt(\lambda) \zeta} \, \de \zeta \de \lambda \\
                &= - \chi_- (\xi)\int_\xi^0 \int_{\Gamma_{c,\delta}} \frac{\nuwt(\lambda)\re^{\nuwt(\lambda)} }{\lambda} \chi_-(\zeta-1) \re^{\lambda t + \nuwt(\lambda) (\zeta-1)} \, \de \lambda \de \zeta
            \end{align*}
            for $\xi \in \R$, $\lambda \in \Gamma_{c,\delta}$, and $t \geq 0$, 
            where we have used that $\chi_-(\zeta -1) = 1$ for all $\zeta \leq 0$. Applying the pointwise estimate of Proposition~\ref{p: pointwise estimates}, we then obtain, provided $\delta > 0$ is sufficiently small, constants $D_0,\mu > 0$ such that
            \begin{align} \label{e: first spt1 approx bound}
                \left|\int_{\Gamma_{c,\delta}} \frac{\re^{\lambda t}}{\lambda} \chi_- \big(\re^{\nuwt(\lambda) \xi} -1\big) \, \de \lambda \right| \leq \chi_- (\xi)\int_\xi^0 G(t, \zeta-1) \, \de \zeta
            \end{align}
            for $\xi \in \R$ and $t \geq 1$, where we denote
            \begin{align*}
                G(t, z) = \chi_-(z) \chi_+ (z - (c_g - \Delta \mu) t) \chi_- (z- (c_g + \Delta \mu) t+1) \frac{\re^{-\frac{(z-c_gt)^2}{D_0t}}}{\sqrt{1+t}} + \chi_- (z) \re^{-\mu t} \re^{-\mu |z|}.
            \end{align*}
            Using Lemma~\ref{l: exp weighted pointwise bound} and taking $\kappa > 0$ so small that  $-\mu \leq \kappa (c_g + \Delta \mu)$, we find that, $\|\omega_{\kappa,0}G(\cdot,t)\|_{L^1} \lesssim \re^{\kappa (c_g + \Delta \mu) t}$ for $t \geq 0$. Combining the latter with~\eqref{e: ptr estimate},~\eqref{e: residue},~\eqref{e:spt1 approx},~\eqref{e: first spt1 approx bound}, and the fact that $\tilde{\chi} - 1$ is supported on $[0,2]$, we obtain the desired estimate.  
        \end{proof}

        Together, Lemmas~\ref{l: estimates on exponentially damped terms} through~\ref{l: asymptotics for leading order terms} establish Theorem~\ref{t: linear estimates}. In the upcoming Sections~\ref{sec: nonlinear iteration} and~\ref{sec:proof}, we will see that these estimates are sufficient for the proof of the estimates~\eqref{e:maindecaybounds}-\eqref{e:maindecaybounds3} in Theorem~\ref{t: main}. 

        \subsection{Light cone estimates for refined convergence}
        To prove the refined estimates~\eqref{e:refineddecaybounds} in Theorem~\ref{t: main}, which give a more precise description of the dynamics in the spacetime regions $\xi \geq (c_g + \Delta c) t$ and $\xi \leq (c_g - \Delta c) t$, we will need the following estimates which characterize the linearized dynamics in these regions. 

        \begin{prop}[Right light cone estimates] \label{prop: right light cone}
        Fix $\Delta \theta, \Delta c > 0$ such that $\Delta \theta < \Delta c < -c_g$. Set $\tilde{c} = c_g + \Delta c$. Let $j,k \in \N_0$. Then, for any sufficiently small $\delta > 0$, there exist constants $\mu_r,\kappa_r > 0$ such that for all $\kappa \in (0,\kappa_r]$ the following estimates hold.
            \begin{enumerate}[(i)]
                \item \emph{(Leading-order excited terms).} If $j+k \geq 1$, then the term $s_{p,1}(t)$, given by~\eqref{e:defspsc}, obeys
                \begin{align}
                \begin{split}
                    \| \chi_+(\cdot - \tilde{c} t) [s_{p,1}(t-s) -\Ptr ]\g \|_{L^\infty}  &\lesssim \re^{-\kappa (\Delta c - \Delta \theta)(t-s) } \re^{- \kappa \tilde{c} s} \|\omega_{\kappa, 0} \g\|_{L^\infty},\\
                    \| \chi_+(\cdot - \tilde{c} t) \partial_t^j \partial_\xi^k s_{p,1}(t-s)\g \|_{L^\infty} &\lesssim \re^{-\kappa (\Delta c - \Delta \theta)(t-s) } \re^{- \kappa \tilde{c} s} \|\omega_{\kappa, 0} \g\|_{L^\infty}
                  \end{split} \label{e: refined linear 1}
                \end{align}
                for $\g \in C_0(\R)$, and $t, s \geq 0$ with $t \geq s$.
                \item \emph{(Leading-order scattering terms).} The term $s_{p,2}(t)$, given by~\eqref{e:defspsc}, satisfies 
                \begin{align}
                \begin{split}
                    \| \chi_+(\cdot - \tilde{c} t) \partial_t^j \partial_\xi^k s_{p,2}(t-s)\g \|_{L^\infty}  &\lesssim \re^{-\kappa (\Delta c - \Delta \theta)(t-s) } \re^{- \kappa \tilde{c} s} \|\omega_{\kappa, 0} \g\|_{L^\infty}\end{split} \label{e: refined linear 11}
                \end{align}
                for $\g \in C_0(\R)$ and $t, s \geq 0$ with $t \geq s$.
                \item \emph{(Residual terms).} The term $s_{c,i}(t)$, given by~\eqref{e:defspsc}, fulfills
                \begin{align}
                    \| \chi_+(\cdot - \tilde{c} t) s_{c, i}(t-s)\g \|_{L^\infty}  &\lesssim \re^{-\kappa (\Delta c - \Delta \theta)(t-s) } \re^{- \kappa \tilde{c} s} \|\omega_{\kappa, 0} \g\|_{L^\infty}, \qquad i = 1,2 \label{e: refined linear 2}
                \end{align}
                for $\g \in C_0(\R)$ and $t, s \geq 0$ with $t \geq s$.
                \item \emph{(Exponentially damped terms).} Finally, the term $s_e(t)$, given by~\eqref{e: s e t}, satisfies
                \begin{align}
                    \| \chi_+(\cdot - \tilde{c}t) s_e(t-s) \g \|_{L^\infty} \lesssim \re^{- \mu_r (t-s)} \re^{- \kappa \tilde{c} s} \| \omega_{\kappa, 0} \g \|_{L^\infty} \label{e: refined linear 3}
                \end{align}
                for $\g \in C_0(\R)$ and $t, s \geq 0$ with $t \geq s$.
            \end{enumerate}
        \end{prop}
        \begin{proof}
            We start by proving the first estimate in~\eqref{e: refined linear 1}. The proofs of the other estimates in~\eqref{e: refined linear 1},~\eqref{e: refined linear 11}, and~\eqref{e: refined linear 2} are completely analogous. First, we rewrite
            \begin{align*}
                \big| \chi_+(\cdot - \tilde{c} t) [s_{p,1}(t-s) -\Ptr]\g\big| = \left| \frac{\chi_+ (\cdot - \tilde{c} t)}{\omega_{\kappa, 0}} \right| | \omega_{\kappa, 0} [s_{p,1}(t-s) - \Ptr]\g| 
            \end{align*}
            for $\g \in C_0(\R)$ and $t,s \geq 0$ with $t\geq s $.  Since $\chi_+ (\xi - \tilde{c} t) \neq 0$ implies $\xi \geq \tilde{c} t - 1$, and $\omega_{\kappa, 0}$ is non-decreasing, we have $|\chi_+(\xi - \tilde{c} t)/\omega_{\kappa, 0} (\xi)| \lesssim \re^{-\kappa \tilde{c} t}$ for $\xi \in \R$ and $t \geq 0$. Hence, applying Lemma~\ref{l: asymptotics for leading order terms} with $\Delta \theta$ here playing the role of $\Delta \mu$, we establish, for any sufficiently small $\kappa > 0$, the estimate
            \begin{align*}
                \sup_{\xi \in \R} \big| \chi_+(\xi - \tilde{c} t) [s_{p,1}(t-s)\g -\Ptr\g](\xi)\big| \lesssim \re^{- \kappa \tilde{c} t} \re^{\kappa (c_g + \Delta \theta)(t-s)} \|\omega_{\kappa, 0} \g \|_{L^\infty}
            \end{align*}
            for $\g \in C_0(\R)$ and $t,s\geq 0$ with $t \geq s$. Writing out $\tilde{c} = c_g + \Delta c$ and rearranging, one arrives at the first estimate in~\eqref{e: refined linear 1}.
            
            Now we prove~\eqref{e: refined linear 3}. Proceeding similarly as before, we establish
            \begin{align}
                \big|\chi_+(\xi - \tilde{c} t) [s_e (t-s) \g](\xi) \big| = \left| \frac{\chi_+ (\xi - \tilde{c} t)}{\omega_{\kappa, 0} (\xi)} \right| \big| \omega_{\kappa, 0} (\xi) [s_e (t-s) \g](\xi) \big| \lesssim \re^{-\kappa \tilde{c} t}  | \omega_{\kappa, 0} (\xi) [s_e (t-s) \g](\xi)|. \label{e: refined linear damped proof 1}
            \end{align}
            for $\xi \in \R$, $u \in C_0(\R)$, and $t, s \geq 0$ with $t \geq s$. 
            On the other hand, applying Lemma~\ref{l: estimates on exponentially damped terms} yields a constant $\mu > 0$ such that, for any sufficiently small $\kappa > 0$, we have
            \begin{align*}
                \| \omega_{\kappa, 0} s_e(t-s) \g\| \lesssim \re^{-\mu (t-s)} \| \omega_{\kappa, 0} \g \|_{L^\infty} 
            \end{align*}
            for $\g \in C_0(\R)$ and $t,s \geq 0$ with $t \geq s$. Combining the latter estimate with~\eqref{e: refined linear damped proof 1}, we arrive, provided $\kappa > 0$ is sufficiently small, at~\eqref{e: refined linear 3}. 
        \end{proof}

        \begin{prop}[Left light cone estimates] \label{prop: left light cone}
            Fix $\Delta \theta, \Delta c > 0$ such that $\Delta \theta < \Delta c < -c_g$. Set $\undertilde{c} = c_g - \Delta c$. Let $j,k \in \N_0$. Then, for any sufficiently small $\delta > 0$, there exists a constant $\mu_l > 0$ such that the following estimates hold.
            \begin{enumerate}[(i)]
                \item \emph{(Excited terms).} The terms $s_{p,1}(t)$ and $s_{c,1}(t)$, given by~\eqref{e:defspsc}, obey
                \begin{align}
                \begin{split}
                    \|\chi_-(\cdot - \undertilde{c} t) \partial_t^j \partial_\xi^k s_{p,1}(t) \g \|_{L^p} &\lesssim \re^{-\mu_l t} \| \g \|_{L^\infty},  \\
                    \|\chi_-(\cdot - \undertilde{c} t) s_{c,1}(t) \g \|_{L^p} &\lesssim \re^{-\mu_l t} \|  \g \|_{L^\infty} \end{split} \label{e: refined left 1}
                \end{align}
                for $\g \in C_0(\R)$, $t \geq 0$, and $p = 2,\infty$.
                \item \emph{(Scattering terms).} The terms $s_{p,2}(t)$ and $s_{c,2}(t)$, given by~\eqref{e:defspsc}, satisfy
                \begin{align}
                \begin{split}
                    \| \chi_- (\cdot - \undertilde{c} t) \partial_t^j \partial_\xi^k s_{p,2}(t-s) \g\|_{L^p} &\lesssim (1+t-s)^{\frac{1}{2 p} -\frac{1}{4}-\frac{j+k}{2}} \| \chi_- (\cdot - \undertilde{c} s + 1) \g\|_{L^2} + \re^{- \mu_l (t-s)} \|\g \|_{L^\infty}, \\
                    \| \chi_- (\cdot - \undertilde{c} t) s_{c,2}(t-s) \g\|_{L^p} &\lesssim (1+t-s)^{\frac{1}{2 p}-\frac{3}{4}} \| \chi_- (\cdot - \undertilde{c} s + 1) \g\|_{L^2} + \re^{- \mu_l (t-s)} \|\g \|_{L^\infty}
                    \end{split}\label{e: refined left 2}
                \end{align}
                for $\g \in L^2(\R) \cap C_0(\R)$, $t,s \geq 0$ with $t \geq s$, and $p = 2,\infty$.
            \end{enumerate}
        \end{prop}
        \begin{proof}
        The first estimate in~\eqref{e: refined left 1} for $j = k = 0$ follows readily from the pointwise bound~\eqref{e:sp1_pointwise}, Young's convolution inequality, and the fact that $\chi_-(\cdot - \undertilde{c} t) \neq 0$ implies $\xi - c_g t \leq -\Delta c t$. On the other hand, in case $j + k \geq 1$, there exist, by the pointwise estimates~\eqref{e: sc1 pointwise} and~\eqref{e:sp1_pointwise}, constants $D_0,\mu > 0$ such that
            \begin{align*}
                \big| [\partial_t^j \partial_\xi^k s_{p,1} (t)\g](\xi)\big| + \big| [s_{c,1}(t) \g] (\xi) \big| \lesssim \left( \re^{-\frac{(\xi-c_gt)^2}{D_0t}} + \re^{-\mu |\xi|} \re^{- \mu t} \right) \| \g \|_{L^\infty}
            \end{align*}
        for $\xi \in \R$, $t \geq 1$, and $\g \in C_0(\R)$. Combining the latter with Young's convolution inequality and the fact that we have $\xi - c_g t \leq -\Delta c t$ whenever $\chi_-(\cdot - \undertilde{c} t) \neq 0$, establishes~\eqref{e: refined left 1} for $j+k \geq 1$. 
            
        To prove~\eqref{e: refined left 2}, we focus on the estimate for $s_{p,2}(t)$; the estimate for $s_{c,2}(t)$ is similar. Recall that $s_{p,2}(t)\g$ is given by the inverse Laplace representation~\eqref{e:defspsc}, where we have $\bar{s}_{p,2}(\lambda)\g = \chi_- \bar{s}_p^\mathrm{wt}(\lambda) \g_-$ with $\bar{s}_p^\mathrm{wt}(\lambda)\g_-$ given by~\eqref{e: spwt def}. Hence, applying Lemma~\ref{lem: summary wave train resolvent} and Proposition~\ref{p: pointwise estimates}, we obtain, provided $\delta > 0$ is sufficiently small, constants $D_1,\mu_1 > 0$ such that
        \begin{align*}
                \big|[\partial_t^j\partial_\xi^k s_{p,2} (t-s) \g](\xi)\big| \lesssim \chi_-(\xi)\int_\R G^{j, k}(\xi - \zeta, t - s) \chi_- (\zeta) |\g(\zeta)| \, \de \zeta
        \end{align*}
        for $\xi \in \R$, $\g \in L^2(\R) \cap C_0(\R)$, and $t,s \geq 0$ with $t - s\geq 1$, where we denote
        \begin{align*}
                G^{j,k}(z, s) = \frac{\chi_- (z)}{(1+s)^{\frac{1}{2} + \frac{j+k}{2}}} \re^{-\frac{(z - c_g(t-s))^2}{D_1s}}  + \chi_-(z) \re^{-\mu_1 s} \re^{-\mu_1 |z|}. 
            \end{align*}
        We decompose
            \begin{align*}
                \chi_-(\xi - \undertilde{c} t) G^{j,k}(\xi -\zeta, t - s) = G^{j,k}_I(\xi, \zeta, t, s) + G^{j,k}_{II}(\xi, \zeta, t, s), 
            \end{align*}
            with 
            \begin{align*}
                G^{j,k}_{I} (\xi, \zeta, t, s) = \chi_- (\xi - \undertilde{c} t) G^{j,k}(\xi- \zeta, t - s) \chi_-(\zeta - \undertilde{c} s + 1), 
            \end{align*}
            and
            \begin{align*}
                G^{j,k}_{II} (\xi, \zeta, t, s) = \chi_-(\xi - \undertilde{c} t) G^{j,k}(\xi- \zeta, t -s) \big[1-\chi_-(\zeta - \undertilde{c} s + 1)\big]. 
            \end{align*}
            Young's convolution inequality readily yields the estimate 
            \begin{align} \label{e: GI bound}
                \left\|\chi_- \int_\R G^{j,k}_{I}(\cdot, \zeta, t, s) \chi_-(\zeta) |\g(\zeta)| \, \de \zeta\right\|_{L^p} \lesssim \frac{\|\chi_-(\cdot - \undertilde{c} s + 1) \g \|_{L^2}}{(1+t-s)^{\frac14 - \frac{1}{2p}+\frac{j+k}{2}}} + \re^{-\mu_1 (t-s)} \|\g\|_{L^\infty} 
            \end{align}
            for $\g \in L^2(\R) \cap C_0(\R)$ and $t,s \geq 0$ with $t - s \geq 0$. On the other hand, the term
            $G^{j,k}_{II}(\xi,\zeta,t,s)$ contains the factor $\chi_-(\xi - \undertilde{c} t) [1-\chi_-(\zeta - \undertilde{c} s + 1)]$. So, in estimating this term we may assume $\xi - \undertilde{c} t \leq 0$ and $\zeta - \undertilde{c} s + 1 \geq -1$, which implies $\xi - \zeta \leq \undertilde{c} (t-s) + 2$. Thus, in this region, we can extract exponential decay in time from the Gaussian factor
            \begin{align*}
                \re^{-\frac{(\xi - \zeta - c_g (t-s))^2}{2D_1(t-s)}},
            \end{align*}
            which we bound by
            \begin{align*}
            \re^{\frac{2 \Delta c}{D_1}} \re^{- \frac{(\Delta c)^2}{2D_1} (t-s)}
            \end{align*}
            for $\xi,\zeta \in \R$ and $t,s \geq 0$ with $\xi - \zeta \leq \undertilde{c}(t-s) + 2$ and $t-s\geq0$. Hence, using Young's convolution inequality, we find a constant $\mu_0 > 0$ such that
            \begin{align} \label{e: GII bound}
                \left\|\chi_- \int_\R G^{j,k}_{II}(\cdot, \zeta, t, s) \chi_-(\zeta) |\g(\zeta)| \, \de \zeta\right\|_{L^p} \lesssim \re^{-\mu_0 (t-s)} \|\g\|_{L^\infty} 
            \end{align}
            for $\g \in C_0(\R)$ and $t,s \geq 0$ with $t - s \geq 0$. Combining the estimates~\eqref{e: GI bound} and~\eqref{e: GII bound}, we obtain the desired bound on $\| \chi_- (\cdot - \undertilde{c} t) \partial_t^j \partial_\xi^k s_{p,2}(t-s) \g\|_{L^p}$.
        \end{proof}        
        
\section{Nonlinear iteration scheme and nonlinear estimates} \label{sec: nonlinear iteration}

In this section, we set up the iteration scheme which will be employed to prove our nonlinear stability result, Theorem~\ref{t: main}. In addition, we derive estimates on the associated nonlinearities.

\subsection{The unmodulated perturbation}

We wish to control the long-time dynamics of sufficiently localized perturbations of the pushed pattern-forming front $\Ups$. To this end, we consider the solution $\u(t)$ of~\eqref{e: FHN comoving} with $\u(0) = \Ups + \w_0$, where $\v_0 := \omega_0\w_0 \in H^3(\R) \times H^2(\R)$ is small. This implies that the initial perturbation $\w_0$ is $L^2$-localized on $(-\infty,0]$ and exponentially localized on $[0,\infty)$. The exponential localization on the right is required to stabilize the essential spectrum associated with the leading edge of the front.

We measure the deviation of the solution $\u(t)$ from the front by the weighted perturbation 
\begin{align*}\vt(t) = \omega_0(\u(t) - \Ups).\end{align*} 
It arises as the solution the semilinear evolution equation
\begin{align}
\vt_t = \Lps \vt + \NT(\vt) \label{e:umodpert}
\end{align}
with initial condition $\vt(0) = \v_0$, where the nonlinear remainder $\NT \colon H^1(\R) \to H^1(\R)$ given by
\begin{align*}
\NT(\vt) = \omega_0\widetilde{N}\left(\frac{\vt}{\omega_0}\right), \qquad \widetilde{N}(\w) = F(\Ups+\w) - F(\Ups) - F'(\Ups) \w
\end{align*}
is locally Lipschitz continuous. Since $\Lps$ generates a $C^0$-semigroup on the Hilbert space $H^1(\R) \times H^1(\R)$ with domain $H^3(\R) \times H^2(\R)$ by Proposition~\ref{p: C0 semigroup}, local existence and uniqueness of the solution $\vt(t)$ follows readily from classical semigroup theory, see e.g.~\cite[Theorems~6.1.4 and~6.1.6]{Pazy}. 

\begin{prop}[Local well-posedness of the unmodulated perturbation] \label{p:local_unmod}
There exists a maximal time $T_{\max} \in (0,\infty]$ such that~\eqref{e:umodpert} admits a unique classical solution
\begin{align} \vt \in C\big([0,T_{\max}),H^3(\R) \times H^2(\R)\big) \cap C^1\big([0,T_{\max}),H^1(\R) \times H^1(\R)\big), \label{e:regv}\end{align}
with initial condition $\vt(0) = \v_0$. Moreover, $T_{\max} < \infty$ implies
\begin{align} \limsup_{t \nearrow T_{\max}} \left\|\vt(t)\right\|_{H^1} = \infty. \label{e:blowup}\end{align}
\end{prop}

\subsection{The inverse-modulated perturbation}

As explained in~\S\ref{sec:techniques}, the linearization $\Lps$ possesses a neutral translational eigenvalue at $0$, which is embedded in the essential spectrum associated with the diffusively stable wave train in the wake of the front. The presence of this marginal spectrum prevents decay of the semigroup $\re^{\Lps t}$ and, consequently, obstructs a direct nonlinear stability argument based on iterative estimates of $\vt(t)$ via the Duhamel representation of~\eqref{e:umodpert}. To address this difficulty, we introduce a smooth modulation function $\psi(\xi,t)$ that simultaneously captures the response of the front interface to the excitation of the translational mode and the phase response of the diffusive wave train in the wake. As a result, the \emph{inverse-modulated perturbation}
\begin{align}
\v(\xi,t) = \omega_0(\xi)\left(\u(\xi - \psi(\xi,t),t) - \Ups(\xi)\right) \label{e:modpert}
\end{align}
is effectively decoupled from the leading-order neutral dynamics, and one may therefore expect it to decay. 

Following~\cite{JONZ,FHNpulled}, we first derive an equation for $\v$ and then define $\psi$ a posteriori. Using that both $\u(t) = \Ups + \omega_0^{-1} \vt(t)$ and $\Ups$ solve~\eqref{e: FHN comoving}, $\psi$ is smooth, and $\vt$ satisfies~\eqref{e:regv}, one finds that the inverse-modulated perturbation $\v$ obeys the \emph{quasilinear} equation
\begin{align}
\left(\partial_t - \Lps\right)\left[\v + \omega_0 \Ups' \psi\right] = \mathcal{N}(\v,\psi,\partial_t \psi) + \left(\partial_t - \Lps\right)\left[\psi_\xi \v\right], \label{e:modpertbeq}
\end{align}
with nonlinear remainder
\begin{align*}
\mathcal{N}(\v,\psi,\psi_t) &= \omega_0\left(\mathcal Q\left(\frac{\v}{\omega_0},\psi\right) + \partial_\xi \mathcal R\left(\frac{\v}{\omega_0},\psi,\psi_t\right)\right),
\end{align*}
where
\begin{align} \label{e:defNLQ}
\mathcal Q(\w,\psi) &= \left(F(\Ups+\w) - F(\Ups) - F'(\Ups) \w\right)\left(1-\psi_\xi\right)
\end{align}
is quadratic in $\w$ and 
\begin{align} \label{e:defNLR}
\mathcal R(\w,\psi) = \left(c\psi_\xi - \psi_t\right)\w + D\left(\frac{\left(\w_\xi + \Ups'\psi_\xi\right)\psi_\xi}{1-\psi_\xi} + \left(\w\psi_\xi\right)_\xi\right) 
\end{align}
contains all terms which are linear in $\w$. We refer to~\cite[Appendix~D]{FHNpulled} for further details on the derivatation of~\eqref{e:modpertbeq}.

We establish estimates on the nonlinearity in~\eqref{e:modpertbeq}, which readily follow from Taylor's theorem, the continuous embedding $H^1(\R) \hookrightarrow L^\infty(\R)$, the inequality $0 \leq \chi_-(\xi+1) \leq \chi_-(\xi)^2$ for $\xi \in \R$, and the fact that $\omega_0^{-1}$, $\omega_0' \omega_0^{-1}$, $\omega_0'' \omega_0^{-1}$, $\omega_0 \Ups'$, and $\omega_0 \Ups''$ are bounded; see Hypothesis~\ref{hyp: leading edge}.

\begin{lemma}[Nonlinear estimates] \label{l: nonlinear estimates}
Fix $\kappa > 0$. We have
\begin{align*} 
\begin{split}
\left\|\mathcal{N}\left(\v,\psi,\psi_t\right)\right\|_{L^2} &\lesssim\|v_1\|_{L^2}\|v_1\|_{L^\infty} + \left\|\nabla \psi\right\|_{W^{2,\infty}}\left(\left\|\v\right\|_{H^2 \times H^1} + \left\|\psi_\xi\right\|_{L^2}\right),\\
\left\|\chi_-(\cdot - \xi_0 + 1)\mathcal{N}\left(\v,\psi,\psi_t\right)\right\|_{L^2} &\lesssim \|\chi_-(\cdot - \xi_0) v_1\|_{L^2}\|\chi_-(\cdot - \xi_0)v_1\|_{L^\infty} + \sum_{j = 0}^2 \big\|\chi_-(\cdot - \xi_0) \partial_\xi^j \nabla \psi\big\|_{L^\infty} \\ 
& \qquad \cdot \left(\sum_{j = 0}^1 \left\|\chi_-(\cdot - \xi_0)\partial_\xi^j \v\right\|_{L^2} + \big\|\chi_-(\cdot - \xi_0)\partial_\xi^2 v_1\big\|_{L^2} + \left\|\chi_-(\cdot - \xi_0)\psi_\xi\right\|_{L^2}\right)
\end{split}
\end{align*}
for each $\xi_0 \in \R$, $\v = (v_1,v_2)^\top \in H^2(\R) \times H^1(\R)$, and $(\psi,\psi_t) \in W^{3,\infty}(\R) \times W^{2,\infty}(\R)$ satisfying $\psi_\xi \in L^2(\R)$ and $\|v_1\|_{L^\infty}, \|\psi_\xi\|_{L^{\infty}} \leq \frac{1}{2}$, where we abbreviate $\nabla \psi = (\psi_\xi,\psi_t)^\top$. Moreover, the estimates
\begin{align*} 
\begin{split}
\left\|\mathcal{N}\left(\v,\psi,\psi_t\right)\right\|_{L^\infty} &\lesssim\|v_1\|_{L^\infty}^2 + \left\|\nabla \psi\right\|_{W^{2,\infty}}\left(\left\|\v\right\|_{W^{2,\infty} \times W^{1,\infty}} + \left\|\psi_\xi\right\|_{L^\infty}\right),\\
\left\|\omega_{\kappa,0}\mathcal{N}\left(\v,\psi,\psi_t\right)\right\|_{L^\infty} &\lesssim\|\omega_{\kappa,0} v_1\|_{L^\infty}\|v_1\|_{L^\infty} + \sum_{j = 0}^2 \big\|\omega_{\kappa,0} \partial_\xi^j \nabla \psi\big\|_{L^\infty}\left(\left\|\v\right\|_{W^{2,\infty} \times W^{1,\infty}} + \left\|\psi_\xi\right\|_{L^\infty}\right)
\end{split}
\end{align*}
hold for each $\xi_0 \in \R$, $\v = (v_1,v_2)^\top \in W^{2,\infty}(\R) \times W^{1,\infty}(\R)$ and $(\psi,\psi_t) \in W^{3,\infty}(\R) \times W^{2,\infty}(\R)$ satisfying $\|v_1\|_{L^\infty}, \|\psi_\xi\|_{L^{\infty}} \leq \frac{1}{2}$.
\end{lemma}

Assuming for the moment that the modulation $\psi(t)$ vanishes identically at $t=0$, the Duhamel formulation for the inverse-modulated perturbation $\v(t)$ reads
\begin{align}
\v(t) + \omega_0\Ups'\psi(t) = \re^{\Lps t} \v_0 + \int_0^t \re^{\Lps(t-s)}\mathcal{N}(\v(s),\psi(s),\partial_s \psi(s))\de s + \psi_\xi(t)\v(t). \label{e:intv}
\end{align}
We now make a judicious choice for $\psi(t)$ so that it compensates for the critical contributions on the right-hand side of~\eqref{e:intv}. Specifically, motivated by the semigroup decomposition~\eqref{e: semigroup decomposition}, we define $\psi(t)$ by
\begin{align}
\psi(t) = s_p(t)\v_0 + \int_0^t s_p(t-s) \mathcal{N}(\v(s),\psi(s),\partial_s \psi(s))\de s. \label{e:intpsi}
\end{align}
Since $s_p(t) = 0$ for $t \in [0,1]$ by Theorem~\ref{t: linear estimates}, $\psi(t)$ must vanish identically on $[0,1]$ and~\eqref{e:intpsi} provides an iterative definition of $\psi(t)$ for $t \in [0,T_{\max})$ as long as $\|\psi_\xi(t)\|_{L^\infty} \leq \frac12$ (so that the nonlinearity $\mathcal{N}$ is well-defined). More precisely, suppose that we have defined $\psi$ on $[0,n]$ for some $n \in \N$ with $n < T_{\max}$ such that $\|\psi_\xi(s)\|_{L^\infty} < \frac12$ for $s \in [0,n]$. Then, for $t \in [0,1]$ with $n +t < T_{\max}$, we define $\psi(n+t)$ via~\eqref{e:intpsi} by noting that the right-hand side only depends on $\vt|_{[0,n]}$ and $\psi|_{[0,n]}$, since $s_p(s)$ vanishes for $s \in [0,1]$ and the inverse-modulated perturbation may be expressed as
\begin{align}
\v(\xi,t) = \vt(\xi-\psi(\xi,t),t) + \omega_0(\xi)\left(\Ups(\xi - \psi(\xi,t)) - \Ups(\xi)\right). \label{e:definvmod2}  
\end{align}

In the next result, we establish localization and regularity properties of $\psi(t)$ and $\v(t)$. The estimates in Theorem~\ref{t: linear estimates} imply that the linear term $s_p(t) \v_0$ in~\eqref{e:intpsi} is smooth and bounded for all $t \geq 0$. Moreover, it is $L^2$-localized on $(-\infty,0]$, but not integrable on $[0,\infty)$. On the other hand, spatial and temporal derivatives of $s_p(t) \v_0$ are $L^2$-localized. We show that these localization and regularity properties carry over to $\psi(t)$. Moreover, since $\omega_0 \Ups'$ is exponentially decaying on $[0,\infty)$ by Hypothesis~\ref{hyp: leading edge}, the inverse-modulated perturbation~\eqref{e:definvmod2} is also $L^2$-localized and inherits its regularity properties from $\vt(t)$.

\begin{prop}[Regularity of the modulation and inverse-modulated perturbation] \label{p:psi}
Let $T_{\max}$ and $\vt$ be as in Proposition~\ref{p:local_unmod}. 
Then, there exists a maximal time $\tau_{\max} \in (0,T_{\max}]$ such that the following conditions hold:
\begin{itemize}
    \item The modulation and the inverse-modulated perturbation given by~\eqref{e:intpsi} and~\eqref{e:definvmod2}, respectively, satisfy
\begin{align} \label{e:psiid1}
\psi &\in C^\infty\big(\R \times [0,\tau_{\max}),\R\big) \cap C\big([0,\tau_{\max}),L^\infty(\R)\big),  \\
\v = (v_1,v_2)^\top &\in C\left([0,\tau_{\max}),\big(C_0^2(\R) \times C^1_0(\R) \big) \cap \big(H^2(\R) \times H^1(\R)\big) \right); \label{e:vid1}
\end{align}
\item $\|\psi_\xi(t)\|_{L^\infty} < \frac12$ for all $t \in [0,\tau_{\max})$;
\item For any $\ell, j \in \N_0$ with $\ell+j \geq 1$, we have 
\begin{align}
\chi_- \psi, \partial_\xi^\ell \partial_t^j \psi \in \big([0,\tau_{\max}), L^2(\R)\big). \label{e:psiid2}
\end{align}
\end{itemize}
In particular, $\tau_{\max} < T_{\max}$ implies that
\begin{align} \limsup_{t \nearrow \tau_{\max}} \|\psi_\xi(t)\|_{L^\infty} = \frac12. \label{e:blowup3}\end{align}
In addition, we have $\psi(t)=0$ for $t \in [0,\tau_{\max})$ with $t \leq 1$. 
Finally, the Duhamel formulation~\eqref{e:intv} holds for all $t \in [0,\tau_{\max})$.
\end{prop}
\begin{proof}
Since $s_p(t)$ vanishes identically for $t \in [0,1]$ by Theorem~\ref{t: linear estimates}, we have $\psi(t) = 0$ and $\v(t) = \vt(t)$ for $t \in [0,T_{\max})$ with $t \leq 1$. Next, suppose that $\psi$ and $\v$ are defined on $[0,n]$ for some $n \in \N$ with $n < T_{\max}$, we have $\|\psi_\xi(s)\| < \frac12$ for all $s \in [0,n]$, and identities~\eqref{e:psiid1},~\eqref{e:vid1}, and~\eqref{e:psiid2} hold with $[0,\tau_{\max})$ replaced by $[0,n]$. Let $t_0 \in [0,1]$ with $n+t_0 < T_{\max}$. Using the continuous embedding $H^1(\R) \hookrightarrow L^\infty(\R)$ and noting that $\omega_0^{-1}$, $\omega_0' \omega_0^{-1}$, $\omega_0'' \omega_0^{-1}$, $\omega_0 \Ups'$ and $\omega_0 \Ups''$ are bounded, one readily observes that the nonlinearity $N(s) = \mathcal{N}(\v(s),\psi(s),\partial_s \psi(s))$ satisfies $N \in C\big([0,n],L^2(\R)\cap L^\infty(\R)\big)$. Combining the latter with the estimates in Theorem~\ref{t: linear estimates} and the fact that $s_p(s)$ vanishes for $s \in [0,1]$, we use~\eqref{e:intpsi} to extend $\psi$ to $[0,n+t_0]$ such that identities~\eqref{e:psiid1} and~\eqref{e:psiid2} hold with $[0,\tau_{\max})$ replaced by $[0,n+t_0]$. Subsequently defining $\v$ via~\eqref{e:definvmod2} on $[0,n+t_0]$ and using Proposition~\ref{p:local_unmod}, identities~\eqref{e:psiid1} and~\eqref{e:psiid2}, the mean value theorem, the continuous embedding $H^1(\R) \hookrightarrow C_0(\R)$, and the uniform continuity of functions in $C_0(\R)$, we infer $\v \in C\big([0,n+t_0],C_0^2(\R) \times C_0^1(\R)\big)$. Furthermore, we have $\v \in C\big([0,n+t_0],H^2(\R) \times H^1(\R)\big)$ by Lemmas~\ref{l:MVT in L2} and~\ref{l:H2estimate}. Therefore,~\eqref{e:vid1} holds with $[0,\tau_{\max})$ replaced by $[0,n+t_0]$. It follows, since $\psi(s)$ vanishes identically at $s = 0$, that $\v(t)$ obeys the Duhamel formula~\eqref{e:intv} for $t \in [0,n+t_0]$. Finally, if there exists a point $s_0 \in [0,n+t_0]$ with $\|\psi_\xi(s_0)\|_{L^\infty} \geq \frac12$, then, by continuity, 
$$\tau_{\max} := \min\left\{s \in [0,n+t_0] : \|\psi_\xi(s)\|_{L^\infty} = \tfrac12\right\} < T_{\max}$$
exists. The statement now follows by induction on $n$. 
\end{proof}

Substituting~\eqref{e:intpsi} into~\eqref{e:intv} and recalling the semigroup decomposition~\eqref{e: semigroup decomposition}, we arrive at
\begin{align}
\begin{split}
\v(t) &= \left(s_c(t) + s_e(t)\right)\v_0+\int_0^t\left(s_c(t-s)+s_e(t-s)\right)\mathcal{N}(\v(s),\psi(s),\partial_t \psi(s)) \de s + \psi_\xi(t)\v(t)
\end{split}\label{e:intv2}
\end{align}
for $t \in [0,\tau_{\max})$. Comparing~\eqref{e:intv2} with~\eqref{e:intv} and recalling the linear estimates in Theorem~\ref{t: linear estimates}, we observe that the slowest decaying terms on the right-hand side of~\eqref{e:intv} have been eliminated by our choice of $\psi(t)$. Moreover, the nonlinearity $\mathcal{N}$ in~\eqref{e:intv2} only depends on spatial and temporal derivatives of $\psi$, which satisfy
\begin{align}
\partial_\xi^\ell \partial_t^j \psi(t) = \partial_\xi^\ell \partial_t^j s_p(t)\v_0 + \int_0^t \partial_\xi^\ell \partial_t^j s_p(t-s) \mathcal{N}(\v(s),\psi(s),\partial_t \psi(s))\de s, \label{e:intpsi2}
\end{align}
for $\ell,j \in \N_0$ and $t \in [0,\tau_{\max})$. Since the propagators $\partial_\xi^\ell \partial_t^j s_p(t)$ decay at the same rate as $s_c(t) + s_e(t)$ for $\ell + j \geq 1$ by Theorem~\ref{t: linear estimates}, we expect that derivatives of $\psi(t)$ exhibit the same decay rates as $\v(t)$. We show in the upcoming section that these improved decay rates are sufficient to close a nonlinear argument.

\subsection{Forward-modulated damping}

We control regularity in the quasilinear equation~\eqref{e:modpertbeq} with the aid of forward-modulated damping estimates~\cite{ZUM22}. The same strategy was employed in the nonlinear stability analysis~\cite{FHNpulled} of pulled pattern-forming fronts in the FitzHugh--Nagumo system. Here, we collect the relevant results.

The \emph{forward-modulated perturbation}
\begin{align}
\vf(\xi,t) = \omega_0(\xi)\left(\u(\xi,t) - \Ups(\xi + \psi(\xi,t))\right) = \vt(\xi,t) + \omega_0(\xi)\left(\Ups(\xi) -  \Ups(\xi + \psi(\xi,t))\right) \label{e:fmodpert}
\end{align}
obeys the equation
\begin{align}
\begin{split}
\vf_t &= D\left(\vf_{\xi \xi} + 2 \omega_0 (\omega_0^{-1})' \vf_\xi + \omega_0 (\omega_0^{-1})'' \vf\right) + \omega_0 \left(F\left(\frac{\vf}{\omega_0} + \uf_{\mathrm{ps},0}\right) - F(\uf_{\mathrm{ps},0})\right)\\ &\qquad + \, \cps \left(\vf_\xi + \omega_0 (\omega_0^{-1})'\vf\right) + \omega_0 \left(\cps \psi_\xi - \partial_t \psi\right)\uf_{\mathrm{ps},1} + \omega_0 D\left(\psi_{\xi\xi} \uf_{\mathrm{ps},1} + \psi_\xi\left(\psi_\xi + 2\right)\uf_{\mathrm{ps},2}\right),
\end{split}
\label{e:fmodpertbeq}
\end{align}
where we denote $\uf_{\mathrm{ps},j}(\xi,t) = \big(\partial_\xi^j \Ups\big)(\xi + \psi(\xi,t))$. It inherits its regularity properties from those of $\vt(t)$ and $\psi(t)$.

\begin{corollary}[Regularity of the forward-modulated perturbation] \label{c:local_vf}
Let $\vt(t)$, $\psi(t)$, and $\tau_{\max}$ be as in Propositions~\ref{p:local_unmod} and~\ref{p:psi}. Then, the forward-modulated perturbation, given by~\eqref{e:fmodpert}, satisfies $\vf \in C\big([0,\tau_{\max}),H^3(\R) \times H^2(\R)\big) \cap C^1\big([0,\tau_{\max}),H^1(\R) \times H^1(\R)\big)$.
\end{corollary}
\begin{proof}
The result follows directly from Propositions~\ref{p:local_unmod} and~\ref{p:psi} and Lemma~\ref{l:H2estimate}.
\end{proof}

Using that equation~\eqref{e:fmodpertbeq} is semilinear in $\vf(t)$, and it is linearly damped by the term $\partial_{\xi\xi}\mathring{v}_1$ in the first component and by the term $-\varepsilon \gamma \mathring{v}_2$ in the second component, one obtains the following nonlinear damping estimate.

\begin{prop}[Nonlinear damping estimate] \label{p: nonlinear damping}
Let $\vf(t)$, $\psi(t)$, and $\tau_{\max}$ be as in Proposition~\ref{p:psi} and Corollary~\ref{c:local_vf}. Fix $R > 0$. There exist constants $C, \vartheta > 0$ such that the forward-modulated perturbation $\vf(t)$ satisfies
\begin{align*} 
\begin{split}
\|\vf(t)\|_{H^3 \times H^2}^2 &\leq C\left(\re^{-\vartheta t} \|\v_0\|_{H^3 \times H^2}^2 + \left\|\vf(t)\right\|_{L^2}^2+ \int_0^t \re^{-\vartheta(t-s)} \left(\|\vf(s)\|_{L^2}^2 + \|\psi_\xi(s)\|_{H^3}^2 + \|\partial_s \psi(s)\|_{H^2}^2 \right) \de s\right)
\end{split}
\end{align*}
for each $t \in [0,\tau_{\max})$ with
\begin{align*}
\sup_{0 \leq s \leq t} \left(\|\mathring{v}_1(s)\|_{H^2} + \|\psi(s)\|_{W^{2,\infty}}\right) \leq R.
\end{align*}  
\end{prop}
\begin{proof}
The result follows verbatim from the proof of~\cite[Proposition~8.6]{FHNpulled}, where the same nonlinear damping estimate was obtained in the setting of pulled pattern-forming fronts in the FitzHugh--Nagumo system, with one modification: instead of using that $\psi(t)$ vanishes identically on $[-1,\infty)$ as in~\cite{FHNpulled}, we use that $\smash{\omega_0 \uf_{\mathrm{ps},j} = \omega_0 (\partial_\xi^j \Ups)(\cdot + \psi(\cdot,s))}$ is bounded in $L^\infty(\R)$ by a $t$-independent constant for $s \in [0,t]$ and $j = 1,2$. This follows from Lemma~\ref{l:omegaid} in combination with the facts that $\smash{\omega_0 (\partial_\xi^j \Ups)}$ is bounded by Lemma~\ref{l:H2estimate} and we have $\|\psi(s)\|_{L^\infty} \leq R$ for $s \in [0,t]$.
\end{proof}

The nonlinear damping estimate in Proposition~\ref{p: nonlinear damping} bounds the $(H^3 \times H^2)$-norm of the forward-modulated perturbation $\vf(t)$ in terms of the $(H^3 \times H^2)$-norm of the initial condition $\v_0$, the $L^2$-norm of $\vf(s)$, and higher-order Sobolev norms of the modulation $\psi(s)$ for $s \in [0,t]$. The following result shows that relevant $W^{k,p}$-norms of the forward- and inverse-modulated perturbations are equivalent up to terms involving $\psi_\xi$. Consequently, the nonlinear damping estimate for $\vf(t)$ provides effectively regularity control of the inverse-modulated perturbation $\v(t)$ within the nonlinear iteration scheme.

\begin{lemma}[Norm equivalences] \label{l: equivalence}
Fix $R > 0$. Let $\v(t)$, $\psi(t)$, and $\tau_{\max}$ be as in Proposition~\ref{p:psi} and let $\vf(t)$ be as in Corollary~\ref{c:local_vf}. Then, there exists a constant $C > 0$ such that
\begin{align}
\|\v(t)\|_{W^{k,p} \times W^{k-1,p}} &\leq C\left(\|\vf(t)\|_{W^{k,p} \times W^{k-1,p}} + \|\psi_\xi(t)\|_{W^{k-1,p}}\right),
\label{e:bdibf}\\
\|\vf(t)\|_{L^p} &\leq C\left(\|\v(t)\|_{L^p} + \|\psi_\xi(t)\|_{L^p}\right)
\label{e:bdibf2}
\end{align}
for $k = 1,2,3$, $p = 2,\infty$, and any $t \in [0,\tau_{\max})$ with $\|\psi(t)\|_{W^{2,\infty}} \leq R$. Moreover, we have
\begin{align}
\|\chi_\pm(\cdot - \xi_0) \vf(t)\|_{L^\infty} &\leq C\Big(\left\|\chi_\pm\left(\cdot - \xi_0 \pm \|\psi(t)\|_{L^\infty}\right) \v(t)\right\|_{L^\infty} + \left\|\chi_\pm\left(\cdot - \xi_0 \pm \|\psi(t)\|_{L^\infty}\right) \psi_\xi(t)\right\|_{L^\infty}\Big)
\label{e:bdibf3}
\end{align}
for any $\xi_0 \in \R$ and $t \in [0,\tau_{\max})$ with $\|\psi(t)\|_{L^\infty} \leq R$.
\end{lemma}
\begin{proof}
Let $\xi_0 \in \R$ and $t \in [0,\tau_{\max})$ with $\smash{\|\psi(t)\|_{W^{2,\infty}} \leq R}$. Proposition~\ref{p:psi} yields that $\|\psi_\xi(t)\|_{L^\infty} \leq \frac12$. So, the function $h_t \colon \R \to \R$ given by $h_t(\xi) = \xi - \psi(\xi,t)$ is strictly increasing and invertible with inverse 
\begin{align} \label{e:htexpr}
h_t^{-1}(\xi) = \xi + \psi(h_t^{-1}(\xi),t)
\end{align}
for $\xi \in \R$. 

Using identity~\eqref{e:htexpr}, Lemma~\ref{l:MVT in L2}, and the fact that $\|\psi(t)\|_{L^\infty} \leq R$, we obtain a $t$-independent constant $C > 0$ such that
\begin{align} \label{e:equivest}
\left|\omega_0(\xi) \left(\partial_\xi^j\Ups\right)\left(h_t^{-1}(\xi)\right) - \left(\partial_\xi^j\Ups\right)(\xi + \psi(\xi,t))\right| &\leq C\left\|\omega_0 \partial_\xi^{j+1}\Ups\right\|_{L^\infty} \left|\psi(h_t^{-1}(\xi),t) - \psi(\xi,t)\right|
\end{align}
and 
\begin{align} \label{e:equivest2}
\begin{split}
&\left|\omega_0(\xi) \left(\partial_\xi^j\Ups\right)\left(\xi - \psi(\xi,t) + \psi(\xi-\psi(\xi,t),t)\right) - \partial_\xi^j \Ups(\xi)\right|\\
&\qquad \leq C\left\|\omega_0 \partial_\xi^{j+1}\Ups\right\|_{L^\infty}\left|\psi(\xi - \psi(\xi,t),t) - \psi(\xi,t)\right|
\end{split}
\end{align}
for $j = 0,1,2$ and $\xi \in \R$. Moreover, using~\eqref{e:htexpr}, the mean value theorem, Lemma~\ref{l:MVT in L2}, and the fact that $\chi_- = 1-\chi_+$ is monotonically decreasing, we find a $t$-independent constant $C > 0$ such that
\begin{align} \label{e:psiboundsL2Linfty}
\left\|\psi(h_t^{-1}(\cdot),t) - \psi(\cdot,t)\right\|_{L^p}, \left\|\psi(\cdot - \psi(\cdot,t),t) - \psi(\cdot,t)\right\|_{L^p} &\leq C\|\psi_\xi(t)\|_{L^p} \|\psi(t)\|_{L^\infty}
\end{align}
for $p = 2,\infty$, and 
\begin{align} \label{e:psiboundscutoff}
\left\|\chi_\pm(\cdot - \xi_0) \left(\psi(h_t^{-1}(\cdot),t) + \psi(\cdot,t)\right)\right\|_{L^\infty} \leq C\left\|\chi_\pm\left(\cdot - \xi_0 \pm \|\psi(t)\|_{L^\infty}\right) \psi_\xi(t)\right\|_{L^\infty} \|\psi(t)\|_{L^\infty}.
\end{align}
Furthermore, it has been shown in the proof of~\cite[Lemma~8.7]{FHNpulled} that we have
\begin{align} \label{e:L2subs}
\left\|f \circ h_t^{-1}\right\|_{L^2}, \left\|f(\cdot - \psi(\cdot,t))\right\|_{L^2} \leq \sqrt{2} \|f\|_{L^2}
\end{align}
for $f \in L^2(\R)$.

We substitute~\eqref{e:fmodpert} into~\eqref{e:modpert} and arrive at
\begin{align*}
\begin{split}
\v(\xi,t) &= \omega_0(\xi)\left(\frac{\vf(\xi-\psi(\xi,t),t)}{\omega_0(\xi - \psi(\xi,t))} + \Ups\left(\xi-\psi(\xi,t) + \psi(\xi-\psi(\xi,t),t)\right) - \Ups(\xi)\right)
\end{split}
\end{align*}
for $\xi \in \R$. Thus, recalling $\smash{\|\psi_\xi(t)\|_{W^{2,\infty}} \leq R}$, employing estimates~\eqref{e:equivest2},~\eqref{e:psiboundsL2Linfty}, and~\eqref{e:L2subs}, applying Lemma~\ref{l:omegaid}, and using that $\smash{\omega_0^{-1} \partial_\xi^j \omega_0}$ and $\smash{\omega_0 \partial_\xi^j \Ups}$ are bounded for $j = 1,2,3$ by Lemma~\ref{l:H2estimate} , we establish~\eqref{e:bdibf}. Conversely, substituting~\eqref{e:modpert} into~\eqref{e:fmodpert}, we obtain
\begin{align} \label{e:vf_id}
\vf(\xi,t) &= \omega_0(\xi)\left(\frac{\v(h_t^{-1}(\xi),t)}{\omega_0(h_t^{-1}(\xi))} + \Ups\left(h_t^{-1}(\xi)\right) - \Ups(\xi + \psi(\xi,t))\right)
\end{align}
for $\xi \in \R$. Therefore, noting that $\smash{\|\psi(t)\|_{L^{\infty}} \leq R}$, using~\eqref{e:htexpr}, employing the estimates~\eqref{e:equivest},~\eqref{e:psiboundsL2Linfty}, and~\eqref{e:L2subs}, applying Lemma~\ref{l:omegaid}, and recalling that $\smash{\omega_0 \Ups'}$ is bounded by Lemma~\ref{l:H2estimate}, we arrive at~\eqref{e:bdibf2}. Finally, multiplying~\eqref{e:vf_id} with $\chi_\pm(\cdot - \xi_0)$, applying Lemma~\ref{l:omegaid} and estimates~\eqref{e:equivest} and~\eqref{e:psiboundscutoff}, and using identity~\eqref{e:htexpr} and the facts that $\chi_- = 1-\chi_+$ is monotonically decreasing, $\omega_0 \Ups'$ is bounded by Lemma~\ref{l:H2estimate}, and we have $\|\psi(t)\|_{L^\infty} \leq R$, we obtain~\eqref{e:bdibf3}.
\end{proof}

\begin{remark}{ \upshape
We note that equivalence of $W^{k,p}$-norms of the  forward- and inverse-modulated perturbations of periodic waves was established in~\cite{ZUM22}. However, the proof of Lemma~\ref{l: equivalence} is more tedious due to the presence of the (unbounded) exponential weight $\omega_0$ and the cut-off functions $\chi_\pm$.
}\end{remark}

\section{Nonlinear stability argument --- proof of Theorem~\ref{t: main}} \label{sec:proof}

We close a nonlinear iteration argument by tracking relevant norms of the modulation $\psi(t)$ and the inverse-modulated perturbation $\v(t)$. This is achieved by applying the linear estimates from Theorem~\ref{t: linear estimates}, the left light cone bounds from Proposition~\ref{prop: left light cone}, and the nonlinear estimates from Lemma~\ref{l: nonlinear estimates} to the integral equations~\eqref{e:intv2} and~\eqref{e:intpsi2}.  Regularity is controlled via norm equivalence with the forward-modulated perturbation $\vf(t)$, which obeys a nonlinear damping estimate; see Proposition~\ref{p: nonlinear damping}. After closing the nonlinear argument and establishing global control, we determine the asymptotic phase $\psi_\infty$ and derive a posteriori bounds on $\psi(t)$, $\v(t)$, and $\vf(t)$ in the right light cone using Proposition~\ref{prop: right light cone}.

\begin{proof}[Proof of Theorem~\ref{t: main}]
Fix constants $K, \delta_c > 0$ such that $c_g + \delta_c < 0$. Take $\v_0 \in H^3(\R) \times H^2(\R)$ and set $\smash{E_0 := \|\v_0\|_{H^3 \times H^2}}$. Let $\vt(t)$ be the maximally defined solution to~\eqref{e:umodpert} with initial condition $\vt(0) = \v_0$ given by Proposition~\ref{p:local_unmod}. Let $\tau_{\max}$ be as in Proposition~\ref{p:psi}. On the interval $[0,\tau_{\max})$, we define the modulation $\psi(t)$ by~\eqref{e:intpsi} and the inverse- and forward-modulated perturbations $\v(t)$ and $\vf(t)$ by~\eqref{e:modpert} and~\eqref{e:fmodpert}, respectively; see Proposition~\ref{p:psi} and Corollary~\ref{c:local_vf}. 

Let $\eta_0 > 0$ be as in Proposition~\ref{prop: fredholm properties} and take $\eta \in (0,\eta_0)$. We further fix the parameters
\begin{align*}
\ell = 4, \qquad \Delta \mu = \frac18 \delta_c, \qquad \Delta \theta = \frac14 \delta_c, \qquad \Delta c = \frac12 \delta_c, \qquad \tilde{c} = c_g + \Delta c, \qquad \undertilde{c} = c_g - \Delta c,
\end{align*}
and then take $\kappa,\mu > 0$ as in Theorem~\ref{t: linear estimates}, $\kappa_r,\mu_r > 0$ as in Proposition~\ref{prop: right light cone}, and $\mu_l > 0$ as in Proposition~\ref{prop: left light cone}. Finally, we set 
\begin{align*}\kappa_0 = \min\left\{\kappa,\kappa_r,\frac{-\mu}{c_g + \Delta \mu}, \frac{\mu_r}{\Delta c - \Delta \theta}, \eta\right\} > 0.\end{align*} 
\noindent\textbf{Template function.} It follows by Propositions~\ref{p:local_unmod} and~\ref{p:psi} and Corollary~\ref{c:local_vf} that the template function $\varsigma \colon [0,\tau_{\max}) \to [0,\infty)$ given by
\begin{align*}
\varsigma(t) &= \smash{\sup_{0\leq s\leq t}} \left[\|\psi(s)\|_{L^\infty} + (1+s)^{\frac14} \left(\|\vf(s)\|_{H^3 \times H^2} + \left\|\nabla \psi(s)\right\|_{H^3} \right)\phantom{\sum_{j = 0}^2 } \right. \\
&\qquad  \left. \phantom{\frac{1}{\sqrt{1+s}}} + \sqrt{1+s}\left(\|\v(s)\|_{L^\infty} + \left\|\nabla \psi(s)\right\|_{W^{3,\infty}} + \left\|\chi_-(\cdot - \undertilde{c} s) \v(s)\right\|_{L^2} + \left\|\chi_-(\cdot - \undertilde{c} s) \psi_\xi(s)\right\|_{L^2}\right) \right.\\
&\qquad  \left. \phantom{\frac{1}{\sqrt{1+s}}} + (1+s)^{\frac34} \left(\left\|\chi_-(\cdot - \undertilde{c} s) \v(s)\right\|_{L^\infty} + \sum_{j = 0}^2 \big\|\chi_-(\cdot - \undertilde{c} s) \partial_\xi^j \nabla \psi(s)\big\|_{L^\infty}\right) \right.\\
&\qquad\left. \phantom{\frac{1}{\sqrt{1+s}}} + \re^{-\kappa_0 \left(c_g + \frac14 \delta_c\right) s} \Bigg(\|\omega_{\kappa_0,0} \v(s)\|_{L^\infty} + \|\omega_{\kappa_0,0} \psi(s)\|_{L^\infty} + 
\sum_{j = 0}^2 \big\|\omega_{\kappa_0,0} \partial_\xi^j \nabla \psi(s)\big\|_{L^\infty}\Bigg)\right]
\end{align*}
is well-defined, continuous, and monotonically increasing, where  $\nabla \psi(s) = (\psi_\xi(s),\psi_t(s))^\top$ denotes the gradient.

\medskip

\noindent\textbf{Approach.} We outline the overall strategy of our iteration argument. Our first step is to establish a nonlinear inequality for the template function $\varsigma(t)$. Specifically, we show that there exists a $t$- and $E_0$-independent constant $C_0 \geq 1$ such that $\varsigma(0) \leq C_0E_0$ and, for any $t \in [0,\tau_{\max})$ with $\varsigma(t) \leq \frac12$, it holds
\begin{align}\label{e:key}
\varsigma(t) \leq C_0\left(E_0 + \varsigma(t)^2\right).
\end{align}
We prove~\eqref{e:key} by combining iterative estimates on the $\v(t)$ and $\psi(t)$ with the norm equivalences from Lemma~\ref{l: equivalence} and the nonlinear damping estimate of Proposition~\ref{p: nonlinear damping}. In the second step, we show that, provided $E_0 < 1/(4C_0^2)$, the key inequality~\eqref{e:key} yields $\varsigma(t) \leq 2C_0E_0$ for all $t \in [0,\tau_{\max})$, which in turn implies $\tau_{\max} =  \infty$. The third and final step is to derive the decay estimates~\eqref{e:maindecaybounds}-\eqref{e:refineddecaybounds} from the bound $\varsigma(t) \leq 2C_0E_0$ for all $t \in [0,\infty)$.

\medskip

\noindent\textbf{Proof of key inequality.} We now proceed with establishing the key inequality~\eqref{e:key}. To this end, we denote by $C \geq 1$ any constant, which is independent of $t$ and $E_0$. 

First, we employ Lemma~\ref{l: equivalence} and the continuous embedding $H^1(\R) \hookrightarrow L^\infty(\R)$ to bound
\begin{align} \label{e:nonl1}
\|\v(t)\|_{H^3 \times H^2}, \|\v(t)\|_{W^{2,\infty} \times W^{1,\infty}} \leq C \frac{\varsigma(t)}{(1+t)^{\frac14}}
\end{align}
for $t \in [0,\tau_{\max})$ with $\varsigma(t) \leq \frac12$. Using that $\chi_-'$ is supported on $[-1,0]$ and applying~\eqref{e:nonl1}, we obtain through interpolation:
\begin{align} \label{e:nonl10}
\begin{split}
\big\|\chi_-(\cdot - \undertilde{c} t) \partial_\xi^j \v(t)\big\|_{L^2} &\leq \big\|\chi_-(\cdot - \undertilde{c} t) \partial_\xi^j \v(t) - \partial_\xi^j \left(\chi_-(\cdot - \undertilde{c} t) \v(t)\right)\big\|_{L^2} + \big\|\partial_\xi^j\left(\chi_-(\cdot - \undertilde{c} t) \v(t)\right)\big\|_{L^2}\\ 
&\leq C \left( \big\|\v(t)\big\|_{L^\infty} + \big\|\partial_\xi^{j+1} \v(t)\big\|_{L^2}^{\frac{j}{j+1}} \|\chi_-(\cdot - \undertilde{c} t) \v(t)\|_{L^2}^{\frac1{j+1}} \right) \leq C \frac{\varsigma(t)}{(1+t)^{\frac13}}
\end{split}
\end{align}
for $j = 1,2$ and $t \in [0,\tau_{\max})$ with $\varsigma(t) \leq \frac12$. Plugging~\eqref{e:nonl1} and~\eqref{e:nonl10} into Lemma~\ref{l: nonlinear estimates}, we arrive at the nonlinear bounds
\begin{align} \label{e:nonl2}
\begin{split}
\left\|\mathcal{N}(\v(t),\psi(t),\partial_t \psi(t))\right\|_{L^p} &\leq C \frac{\varsigma(t)^2}{(1+t)^{\frac34}},\\
\left\|\chi_-(\cdot - \undertilde{c} t + 1) \mathcal{N}(\v(t),\psi(t),\partial_t \psi(t))\right\|_{L^2} &\leq C \frac{\varsigma(t)^2}{(1+t)^{\frac{13}{12}}},\\
\left\|\omega_{\kappa,0}\mathcal{N}(\v(t),\psi(t),\partial_t \psi(t))\right\|_{L^\infty} &\leq C \re^{\kappa_0 \left(c_g + \frac14 \delta_c\right) t} \varsigma(t)^2
\end{split}
\end{align}
for $p = 2,\infty$ and $t \in [0,\tau_{\max})$ with $\varsigma(t) \leq \frac12$. 

Next, we apply the linear estimates in Theorem~\ref{t: linear estimates} and the nonlinear bounds~\eqref{e:nonl2} to the Duhamel formulation~\eqref{e:intv2} of the inverse-modulated perturbation, which yields
\begin{align}
\label{e:duh1}
\begin{split}
\|\v(t)\|_{L^p} &\leq C\left( \frac{E_0}{(1+t)^{\frac12 - \frac{1}{2p}}} + \int_0^t \frac{\re^{\kappa_0 \left(c_g + \frac14 \delta_c\right) s} \varsigma(t)^2}{(1+t-s)^{\frac12 - \frac{1}{2p}}}  \de s + \int_0^t \frac{\varsigma(t)^2}{(1+t-s)^{\frac34 - \frac{1}{2p}}(1+s)^{\frac34}} \de s\right.\\ 
&\qquad \left. + \, \int_0^t \frac{\re^{-\mu(t-s)}\varsigma(t)^2}{(1+s)^{\frac34}} \de s + \frac{\varsigma(t)^2}{(1+t)^{1 - \frac1{2p} }}\right) \leq C \frac{E_0 + \varsigma(t)^2}{(1+t)^{\frac12 - \frac1{2p}}}
\end{split}
\end{align}
and, using $\mu \geq -\kappa_0(c_g + \Delta \mu)$,
\begin{align}\label{e:duh12}
\begin{split}
\|\omega_{\kappa,0}\v(t)\|_{L^\infty} &\leq C\left(\re^{\kappa_0 \left(c_g + \Delta \mu \right) t} \left(E_0+\varsigma(t)^2\right) + \int_0^t \re^{\kappa_0 \left(c_g + \Delta \mu \right)(t-s)} \re^{\kappa_0 \left(c_g + \frac14 \delta_c\right) s} \varsigma(t)^2 \de s\right)\\ 
&\leq C \re^{\kappa_0 \left(c_g + \frac14 \delta_c\right) t} \left(E_0 + \varsigma(t)^2\right)
\end{split}
\end{align}
for $p = 2,\infty$ and $t \in [0,\tau_{\max})$ with $\varsigma(t) \leq \frac12$. On the other hand, setting $\mu_0 = \min\{\mu,\mu_l\} > 0$, applying Theorem~\ref{t: linear estimates}, Proposition~\ref{prop: left light cone}, and~\eqref{e:nonl2} to~\eqref{e:intv2}, and using that $\chi_-$ is monotonically decreasing, we obtain
\begin{align} \label{e:duh13}
\begin{split}
\|\chi_-(\cdot - \undertilde{c} t) \v(t)\|_{L^p} &\leq C\left( \frac{E_0 + \varsigma(t)^2}{(1+t)^{\frac34 - \frac{1}{2p}}} + \int_0^t \frac{\re^{-\mu_0(t-s)} \varsigma(t)^2}{ (1+s)^{\frac34}}  \de s + \int_0^t \frac{\varsigma(t)^2}{(1+t-s)^{\frac34 - \frac{1}{2p}}(1+s)^{\frac{13}{12}}} \de s\right)\\ 
&\leq C \frac{E_0 + \varsigma(t)^2}{(1+t)^{\frac34 - \frac1{2p}}}
\end{split}
\end{align}
for $p = 2,\infty$ and $t \in [0,\tau_{\max})$ with $\varsigma(t) \leq \frac12$. Similarly, we bound the modulation $\psi(t)$ by applying Theorem~\ref{t: linear estimates}, Proposition~\ref{prop: left light cone}, and estimate~\eqref{e:nonl2} to~\eqref{e:intpsi2}, which yields
\begin{align}
\label{e:duh20}
\begin{split}
\|\psi(t)\|_{L^\infty} &\leq C\left(E_0 + \int_0^t \re^{\kappa_0 \left(c_g + \frac14 \delta_c\right) s}\varsigma(t)^2 \de s + \int_0^t \frac{\varsigma(t)^2}{(1+t-s)^{\frac14}(1+s)^{\frac34}} \de s\right) \\
&\leq C\left( E_0 + \varsigma(t)^2\right),\\
\big\|\partial_\xi^k \partial_t^j \psi(t)\big\|_{L^p} &\leq C\left(\frac{E_0}{(1+t)^{\frac12 - \frac{1}{2p}}} + \int_0^t \frac{\re^{\kappa_0 \left(c_g + \frac14 \delta_c\right) s} \varsigma(t)^2}{(1+t-s)^{\frac12 - \frac{1}{2p}}}  \de s + \int_0^t \frac{\varsigma(t)^2}{(1+t-s)^{\frac34 - \frac{1}{2p}}(1+s)^{\frac34}} \de s\right)\\ 
&\leq C \frac{E_0 + \varsigma(t)^2}{(1+t)^{\frac12 - \frac1{2p}}},\\
\big\|\omega_{\kappa,0}\partial_\xi^i \partial_t^m \psi(t)\big\|_{L^\infty} &\leq C\left(\re^{\kappa_0 \left(c_g + \Delta \mu \right) t} E_0 + \int_0^t \re^{\kappa_0 \left(c_g + \Delta \mu \right)(t-s)} \re^{\kappa_0 \left(c_g + \frac14 \delta_c\right) s} \varsigma(t)^2 \de s\right)\\ 
&\leq C \re^{\kappa_0 \left(c_g + \frac14 \delta_c\right) t} \left(E_0 + \varsigma(t)^2\right),
\end{split}
\end{align}
and
\begin{align}\label{e:duh22}
\begin{split}
\|\chi_-(\cdot - \undertilde{c} t) \psi(t)\|_{L^\infty} &\leq C\left(\frac{E_0}{(1+t)^{\frac14}} + \int_0^t \frac{\re^{-\mu_l(t-s)}\varsigma(t)^2}{ (1+s)^{\frac34}} \de s + \int_0^t \frac{\varsigma(t)^2}{(1+t-s)^{\frac14}(1+s)^{\frac{13}{12}}} \de s\right)\\ 
&\leq C \frac{E_0 + \varsigma(t)^2}{(1+t)^{\frac14}},\\
\|\chi_-(\cdot - \undertilde{c} t) \partial_\xi^k \partial_t^j \psi(t)\|_{L^p} &\leq C\left(\!\frac{E_0}{(1+t)^{\frac34 - \frac{1}{2p}}} + \int_0^t \frac{\re^{-\mu_l(t-s)}\varsigma(t)^2}{(1+s)^{\frac34}} \de s + \int_0^t \frac{\varsigma(t)^2}{(1+t-s)^{\frac34 - \frac{1}{2p}}(1+s)^{\frac{13}{12}}} \de s\!\right)\\ 
&\leq C \frac{E_0 + \varsigma(t)^2}{(1+t)^{\frac34 - \frac1{2p}}}
\end{split}
\end{align}
for $p = 2,\infty$, $t \in [0,\tau_{\max})$ and $i,j, k, m \in \N_0$ with $1 \leq j + k \leq \ell$, $i, m\leq \ell$, and $\varsigma(t) \leq \frac12$. 

Combining the estimates~\eqref{e:duh1} and~\eqref{e:duh20} with Lemma~\ref{l: equivalence}, we infer
\begin{align} \label{e:duh3}
\|\vf(t)\|_{L^p} \leq C \frac{E_0 + \varsigma(t)^2}{(1+t)^{\frac12 - \frac{1}{2p}}}
\end{align}
for $p = 2,\infty$ and $t \in [0,\tau_{\max})$ with $\varsigma(t) \leq \frac12$. Next, we plug the bounds~\eqref{e:duh20} and~\eqref{e:duh3} into the nonlinear damping estimate from Proposition~\ref{p: nonlinear damping}, which yields
\begin{align} \label{e:duh4}
\|\vf(t)\|_{H^3 \times H^2}^2 \leq C\left(\re^{-\vartheta t} E_0^2 + \frac{\left(E_0 + \varsigma(t)^2\right)^2}{\sqrt{1+t}} + \int_0^t \frac{\left(E_0 + \varsigma(t)^2\right)^2}{\re^{\vartheta (t-s)}\sqrt{1+s}} \de s\right) \leq C \frac{\left(E_0 + \varsigma(t)^2\right)^2}{\sqrt{1+t}}
\end{align}
for $t \in [0,\tau_{\max})$ with $\varsigma(t) \leq \frac12$. 

Finally, we recall that $\psi(0) = 0$ and $\v(0) = \vf(0) = \v_0$ by Proposition~\ref{p:psi}, which implies
\begin{align} \label{e:innit}
\varsigma(0) \leq C E_0.
\end{align}

All in all, combining the estimates~\eqref{e:duh1},~\eqref{e:duh12},~\eqref{e:duh13},~\eqref{e:duh20},~\eqref{e:duh22},~\eqref{e:duh4}, and~\eqref{e:innit}, we conclude that there exists a $t$- and $E_0$-independent constant $C_0 \geq 1$ such that $\varsigma(0) \leq C_0E_0$ and the key inequality~\eqref{e:key} holds for all $t \in [0,\tau_{\max})$ with $\varsigma(t) \leq \frac12$. 

\medskip

\noindent\textbf{Closing the nonlinear iteration.} Take $\delta_0 = 1/(4C_0^2)$ and $E_0 \in (0,\delta_0)$. Arguing by contradiction, we assume that there exists $t \in [0,\tau_{\max})$ such that $\varsigma(t) > 2C_0E_0$. Then, since we have $\varsigma(0) \leq C_0E_0$, it follows by continuity of $\varsigma$ that there exists $t_0 \in (0,t)$ such that $\varsigma(t_0) = 2C_0E_0 < \frac12$. Applying~\eqref{e:key}, we bound
$$\varsigma(t_0) \leq C_0\left(E_0 + 4C_0^2E_0^2\right) < 2C_0E_0,$$
which contradicts $\varsigma(t_0) = 2C_0E_0$. We conclude that $\varsigma(t) \leq 2C_0E_0$ for all $t \in [0,\tau_{\max})$. In particular, since we have $\varsigma(t) \leq 2C_0E_0 < \frac12$ for all $t \in [0,\tau_{\max})$,~\eqref{e:blowup3} cannot hold, which implies $\tau_{\max} = T_{\max}$ by Proposition~\ref{p:psi}. 

To argue that $T_{\max} = \infty$, we use~\eqref{e:fmodpert} and apply Lemma~\ref{l:H2estimate}, which yields
\begin{align*}
\begin{split}
\|\vt(t)\|_{H^1} &\lesssim\|\vf(t)\|_{H^1} + \|\psi(t)\|_{L^\infty} + \|\psi_\xi(t)\|_{L^2} + \big\|(\chi_- - \chi_-(\cdot - \undertilde{c} t)) \psi_\xi(t)\big\|_{L^2} + \big\|\chi_-(\cdot - \undertilde{c} t) \psi_\xi(t)\big\|_{L^2} \\
&\lesssim \|\vf(t)\|_{H^1} + \sqrt{1+t} \, \|\psi(t)\|_{L^\infty} + \|\psi_\xi(t)\|_{L^2} + \big\|\chi_-(\cdot - \undertilde{c} t) \psi_\xi(t)\big\|_{L^2}\lesssim \sqrt{1+t}\, \varsigma(t)
\end{split}
\end{align*}
for $t \in [0,T_{\max})$. Since $\varsigma(t) \leq 2C_0E_0$ for all $t \in [0,T_{\max})$, this estimate precludes~\eqref{e:blowup}. Hence, we have $\tau_{\max} = T_{\max} = \infty$ by Proposition~\ref{p:local_unmod}. Consequently, it holds
\begin{align} \varsigma(t) \leq 2CE_0 \label{e:conclit}\end{align}
for all $t \geq 0$. Combining the latter with estimates~\eqref{e:duh20} and~\eqref{e:duh3} readily yields the bounds in~\eqref{e:maindecaybounds3}. Moreover, using~\eqref{e:fmodpert}, Lemma~\ref{l:MVT in L2}, and the fact that $\omega_0 \Ups' = \omega_{0,\eta_0} \Ups'$ is bounded by Lemma~\ref{l:H2estimate}, we obtain
\begin{align*}
\begin{split}
\|\vt(t)\|_{L^\infty} &\leq \|\vf(t)\|_{L^\infty} + \re^{\eta_0 \|\psi(t)\|_{L^\infty}}\|\omega_0 \Ups'\|_{L^\infty}\|\psi(t)\|_{L^\infty}
\end{split}
\end{align*}
for all $t \geq 0$, which together with~\eqref{e:conclit}  implies~\eqref{e:maindecaybounds}.

\medskip

\noindent\textbf{A posteriori estimates.} In the following we establish the remaining estimates~\eqref{e:maindecaybounds21},~\eqref{e:maindecaybounds22} and~\eqref{e:refineddecaybounds} with asymptotic phase
\begin{align} \label{e:defasymp}
\psi_\infty = \Ptr \v_0 + \int_0^\infty \Ptr \mathcal{N}(\v(s),\psi(s),\partial_s \psi(s))\de s.
\end{align}
Again we denote by $C \geq 1$ any constant, which is independent of $t$ and $E_0$. 

Applying Theorem~\ref{t: linear estimates} and using~\eqref{e:nonl2} and~\eqref{e:conclit}, we note that $\psi_\infty$ is well-defined and satisfies
\begin{align} \label{e:post1}
|\psi_\infty - \Ptr \v_0| &\leq C \int_0^\infty \re^{\kappa_0 \left(c_g + \frac14 \delta_c\right) s} \varsigma(t)^2 \de s \leq C E_0^2, \qquad |\psi_\infty| \leq C E_0,
\end{align}
yielding~\eqref{e:maindecaybounds21}. Moreover, subtracting~\eqref{e:defasymp} from~\eqref{e:intpsi}, we arrive at
\begin{align} \label{e:intpsi3}
\begin{split}
\psi(t) - \psi_\infty &= [s_p(t) - \Ptr]\v_0 + \int_0^t [s_p(t-s) - \Ptr] \mathcal{N}(\v(s),\psi(s),\partial_s \psi(s))\de s\\ &\qquad + \, \int_t^\infty \Ptr \mathcal{N}(\v(s),\psi(s),\partial_s \psi(s))\de s
\end{split}
\end{align}
for $t \geq 0$. Applying the right light cone estimates from Proposition~\ref{prop: right light cone} to~\eqref{e:intpsi2} and~\eqref{e:intpsi3}, using the bounds~\eqref{e:nonl2} and~\eqref{e:conclit}, and noting that $\Delta c - \Delta \theta = \frac14 \delta_c$, $\tilde{c} = c_g + \Delta c$, and $c_g + \delta_c < 0$, we obtain
\begin{align} \label{e:appsi}
\begin{split}
&\|\chi_+(\cdot - \tilde{c} t)(\psi(t) - \psi_\infty)\|_{L^\infty}\\ 
&\qquad \leq C E_0 \left( \re^{\kappa_0 (\Delta \theta - \Delta c) t}  + \int_0^t \frac{\re^{\kappa_0 (c_g + \frac14 \delta_c) s}}{\re^{\kappa_0 (\Delta c - \Delta \theta) (t-s)} \re^{\kappa_0 \tilde{c} s} } \de s + \int_t^\infty \re^{\kappa_0 (c_g + \frac14 \delta_c) s} \de s\right)\\
&\qquad \leq C E_0 \left(\re^{-\frac14 \delta_c \kappa_0 t}  + \int_0^t \re^{-\frac14 \delta_c \kappa_0 t} \de s + \int_t^\infty \re^{-\frac34 \delta_c s} \de s\right)
\leq CE_0 \re^{-\frac18 \delta_c \kappa_0 t}
\end{split}
\end{align}
and
\begin{align} \label{e:appsi2}
\begin{split}
\|\chi_+(\cdot - \tilde{c} t) \psi_\xi(t)\|_{L^\infty} &\leq C E_0 \left(\re^{\kappa_0 (\Delta \theta - \Delta c) t}  + \int_0^t \frac{\re^{\kappa_0 (c_g + \frac14 \delta_c) s}}{\re^{\kappa_0 (\Delta c - \Delta \theta) (t-s)} \re^{\kappa_0 \tilde{c} s} } \de s\right) \leq CE_0 \re^{-\frac18 \delta_c \kappa_0 t}
\end{split}
\end{align}
for $t \geq 0$. Similarly, we bound $\chi_+(\cdot - \tilde{c} t) \v(t)$, now applying Proposition~\ref{prop: right light cone} and the estimates~\eqref{e:nonl2},~\eqref{e:conclit}, and~\eqref{e:appsi2} to the Duhamel formulation~\eqref{e:intv2} and using that $\kappa_0(\Delta c - \Delta \theta) \leq \mu_r$, which yields
\begin{align} \label{e:apv}
\begin{split}
\|\chi_+(\cdot - \tilde{c} t) \v(t)\|_{L^\infty} &\leq C E_0 \left(\re^{\kappa_0 (\Delta \theta - \Delta c) t}  + \int_0^t \frac{\re^{\kappa_0 (c_g + \frac14 \delta_c) s}}{\re^{\kappa_0 (\Delta c - \Delta \theta) (t-s)} \re^{\kappa_0 \tilde{c} s} } \de s + \re^{-\frac18 \delta_c \kappa_0 t}\right) \leq CE_0 \re^{-\frac18 \delta_c \kappa_0 t}
\end{split}
\end{align}
for all $t \geq 0$. 

By~\eqref{e:conclit} we have $\|\psi(t)\|_{L^\infty} \leq \frac12 \delta_c t$ for $t \geq 4 C_0 E_0 \delta_c^{-1}$. Therefore, combining the estimates~\eqref{e:conclit},~\eqref{e:appsi2}, and~\eqref{e:apv} with Lemma~\ref{l: equivalence} and using that $\chi_- = 1-\chi_+$ is monotonically decreasing, we infer
\begin{align} \label{e:large_time_light}
\begin{split}
\big\|\chi_-(\cdot - (c_g - \delta_c) t) \vf(t)\big\|_{L^\infty} &\leq \big\|\chi_-(\cdot - \undertilde{c} t + \|\psi(t)\|_{L^\infty}\big) \vf(t)\|_{L^\infty}\\ 
&\leq C \left(\big\|\chi_-(\cdot - \undertilde{c}t) \v(t)\big\|_{L^\infty} + \big\|\chi_-(\cdot - \undertilde{c}t) \psi_\xi(t)\big\|_{L^\infty}\right) \leq CE_0 (1+t)^{-\frac34},\\
\big\|\chi_+(\cdot - (c_g + \delta_c) t) \vf(t)\big\|_{L^\infty} &\leq \big\|\chi_+(\cdot - \tilde{c} t - \|\psi(t)\|_{L^\infty}\big) \vf(t)\|_{L^\infty}\\ 
&\leq C \left(\big\|\chi_+(\cdot - \tilde{c}t) \v(t)\big\|_{L^\infty} + \big\|\chi_+(\cdot - \tilde{c}t) \psi_\xi(t)\big\|_{L^\infty}\right) \leq CE_0 \re^{-\frac18 \delta_c \kappa_0 t}
\end{split}
\end{align}
for $t \geq 4C_0 E_0\delta_c^{-1}$. On the other hand,~\eqref{e:duh3}  yields the short-time bound
\begin{align} \label{e:short_time_light}
\big\|\chi_\pm(\cdot - (c_g \pm \delta_c) t) \vf(t)\big\|_{L^\infty} \leq \|\vf(t)\|_{L^\infty} \leq CE_0  \leq CE_0 \re^{-\frac18 \delta_c \kappa_0 t}
\end{align}
for $0 \leq t \leq 4C_0 E_0\delta_c^{-1}$. Next, we recall that $\omega_0 \Ups' = \omega_{0,\eta_0} \Ups'$ is bounded by Lemma~\ref{l:H2estimate} . Hence, applying Lemma~\ref{l:MVT in L2}, using estimates~\eqref{e:duh22},~\eqref{e:conclit},~\eqref{e:post1}, and~\eqref{e:appsi}, and noting that $\chi_- = 1-\chi_+$ is monotonically decreasing, we arrive at
\begin{align*}
&\big\|\chi_-(\cdot - (c_g - \delta_c) t) \omega_0\left(\Ups(\cdot + \psi(\cdot,t)) - \Ups\right)\big\|_{L^\infty}\\ 
&\qquad \leq C\re^{\eta_0 \|\psi(t)\|_{L^\infty}} \big\|\omega_0 \Ups'\big\|_{L^\infty}
\big\|\chi_-(\cdot - \undertilde{c} t) \psi(t)\big\|_{L^\infty} \leq \frac{CE_0}{(1+t)^{\frac14}},
\end{align*}
and
\begin{align*}
&\big\|\chi_+(\cdot - (c_g + \delta_c) t) \omega_0\left(\Ups(\cdot + \psi(\cdot,t)) - \Ups(\cdot + \psi_\infty)\right)\big\|_{L^\infty}\\ 
&\qquad \leq C \re^{\eta_0 (\|\psi(t)\|_{L^\infty} + |\psi_\infty|)}
 \big\|\omega_0\Ups'\big\|_{L^\infty} \left\|\chi_+(\cdot - \tilde{c} t)(\psi(t) - \psi_\infty)\right\|_{L^\infty} \leq C E_0 \re^{-\frac18 \delta_c \kappa_0 t}
\end{align*}
for $t \geq 0$. Combining the last two estimates with~\eqref{e:duh22},~\eqref{e:appsi},~\eqref{e:large_time_light}, and~\eqref{e:short_time_light} and using~\eqref{e:fmodpert}, we establish~\eqref{e:refineddecaybounds}. Finally,~\eqref{e:maindecaybounds22} is a direct consequence of~\eqref{e:refineddecaybounds}. 
\end{proof}

\begin{remark} \label{rem: motivation L2 Linfty scheme}
{
\upshape
The estimates in~\eqref{e:duh1} and~\eqref{e:duh20} on the nonlinear terms
\begin{align*}
\int_0^t s_{c,2}(t-s) \mathcal{N}(\v(s),\psi(s),\partial_t \psi(s)) \de s, \qquad
\int_0^t \partial_\xi^\ell s_{p,2}(t-s) \mathcal{N}(\v(s),\psi(s),\partial_t \psi(s)) \de s,
\qquad \ell = 0,1,
\end{align*}
in the integral equations~\eqref{e:intv2} and~\eqref{e:intpsi2} for $\v(t)$ and $\partial_\xi^\ell \psi(t)$, respectively, are critical. Indeed, any weaker decay of $\|\mathcal{N}(\v(s),\psi(s),\partial_t \psi(s))\|_{L^2}$ would prevent us from showing the estimates~\eqref{e:duh1} and~\eqref{e:duh20}. The reason that this marginal decay is nevertheless sufficient is that we can distribute spatial localization between the nonlinearity and the linear scattering terms $s_{c,2}(t-s)$ and $\smash{\partial_\xi^\ell s_{p,2}(t-s)}$. Since $\v(s)$ and derivatives of $\psi(s)$ decay at rate $\smash{s^{-1/2 + 1/(2p)}}$ in $L^p(\R)$ for $p = 2,\infty$, the nonlinearity $\mathcal{N}(\v(s),\psi(s),\partial_t \psi(s))$ may be estimated in any $L^q$-norm with $1 \leq q \leq \infty$, leading to the decay rate $\smash{(1+s)^{-1+1/2q}}$. Correspondingly, we may employ $(L^q \to L^p)$-bounds for the operators $s_{c,2}(t-s)$ and $\partial_\xi^\ell s_{p,2}(t-s)$ for any $1 \leq q \leq p$, which provide the decay rate $\smash{(1+t-s)^{-1/2+1/(2p)-1/(2q)}}$. Choosing $q=2$ avoids picking up additional logarithmic losses when estimating the resulting time convolution, since then both $\smash{\frac12-\frac1{2p}+\frac1{2q} \neq 1}$ and $\smash{1-\frac1{2q} \neq 1}$ for $p = 2,\infty$. Consequently, the nonlinear argument in the proof of Theorem~\ref{t: main} closes despite the marginal decay rates.
}
\end{remark}

\section{Discussion}  \label{sec:discussion}

We discuss how our analysis here relates to remaining open problems in front invasion, and diffusive stability of coherent structures more broadly. 

\medskip

\noindent\textbf{Modulated pushed fronts.} In many systems, pattern-forming fronts are modulated, that is, they are time-periodic (rather than stationary) in a comoving frame, leading to a non-autonomous linearization with time-periodic coefficients, $u_t = \mathcal{L}(t) u$. Such fronts arise, for instance, in reaction-diffusion systems near a subcritical Turing instability. The eigenvalues of the associated period map give a natural notion of spectral stability. In~\cite{SandstedeScheelModulatedTW}, it was shown that spectral stability can be equivalently formulated in terms of eigenvalues of the operator $-\partial_t + \mathcal{L}(t)$ acting on a space of time-periodic functions. This approach was extended in~\cite{BeckSandstedeZumbrun, BNSZ2}, where the temporal Green's function associated with the linearized evolution was expressed as an inverse Laplace-type integral of the resolvent of $-\partial_t + \mathcal{L}(t)$. Pointwise estimates on the Green's function were then extracted from information on the associated resolvent integral kernel, which were in turn established through a spatial dynamics approach relying on pasting of exponential dichotomies. The equivalence of exponential dichotomies and Fredholm properties for linearizations about modulated traveling waves was obtained in~\cite{SandstedeScheelModulatedTW}. We therefore expect that our far-field/core approach to studying the resolvent problem and thereby obtaining linear estimates can be lifted using these techniques to problems involving modulated fronts. 

Natural spectral stability conditions for modulated pushed fronts were formulated in~\cite[Definition 6.6]{AHSReview}. In contrast to the rigid fronts considered here, the neutral eigenspace is now two dimensional, due to invariance under both space and time translations. As a result, perturbations to the front induce an asymptotic shift in both the spatial and temporal phase. Nonetheless, we expect that the essential difficulties are already captured in our analysis here, and believe that this approach will lead to a proof of the selection of generic modulated pushed fronts in general reaction-diffusion systems. 

\medskip

\noindent\textbf{Selection of pulled pattern-forming fronts.} For pushed pattern-forming fronts, the asymptotic phase shift after perturbation of the front is determined by the response of the neutral translational mode. Our analysis focuses on characterizing the interaction between this translational mode and the outgoing diffusive mode in the wake of the front. For pulled fronts which select constant, exponentially stable states in their wake, the position of the front interface is determined by a subtle matching of the front profile with a Gaussian tail determined by the linearized dynamics ahead of the front. Selection results for pulled fronts in general reaction-diffusion systems~\cite{AveryScheelSelection, AverySelectionRD} only capture the interface position as $x_*(t) = \smash{ \clin t - \frac{3}{2 \etalin} \log t + \mathrm{O}(1)}$, where $\clin$ is the linear spreading speed and $\etalin$ corresponds to the pointwise exponential decay rate in the leading edge of the front. Based on the  analysis in this paper, it seems that resolving the $\mathrm{O}(1)$ term in the expansion of the interface position more precisely will be a necessary step for establishing selection of pulled pattern-forming fronts, yet this remains challenging due to the subtle gluing process which determines the front position. 

\medskip

\noindent\textbf{Stability of source defects.} In the mathematical analysis of pattern-forming systems, \emph{defects} are coherent structures which mediate the interaction between distinct periodic patterns, possibly with different wave numbers, phases, or group velocities. \emph{Source defects} consist of a localized core which emits periodic wave trains with a selected wave number~\cite{SandstedeScheelDefects} and a group velocity which points away from the core of the defect. In many respects, pushed pattern-forming fronts may be viewed as ``one-sided'' source defects: the front interface emits periodic wave trains with a group velocity pointing away from the front interface. 
    
    The nonlinear stability of source defects in the complex Ginzburg--Landau equation was established in~\cite{BNSZ1}. This analysis, however, relies heavily on the gauge symmetry of the complex Ginzburg--Landau equation, which is not available in general pattern-forming systems. Our present analysis of pushed pattern-forming fronts is therefore more closely related to the unpublished manuscript~\cite{BNSZ2} which analyzes the nonlinear stability of source defects in general reaction-diffusion systems under spectral assumptions. These defects are time periodic in the co-moving frame, akin to modulated pattern-forming fronts. This time periodicity gives rise to an additional embedded eigenvalue at the origin, associated to translation invariance in time. Perturbations of a general source defect then converge to a spatial and temporal shift of the original defect, locally uniformly in space. As in the case of modulated fronts, we are confident that the effects of time periodicity could be incorporated into our analysis here, and our methods can thereby give rise to an alternative proof of the nonlinear stability of source defects in reaction-diffusion systems. 
    
    The analysis in~\cite{BNSZ2} relies on subtle pointwise Gaussian bounds on the temporal Green's function associated with the linearized equation. These bounds are then passed through an iteration argument, leading to the nonlinear stability of the source defect with a detailed characterization of the pointwise Gaussian decay behavior of perturbations. Our proof is simpler by comparison, relying only on fixed spatially weighted norms. Additionally, our analysis, translated into the framework of~\cite{BNSZ2}, would only require the initial perturbation to lie in $L^2(\R)$, while~\cite{BNSZ2} assumes Gaussian localization of initial perturbations. 

\medskip

\noindent\textbf{Nonlocalized perturbations and modulational data.}
In the nonlinear stability analysis of wave trains, substantial effort has been devoted to relaxing localization requirements on perturbations, ultimately leading to a theory for fully nonlocalized perturbations; see~\cite{Alexopoulos2026,deRijknonlocalized} and references therein. In the present setting of pushed pattern-forming fronts, $L^2$-localization on the left is used to extract diffusive decay from the scattering terms associated with the wave train in the wake of the front; see Theorem~\ref{t: linear estimates}(iii). By contrast, interaction with the neutral translational mode, see Theorem~\ref{t: linear estimates}(i), yields an additional spatial localization factor originating from the exponential localization of the adjoint eigenfunction $\psiad$. These observations suggest that, by using the approach of~\cite{Alexopoulos2026,deRijknonlocalized} to control diffusive scattering terms in the nonlinear stability argument, perturbations of pushed pattern-forming fronts that are nonlocalized on the left can be handled. Similarly, adapting the modulational stability theory for wave trains developed in~\cite{IyerSandstede,Alexopoulos2025Modulation, JNRZ_13_1, JNRZWhitham, JNRZInventiones} to the present setting, we expect that it is possible to control \emph{modulational data} of the form
\begin{align} \label{e:moddata}
\u(\xi,0) = \Ups(\xi + \psi_0(\xi)) + \w_0,
\end{align}
where $\psi_0 \colon \R \to \R$ is a (possibly large) initial phase modulation and $\|\psi_0'\|_{L^\infty}$ and $\|\omega_0 \w_0\|_{L^\infty}$ are small. One then aims to show that solutions with initial condition~\eqref{e:moddata} remain close, for all $t \geq 0$, to a modulated front of the form $\Ups(\xi + \psi(\xi,t))$, thereby establishing a global modulational stability result. These results would emphasize the strength of the selection mechanism of the invasion process: one still expects exponential in time, locally uniform in space convergence to the new phase $\psi_\infty$ selected by the front interface, even with a large initial phase modulation to the left of the interface. Finally, in view of the structural similarities, we anticipate that these ideas can also be adapted to source defects. Such an extension would yield a nonlinear stability theory for source defects that accommodates large phase modulations as well as perturbations that are fully nonlocalized on both the left and the right. 

\appendix

\section{Various estimates involving the exponential weight \texorpdfstring{$\omega_0$}{omega0}} \label{app: exp weight}

The exponential weight $\omega_0$, which is used to stabilize the rest state in the leading edge of the pushed pattern-forming front $\Ups$, is unbounded on $[0,\infty)$. Nevertheless, by estimate~\eqref{e: front right asymptotics} and Hypotheses~\ref{hyp: leading edge} and~\ref{hyp: point spectrum} the products $\omega_0 \Ups$ and $\omega_0 \Ups'$ are $L^2$-localized on $[0,\infty)$. A standard bootstrapping argument then yields that $\omega_0 \Ups$ is smooth and its derivatives are also $L^2$-localized on $[0,\infty)$. In this appendix, we obtain various estimates needed to bound expressions involving $\omega_0$ in our nonlinear stability analysis.

The first result confirms the exponential character of the weight $\omega_{0,\eta}$ and its proof follows directly from the definition of $\omega_{0,\eta}$.

\begin{lemma} \label{l:omegaid}
Fix $\eta > 0$. It holds
\begin{align*}
\frac{\omega_{0,\eta}(\xi)}{\omega_{0,\eta}(\xi+y)} \lesssim \re^{\eta |y|}
\end{align*}
for all $\xi,y \in \R$.
\end{lemma}

Next, we prove a weighted and unweighted $L^2$-version of the mean value inequality, as well as a weighted pointwise mean value inequality.

\begin{lemma} \label{l:MVT in L2}
Fix $\eta > 0$. We have
\begin{align}
\left\|\omega_{0,\eta}^j \left(v(\cdot - \psi_2(\cdot)) - v(\cdot - \psi_1(\cdot))\right)\right\|_{L^2} \lesssim \|\omega_{0,\eta}^j v'\|_{L^2} \|\psi_2-\psi_1\|_{L^\infty} \re^{\eta j \left(\|\psi_1\|_{L^\infty} + \|\psi_2\|_{L^\infty}\right)}
\end{align}
for $j = 0,1$, $\psi_{1,2} \in L^\infty(\R)$, and $v \in H^1(\R)$ with $\omega_{0,\eta}^j v' \in L^2(\R)$. Moreover, it holds
\begin{align*}
\left|\omega_{0,\eta}(\xi)(v(\xi + y_2) - v(\xi+y_1))\right| \lesssim \re^{\eta\left(|y_1| + |y_2|\right)} \big\|\omega_{0,\eta} v'\big\|_{L^\infty} |y_2 - y_1|
\end{align*}
for $\xi, y_1,y_2 \in \R$ and $v \in C^1(\R)$ with $\omega_{0,\eta} v' \in L^\infty(\R)$.
\end{lemma}
\begin{proof}
Using the fundamental theorem of calculus, H\"older's inequality, and Lemma~\ref{l:omegaid} we obtain
\begin{align*}
\begin{split}
&\int_\R \omega_{0,\eta}(\xi)^{2j} \left|v(\xi - \psi_2(\xi)) - v(\xi - \psi_1(\xi))\right|^2 \de \xi
= \int_\R \omega_{0,\eta}(\xi)^{2j} \left|\int_0^{\psi_2(\xi)-\psi_1(\xi)} v'(\xi - \psi_1(\xi) - y) \de y\right|^2 \de \xi\\
&\qquad\leq \re^{2 \eta j \left(\|\psi_1\|_{L^\infty} + \|\psi_2\|_{L^\infty}\right)} \int_{-\|\psi_2-\psi_1\|_{L^\infty}}^{\|\psi_2-\psi_1\|_{L^\infty}}\int_{-\|\psi_2-\psi_1\|_{L^\infty}}^{\|\psi_2-\psi_1\|_{L^\infty}} \int_\R |\omega_{0,\eta}(\xi-\psi_1(\xi) - y)^jv'(\xi - \psi_1(\xi) - y) | \\
&\qquad \qquad \cdot \, |\omega_{0,\eta}(\xi-\psi_1(\xi) - z)^jv'(\xi - \psi_1(\xi) - z)| \de \xi \de y \de z  \\
&\qquad \lesssim \|\omega_{0,\eta}^j v'\|_{L^2}^2 \|\psi_2-\psi_1\|_{L^\infty}^2 \re^{2 \eta j \left(\|\psi_1\|_{L^\infty} + \|\psi_2\|_{L^\infty}\right)}
\end{split}
\end{align*}
for $j = 0,1$, $\psi_{1,2} \in L^\infty(\R)$, and $v \in C_c^\infty(\R)$ with $\omega_{0,\eta}^j v' \in L^2(\R)$. The first estimate now follows by density of test functions in $L^2(\R)$. In addition, using Lemma~\ref{l:omegaid} and the fundamental theorem of calculus, we arrive at
\begin{align*} 
\left|\omega_{0,\eta}(\xi)(v(\xi + y_2) - v(\xi+y_1))\right| = \left|\omega_{0,\eta}(\xi) \int_{y_1}^{y_2} v'(\xi + y) \de y \right| \lesssim \re^{\eta\left(|y_1| + |y_2|\right)} \big\|\omega_{0,\eta} v'\big\|_{L^\infty} |y_2 - y_1|
\end{align*}
for $\xi, y_1,y_2 \in \R$ and $v \in C^1(\R)$ with $\omega_{0,\eta} v' \in L^\infty(\R)$, which establishes the second estimate. 
\end{proof}

We apply the weighted mean value inequalities, established in Lemma~\ref{l:MVT in L2}, to estimate the $H^k$-difference between two modulations of the pulled pattern-forming front $\Ups$. Here, we exploit that $\omega_0 \Ups$ and its derivatives are $L^2$-localized on $[0,\infty)$ and bounded on $(-\infty,0]$. Thus, bounded modulations which are $L^2$-localized on $(-\infty,0]$ can be accommodated.

\begin{lemma} \label{l:H2estimate}
Let $k \in \mathbb N$ and $R > 0$. Then, $\omega_0 \partial_\xi^k \Ups$ is bounded and we have 
\begin{align*}
\left\|\omega_0 \left(\Ups(\cdot-\psi_2(\cdot)) - \Ups(\cdot - \psi_1(\cdot))\right)\right\|_{H^k} \lesssim \|\psi_2 - \psi_1\|_{L^\infty} + \|\psi_2' - \psi_1'\|_{H^{k-1}} + \|\chi_- \left(\psi_2 - \psi_1\right)\|_{L^2}
\end{align*}
for $\psi_{1,2} \in L^\infty(\R)$ with $\chi_-\psi_{1,2} \in L^2(\R)$, $\psi_{1,2}' \in H^{k-1}(\R)$, and $\|\psi_{1,2}\|_{W^{k-1,\infty}} \leq R$.
\end{lemma}
\begin{proof}
We write 
\begin{align} \label{e:Vineq2}
\begin{split}
&\big(\partial_\xi^j \Ups\big)(\xi-\psi_1(\xi)) - \big(\partial_\xi^j \Ups\big)(\xi-\psi_2(\xi)) = \big(\chi_+ \partial_\xi^j \Ups\big)(\xi-\psi_1(\xi)) - \big(\chi_+ \partial_\xi^j \Ups\big)(\xi-\psi_2(\xi)) \\
&\qquad + \, \chi_-(\xi)\left(\big(\chi_- \partial_\xi^j \Ups\big)(\xi-\psi_1(\xi)) - \big(\chi_- \partial_\xi^j \Ups\big)(\xi-\psi_2(\xi))\right)\\
&\qquad + \, \left(\chi_-(\xi - \|\psi_1\|_{L^\infty} - \|\psi_2\|_{L^\infty}-1) - \chi_-(\xi)\right) \left(\big(\chi_- \partial_\xi^j \Ups\big)(\xi-\psi_1(\xi)) - \big(\chi_- \partial_\xi^j \Ups\big)(\xi-\psi_2(\xi))\right)
\end{split}
\end{align}
for $j \in \mathbb N_0$, $\xi \in \R$, and $\psi_{1,2} \in L^\infty(\R)$. Since the weight $\eta_0$ lies in $(0,\etaps)$ by Hypothesis~\ref{hyp: leading edge}, estimate~\eqref{e: front right asymptotics} and Hypothesis~\ref{hyp: point spectrum} imply that $\omega_0 \Ups$ and $\smash{\omega_0 \Ups'}$ are $L^2$-integrable on $[0,\infty)$. On the other hand, $\omega_0$, $\Ups$, and $\smash{\Ups'}$ are bounded on $(-\infty,0]$. Using that $\Ups$ is a stationary solution to~\eqref{e: FHN comoving} to express higher-order derivatives of $\Ups$ in terms of $\Ups$ and $\Ups'$, we conclude that $\smash{\omega_0 \partial_\xi^{1+j} \Ups}$ is bounded and $\smash{\omega_0 (\chi_+ \partial_\xi^j \Ups)'}$ lies in $L^2(\R)$ for $j = 0,\ldots,k$. Hence, using the mean value theorem, Lemmas~\ref{l:omegaid} and~\ref{l:MVT in L2}, the continuous embedding $H^1(\R) \hookrightarrow L^\infty(\R)$, the identity~\eqref{e:Vineq2}, and the facts that $\smash{\omega_0^{-1} \partial_\xi^j \omega_0 }$ is bounded for $j = 0,\ldots,k$ and that $\smash{\chi_-(\cdot - \|\psi_1\|_{L^\infty} - \|\psi_2\|_{L^\infty}-1)} - \chi_-$ has compact support, we obtain
\begin{align*} 
\begin{split}
&\left\|\omega_0 \left(\Ups(\cdot-\psi_1(\cdot)) - \Ups(\cdot - \psi_2(\cdot))\right)\right\|_{H^k} \\
&\qquad \lesssim \sum_{j = 0}^k \left\|\omega_0 \big(\chi_+ \partial_\xi^j \Ups\big)'\right\|_{L^2} \|\psi_2-\psi_1\|_{L^\infty} + \sum_{j = 1}^k \left\|\omega_0 \partial_\xi^j\Ups\right\|_{L^\infty} \|\psi_2' - \psi_1'\|_{H^{k-1}} \\
&\qquad \qquad + \, \left\|\Ups'\right\|_{W^{k,\infty}} \left(\|\chi_-(\psi_2 - \psi_1)\|_{L^2} + \|\psi_2 - \psi_1\|_{L^\infty}\right)
\end{split}
\end{align*}
for $\psi_{1,2} \in L^{\infty}(\R)$ with $\chi_-\psi_{1,2} \in L^2(\R)$, $\psi_{1,2}' \in H^{k-1}(\R)$, and $\|\psi_{1,2}\|_{W^{k-1,\infty}} \leq R$, which establishes the result.
\end{proof}

	\bibliographystyle{abbrv}
	\bibliography{pushedbib}
	
\end{document}